\documentclass[a4paper,11pt]{article}
\usepackage[top=3.0cm, bottom=3.0cm, inner=3.0cm, outer=3.0cm, includefoot]{geometry}

\usepackage[utf8]{inputenc}
\usepackage[T1]{fontenc}
\usepackage{authblk}

\usepackage{verbatim}
\usepackage{multirow}
\usepackage{multicol}

\usepackage{longtable}
\usepackage{hyperref}

\usepackage{amssymb}
\usepackage{amsmath}
\usepackage{mathtools,bbm}
\usepackage{graphicx,tikz}
\usepackage{amsthm}
\usepackage{caption,subcaption}

\usepackage{color}
\usepackage{enumerate}

\usepackage{cancel}

\newcommand{\K}{{\mathcal{K}}}

\renewcommand*{\v}{\mathbf{v}}

\newcommand{\eChar}{\begin{enumerate}[(i)]}
\newcommand{\eCharR}{\begin{enumerate}[(a)]}
\newcommand{\eBr}{\begin{enumerate}[(1)]}

\newcommand{\diag}{\operatorname{diag}}

\newcommand{\Abstract}

\title
{
Discrete Bakry--\'Emery curvature tensors and matrices of connection graphs
}

\author[1]{Chunyang Hu}\author[2]{Shiping Liu}
\affil[1,2]{School of Mathematical Sciences, University of Science and Technology of China, Hefei 230026, China}
\affil[1]{chunyanghu@mail.ustc.edu.cn}
\affil[2]{spliu@ustc.edu.cn}

\date{\today}

\date{\today}

\theoremstyle{plain}
\newtheorem{lemma}{Lemma}[section]
\newtheorem{theorem}[lemma]{Theorem}
\newtheorem{proposition}[lemma]{Proposition}
\newtheorem{corollary}[lemma]{Corollary}

\theoremstyle{definition}

\newtheorem{definition}[lemma]{Definition}
\newtheorem{remark}[lemma]{Remark}

\newtheorem{example}[lemma]{Example}

\numberwithin{equation}{section}

\begin{document}

\maketitle

\pagestyle{plain}

\begin{abstract}
Curvature plays a fundamental role in geometry, and its discrete analogues are powerful tools for applying geometric and analytical techniques to the study of discrete structures. Connection graphs offer a model for discrete vector bundles and are closely tied to a range of theoretical and practical topics, such as the spectral properties of covering graphs and the analysis of high-dimensional data.
Liu, M\"unch, and Peyerimhoff introduced the notion of Bakry–Émery curvature for connection graphs as a means to derive Buser-type bounds on the eigenvalues of connection Laplacians. In this work, we present a reformulation of the Bakry–Émery curvature at a vertex within a connection graph. Our approach expresses this curvature through the smallest eigenvalue of a set of unitarily equivalent curvature matrices. We interpret these matrices as representations of a newly defined curvature tensor, each corresponding to a different orthonormal basis of the vertex’s tangent space. This framework significantly extends earlier studies by Cushing et al. and Siconolfi on curvature matrices of standard graphs.
It is important to note that the Bakry–Émery curvature in connection graphs can behave very differently from that in the underlying graphs. For instance, constant functions generally fail to serve as eigenfunctions of the connection Laplacian, which poses a substantial challenge when attempting to generalize results from standard graphs to connection graphs. We address this issue by employing the Schur complement, applied twice using pseudoinverses.
Additionally, we investigate the Bakry–Émery curvature in Cartesian products of connection graphs, extending and strengthening the earlier findings of Liu, Münch, and Peyerimhoff. While our results for vertices with locally balanced structures encompass previous work, we also shed light on intriguing behaviors that arise in locally unbalanced connection structures.

\end{abstract}
\section{Introduction}
In this paper, we study the discrete Bakry--\'Emery theory on connection graphs.

The Bakry--\'Emery $\Gamma$-calculus and curvature-dimension inequalities \cite{Bakry94,Bakry,BE85,BGL14} generalize the notion of lower bounds on Ricci curvature from Riemannian geometry to broader settings. In recent years, their discrete counterpart---developed for graphs---has become a vibrant area of research. 
The framework has been initiated in \cite{Elworthy, LY10, Sch99} and led to a deep understanding of the geometric and analytic properties of the underlying graphs, including the properties of local clustering coefficients \cite{JL14}, spectra \cite{CLY14,CKLLS,KKRT16,Salez}, diameter bounds \cite{CLMP20,FS18,LMP2,LMP18}, harmonic functions \cite{Hua19}, heat semigroup behavior \cite{GL17,HL17,HMW19}, and even aspects of graph homology \cite{KMY}.
In particular, we mention the following application of Bakry--\'Emery theory to spectral graph theory: For a finite graph $G=(V,E)$ with $|V|=N$, let us list the eigenvalues of the normalized graph Laplacian with multiplicities as below:
$$0=\lambda_{1}\leq\lambda_{2}\leq\cdots\leq\lambda_{N}\leq 2.$$
It is well known that $\lambda_{2}=0$ if and only if $G$ is connected and $2-\lambda_{N}=0$ if and only if $G$ is bipartite. The Bakry--\'Emery curvature helps to estimate the spectral gap $\lambda_{2}$. Indeed, \[\lambda_{2}\geq K,\] where $K$ is the Bakry--\'Emery lower curvature bound of $G$. Moreover, if $K$ is nonnegative, we have that $\lambda_{2}$ is of the same order as $h^{2}$ up to a constant depending only on the maximal degree, where $h$ stands for the Cheeger constant of $G$. 
Recently, Cushing et al. \cite[Theorem 1.2]{CKLP22} and Siconolfi \cite[Theorem 27]{Siconolfi} independently discovered that the Bakry--\'Emery curvature at each vertex of a graph can be reformulated as the smallest eigenvalue of a symmetric matrix. This matrix is called a \emph{curvature matrix}. The concept of curvature matrices leads to an interesting curvature flow on weighted mixed graphs based on the Bakry--\'Emery curvature \cite{CKLMPS, CKLMPS23}. The curvature matrices also lead to a deeper understanding of how curvature changes under certain operations of graphs \cite{CKLP22} and to explicit calculations of the curvature of specific classes of graphs \cite{CLP,Siconolfi,Siconolfi22,Siconolfi21}.

Connection graphs extend the concept of ordinary graphs and can be viewed as simple models for discrete vector bundles. In a $d$-dimensional connection graph $(G, \sigma)$, the base graph $G = (V, E)$ is equipped with a \emph{connection map} 
\[
\sigma: E^{\mathrm{or}} \to O(d) \quad \text{or} \quad U(d),
\]
which satisfies the condition
\[
\sigma_{xy} = \sigma_{yx}^{-1}.
\]
Here, $E^{\mathrm{or}}$ denotes the set of oriented edges, meaning that for each undirected edge $\{x, y\} \in E$ we include both $(x, y)$ and $(y, x)$. For an oriented edge $(x, y)$, the value $\sigma_{xy}$ specifies the linear transformation associated with that direction. In other words, each directed edge carries a transformation that is the inverse of the transformation assigned to its reverse.


Connection graphs constitute a particular class of \emph{gain graphs}. A gain graph is a graph $G$ coupled with a gain function $\sigma:E^{or}\to H$ satisfying $\sigma_{xy}=\sigma_{yx}^{-1}$, where $H$ is an arbitrary group. Gain graphs are special cases of \emph{biased graphs} \cite{Zaslavsky89}. Gain graphs have also been referred to as \emph{voltage graphs} \cite{Gross}, the motivation of which lies in the Heawood map-coloring problem. With special choices of $H=O(d)$ or $U(d)$, the function $\sigma$ can be considered as a connection of the vector bundle on $G$ \cite{Kenyon}. We use the terminology of \emph{connection graph} following the work of Chung et al. \cite{CZK14}. Connection graphs are also referred to as \emph{metrized local systems} in \cite{JL97}, with a close relation to previous works \cite{CS92,Forman}. In the case of $H=O(1)=\{\pm 1\}$, a connection graph is simply a \emph{signed graph} introduced by Harary. Harary \cite{Harary1, Harary2} introduced the concepts of balanced and anti-balanced signed graphs as models of social networks. For a more historical review, we refer to \cite[Section 3]{LPV} and \cite[Section 2.1]{KMY}. 

Liu, M\"unch, Peyerimhoff \cite{LMP} extended the concept of Bakry--\'Emery curvature on graphs to the setting of connection graphs and derive higher order Buser-type eigenvalue estimates. Recall that Cheeger type eigenvalue estimates for the connection Laplacian have been discussed in \cite{AtayLiu,BSS13,LLPP}. The lower bounds of the Bakry--\'Emery curvature of connection graphs have been characterized in terms of the heat semigroup for functions and the heat semigroup for vector fields of the underlying graphs \cite[Theorem 3.2]{LMP}. In particular, we mention the following application of Bakry--\'Emery curvature of connection graphs to the spectral graph theory. Consider a finite graph $G=(V,E)$ with $|V|=N$, and a connection $\sigma_{0}:E^{or}\rightarrow O(1)=\{\pm1\}$ such that $\sigma_{0}\equiv -1$ \label{sym:antibalanced}. If the Bakry--\'Emery curvature of the connection graph $(G,\sigma_0)$ is lower bounded by $K$, then the maximal eigenvalue $\lambda_{N}$ of the normalized graph Laplacian of $G$ satisfies \[2-\lambda_{N}\geq K.\] If, furthermore, $K$ is nonnegative, we have that $2-\lambda_{N}$ is of the same order as $\beta^{2}$ up to a constant depending only on the maximal degree, where $\beta$ is the bipartiteness constant of Trevisan \cite{Trevisan} or one minus the dual Cheeger constant of Bauer-Jost \cite{BJ}.

It is natural to ask whether the reformulation of Bakry--\'Emery curvature as an eigenvalue problem in \cite{CKLP22} and \cite{Siconolfi} can be extended to the setting of connection graphs.

It should be noted that the Bakry–\'Emery curvature for connection graphs can behave very differently from that of the original graphs. For example, it is known that if a graph has uniformly bounded degrees and a strictly positive lower bound on its Bakry–Émery curvature, then the graph must be finite \cite{KMY,LMP18}. In contrast, we construct an example of an infinite $4$-regular connection graph that still possesses a positive Bakry–\'Emery curvature bound (see Section \ref{section:counterexample}).


An important feature used in the proofs of \cite{CKLP22} and \cite{Siconolfi} is that, in their setting, constant functions act as eigenfunctions of the Laplacian with eigenvalue zero, and similarly for the operators $\Gamma$ and $\Gamma_2$ in the Bakry--Émery calculus. In contrast, for connection Laplacians this property typically does not hold: constant functions are generally not eigenfunctions with eigenvalue zero for the Laplacian, nor for the corresponding operators $\Gamma^\sigma$ and $\Gamma_2^\sigma$ defined in \cite{LMP}. This difference creates a notable challenge when attempting to adapt the arguments from \cite{CKLP22} and \cite{Siconolfi} to the context of connection graphs.


We revisit the main approach from \cite{CKLP22}, which reformulates Bakry–Émery curvature as an eigenvalue problem. The goal is to determine the largest value $K$ for which $\Gamma_2 - K\Gamma$ remains positive semidefinite. Because constant functions correspond to the zero eigenvalue, the first row and column of the matrix representation of $\Gamma_2 - K\Gamma$ can be omitted. A crucial insight is that the matrix size can be reduced further using the Schur complement. In particular, removing the first row and column from $\Gamma$ yields a diagonal matrix, which is congruent to the identity. This allows $K$ to be expressed in terms of an eigenvalue.
However, for the connection Laplacian, constant functions generally do not correspond to the zero eigenvalue, meaning we cannot simply discard the first row and column of $\Gamma^\sigma$. To address this, we employ the Schur complement in terms of pseudoinverses, following the framework developed by Albert \cite{Albert}. Verifying that Albert’s conditions apply in our context is a subtle task (see Proposition \ref{prop:aadaggeromega}).
Ultimately, this leads to a reformulation of Bakry–Émery curvature at a vertex of a connection graph as the smallest eigenvalue of a certain curvature matrix. In fact, for each vertex, we identify a family of curvature matrices that are unitarily equivalent (see Theorem \ref{thm:curvature_eigenvalue} and Proposition \ref{prop:unitarily_equivalent}).


Building on the approach introduced by Siconolfi \cite{Siconolfi}, we can associate to each vertex $x$ both a curvature tensor and a metric tensor (see Definition \ref{def:curvature_metric_tensors}). These are constructed using the sesquilinear forms $\Gamma_2^\sigma$ and $\Gamma^\sigma$, along with function extensions from the 1-sphere to the 2-ball obtained via the Schur complement.
The tangent space $T_x(G,\sigma)$ at $x$ (see Definition \ref{def:tangent_space}) is a linear space, which becomes an inner product space once the metric tensor is applied.
A curvature matrix is then the representation of the curvature tensor with respect to an orthonormal basis of the inner product space $(T_x(G,\sigma), g_x(\cdot,\cdot))$.

We further investigate the property of the Bakry--\'Emery curvature at a vertex in a connection graph as a function of the dimension parameter $N\in (0,\infty]$. We show that this function is continuous, monotone non-decreasing and concave. This  extends the previous approach on Bakry--\'Emery curvature function of graphs \cite{CKLP22,CLP}. Moreover, we prove that the curvature function is constant on $[N,\infty]$, if the multiplicity of the smallest eigenvalue of the $N$-dimensional curvature matrix is at least equal to the dimension $d$ of the connection (see Theorem \ref{thm:curvature_function}).


We present decomposition formulas for the curvature matrices of the Cartesian product of two connection graphs. These formulas relate the curvature matrices of the product to the curvature matrices associated with the corresponding vertices in each individual graph (see Theorem \ref{thm:Adecomp} and Corollary \ref{coro:AN}). With these tools, we derive several new results concerning the Bakry–Émery curvature of Cartesian products of connection graphs (see Corollaries \ref{coro:Cartesian}, \ref{coro:U1infty}, and \ref{coro:U1N}), extending and strengthening earlier work from \cite{LMP}. In particular, consider two connection graphs $(G,\sigma)$ and $(G',\sigma')$ with one-dimensional connections, and let $x$ and $x'$ be two vertices of $G$ and $G'$, respectively.
If at least one of these vertices has a \emph{balanced} local connection structure, we prove that the curvature function of the Cartesian product at $(x,x')$ can be expressed as the star product (see \cite[Definition 7.1]{CLP} or Definition \ref{def:star_product}) of the curvature functions of of $G$ at $x$ and that of $G'$ at $x'$.



We study how the curvature at a given vertex $x$ in a connection graph is affected by two types of operations to its local structure. The first operation involves adding an edge with connection between two neighbors of $x$. The second involves merging two vertices in the 2-sphere of $x$ that share no common neighbors in the $1$-sphere. For the second operation, curvature never decreases, whereas for the first, the effect on curvature depends on the connection of the newly added edge.




Before diving into an in-depth analysis of the findings, let us start with a brief summary of what each section covers.\\
$\bullet$ Section \ref{section:pre}: Preliminaries on connection graphs, discrete Bakry--\'Emery curvature and Schur complement.\\
$\bullet$ Section \ref{section:cur_matrix}: Curvature matrices of connection graphs.\\
$\bullet$ Section \ref{section:tensor}: Curvature tensors of connection graphs.\\
$\bullet$ Section \ref{section:function}: Properties of curvature functions.\\
$\bullet$ Section \ref{section:product}: Curvature matrices of Cartesian products.\\
$\bullet$ Section \ref{section:operation}: How the curvature changes under local operations.\\
$\bullet$ Section \ref{section:counterexample}: An infinite connection graph with positive Bakry--\'Emery curvature lower bound.\\
$\bullet$ Appendix \ref{section:appendix}: Explicit calculations of the matrices $\Gamma^{\sigma}_{2}(x)$ \label{notation:gamma_2}and $Q(x)$\label{notation:schur_complement}.\\
$\bullet$ Appendix \ref{section:notation}: Glossary of notations.\\



\section{Preliminaries}\label{section:pre}
In this section, we prepare notations and  preliminaries on connection graphs, Bakry--\'Emery curvature and Schur complement.

\subsection{Connection graphs}
Let $G=(V,w,\mu)$ be a weighted graph with a vertex set $V$, a vertex measure $\mu: V\to \mathbb{R}^+$, and an edge-weight function $w: V\times V\to \mathbb{R}^+\cup \{0\}$ which satisfies $w_{xy}=w_{yx}$ and $w_{xx}=0$ for all $x,y\in V$. The edge set $E\subset V\times V$ of $G$ is defined as $E:=\{\{x,y\}: x,y\in V, w_{xy}\neq 0\}$. In particular, there are no multiple edges and self-loops in $E$. We say a vertex $x$ is adjacent to another vertex $y$, denoted by $x\sim y$, if $\{x,y\}\in E$.  For any two vertices $x$ and $y$, the distance $d(x,y)$ is the length of the shortest path connecting them. We denote the ball centered at $x$ with radius $r$ by $B_{r}(x)=\{y\in V:d(x,y)\leq r\}$\label{notation:Br_x}, and the $r$-sphere of $x$ by $S_{r}(x)=\{y\in V:d(x,y)=r\}$\label{notation:Sr_x}. A graph $G$ is \emph{locally finite} if the $1$-sphere $S_1(x)$ of every vertex $x$ contains only finitely many vertices. We restrict ourselves to locally finite weighted graphs in this article.

The degree $d_x$ of a vertex $x$ is defined as $d_x:=\sum _{y\in V} w_{xy}$. We also use the notation of transition rate $p_{xy}$ from $x$ to $y$, where
\begin{equation*}\label{notation:transition_rate}
p_{xy}:=\frac{w_{xy}}{\mu_x}.
\end{equation*}
For the case of $\mu_x=d_x$ for all $x\in V$, $p_{xy}$ is the transition probability from $x$ to $y$ of a random walk (a reversible Markov chain).

We denote by $E^{or}:=\{(x,y),(y,x):\{x,y\}\in E\}$ the set of oriented edges, one in each direction, for each edge in $E$. For an integer $d$, a connection graph $(G,\sigma)$ with $d$-dimensional connections is a weighted graph $G$ equipped with a connection function $\sigma:E^{or}\rightarrow O(d)\,\,\text{or}\,\,U(d)$ satisfying $\sigma_{xy}=\sigma_{yx}^{-1}$, where we write $\sigma_{xy}:=\sigma((x,y))$ \label{sym:connection} for short.


The \emph{signature of a cycle} $C=x_{0}\sim x_{1}\sim x_{2}\sim\dots x_{n}\sim x_{0}$ is defined as the conjugacy class of the product  $\sigma_{C}=\sigma_{x_{0}x_{1}}\sigma_{x_{1}x_{2}}\dots\sigma_{x_{n}x_{0}}$. A connection graph is called \emph{balanced} if all of its cycles have signature (the conjugacy class of) $I_d$.

A switching function is a function $\tau: V\rightarrow U(d)$ or $O(d)$. Switching the connection function $\sigma$ by the function $\tau$ means replacing $\sigma$ by $\sigma^\tau$ where $\sigma^{\tau}_{xy}:=\tau(x)^{-1}\sigma_{xy}\tau(y)$\label{notation:switching} for each edge $(x,y)\in E^{or}$. We say two connection functions $\sigma_{1}$ and $\sigma_{2}$ are \emph{switching equivalent} if there exists a switching function $\tau$ such that $\sigma_{1}=\sigma^{\tau}_{2}$. It is direct to check that switching equivalence is indeed an equivalent relation.
A property or a quantity is called \emph{switching invariant} if it is preserved by switching operations.
For example, the signature of any cycle is switching invariant. Zaslavsky \cite[Corollary 3.3 and Section 9]{Zaslavsky} obtained the following
characterization lemma of balancedness.
\begin{lemma}[Zaslavsky's switching lemma]A connection graph $(G,\sigma)$ is balanced if and
only if $\sigma$ is switching equivalent to the all $I_d$ signature.
\end{lemma}

\subsection{Bakry--\'Emery curvature}
\par In this section, we discuss the Bakry--\'Emery curvature of connection graphs introduced in \cite{LMP}. The definition of the curvature is motivated by the Bochner identity in Riemannian geometry and Bakry--\'Emery $\Gamma$-calculus \cite{Bakry94,Bakry, BE85}. We give a brief review here. For any smooth function $f:M\rightarrow \mathbb{R}$ on a Riemannian manifold $M$, the following Bochner identity holds:
$$\frac{1}{2}\Delta |\textrm{grad} f|^{2}(x) = |\textrm{Hess} f|^2(x) + \langle \textrm{grad} \Delta f(x),\textrm{grad} f(x)\rangle + \textrm{Ric}(\textrm{grad} f,\textrm{grad} f)(x),$$
where $\textrm{Hess} f$ is the Hessian of $f$, $\Delta$ is the Laplace-Beltrami operator and $\mathrm{Ric}$ stands for the Ricci curvature tensor. By Cauchy-Schwarz inequality, the first term on the right hand side satisfies
$$|\textrm{Hess} f|^2(x)\geq \frac{1}{n}(\Delta f(x))^{2}.$$
Suppose the Ricci curvature is bounded below by $K$ at $x\in M$, i.e.,
$$\textrm{Ric}(v,v)\geq K\langle v,v\rangle, \,\, \text{for any}\,\, v\in T_{x}M.$$
Then the above Bochner identity tells
\begin{equation}\label{eq:cdMflds}
\frac{1}{2}\Delta |\textrm{grad} f|^{2}(x)-\langle \textrm{grad} \Delta f(x),\textrm{grad} f(x)\rangle\geq\frac{1}{n}(\Delta f(x))^{2}+K|\textrm{grad} f|^{2}(x).
\end{equation}
Bakry--\'Emery \cite{BE85} introduced the following definitions of $\Gamma$ and $\Gamma_2$: For any two smooth functions $f,g:M\rightarrow \mathbb{R}$,
 \begin{align*}
 2\Gamma(f,g)&:= \Delta(fg)-f\Delta(g)-g\Delta(f),\\
  2\Gamma_2(f,g)&:=\Delta\Gamma(f,g)-\Gamma(f,\Delta g)-\Gamma(g,\Delta f).
  \end{align*}
   Noticing that
   $$\Gamma(f,g)=\langle\textrm{grad} f , \textrm{grad} g\rangle,\,\,\Gamma_2(f,f)=\frac{1}{2}\Delta |\textrm{grad} f|^{2}-\langle\textrm{grad} \Delta f,\textrm{grad} f\rangle,$$
   the inequality (\ref{eq:cdMflds}) can be reformulated as
   \begin{equation}\label{eq:cd}
   \Gamma_{2}(f,f)(x)\geq \frac{1}{n}(\Delta f(x))^{2}+K\Gamma(f,f)(x).
   \end{equation}
 On an $n$-dimensional Riemannian manifold $M$, (\ref{eq:cd}) holds at $x\in M$ for any smooth function $f$ if and only if the Ricci curvature is bounded from below by $K$ at $x$ \cite[pp.93-94]{Bakry94}.

Let $\mathbb{K}=\mathbb{R}$ or $\mathbb{C}$ \label{notation:real_complex_space}. For a function $f:V\rightarrow \mathbb{K}^d$ defined on a connection graph $(G,\sigma)$ with $d$-dimensional connections, the Laplace operator $\Delta^{\sigma}$ is defined as follows: for any vertex $x$,
$$\Delta^{\sigma} f(x):=\sum_{y,y\sim x}p_{xy}(\sigma_{xy}f(y)-f(x)).$$

This operator has been introduced as the \emph{connection Laplacian} by Singer and Wu \cite{SW12}. The special choices of $H=O(d)$ or $U(d)$ ensure that the operator $\Delta^\sigma$ is self-adjoint. For the case $H=U(1)$, this operator is known as discrete magnetic Laplacian \cite{Shubin, Sunada}; For the case $H=O(1)=\{\pm 1\}$, it is known as the signed Laplace matrix \cite{Zaslavsky10}. 

The following definition of Bakry--\'Emery curvature of connection graphs based on the Laplacian $\Delta^\sigma$ has been introduced in \cite[Definition 3.2]{LMP}.
We first define the corresponding operators $\Gamma^{\sigma}$ and $\Gamma_{2}^{\sigma}$: For any two functions $f,g:V\rightarrow \mathbb{K}^d$, we define 
\begin{align*}
2\Gamma^{\sigma}(f,g):=&\Delta(f^\top\overline{g})-f^\top\overline{\Delta^{\sigma} g}-(\Delta^{\sigma}f)^\top\overline{g};\\
2\Gamma_{2}^{\sigma}(f,g):=&\Delta(\Gamma^{\sigma}(f,g))-\Gamma^{\sigma}(f,\Delta^{\sigma}g)-\Gamma^{\sigma}(\Delta^{\sigma}f,g),
\end{align*}
where $f^\top\overline{g}: V\to \mathbb{K}$ is the function such that $(f^\top\overline{g})(x):=f(x)^\top\overline{g(x)}$ for any $x\in V$.
We also write $\Gamma^{\sigma}(f):=\Gamma^{\sigma}(f,f)$ and $\Gamma_{2}^{\sigma}(f)=\Gamma_{2}^{\sigma}(f,f)$ for simplicity.

\begin{definition}(Bakry--\'Emery curvature)\label{def:BEcurvature}
Let $(G,\sigma)$ be a connection graph. A vertex $x$ is said to satisfy the Bakry--\'Emery \emph{curvature dimension inequality} $CD^\sigma(\K,N)$, where $\K\in \mathbb{R}$ and $N\in(0,\infty]$, if for every function $f:V\rightarrow \mathbb{K}^d$ the inequality
$$\Gamma_{2}^{\sigma}(f,f)(x)\geq\frac{1}{N}|\Delta^{\sigma} f(x)|^{2}+\K\Gamma^{\sigma}(f,f)(x)$$
holds.
Here, $\K$ serves as a lower curvature bound and $N$ plays the role of a dimension parameter. The \emph{$N$-Bakry--\'Emery curvature} of $x$, denoted $\K_{G,\sigma,x}(N)$, is defined as the supremum of all $\K$ values for which the above inequality holds.
The mapping
\[\K_{G,\sigma,x}: (0,\infty]\to \mathbb{R}\]
is referred to as the \emph{Bakry--\'Emery curvature function} at $x$.
\end{definition}

When $N=\infty$, the above curvature dimension inequality becomes
$$\Gamma_{2}^{\sigma}(f,f)(x)\geq\K\Gamma^{\sigma}(f,f)(x).$$

The following switching invariance property has been shown in \cite[Proposition 3.5]{LMP}.
\begin{proposition}\label{prop:switch}
The curvature function at any vertex of a connection graph $(G,\sigma)$ is switching invariant. Indeed, for any switching function $\tau$ and functions $f,g:V\to \mathbb{K}^d$, we have
\begin{equation}
\Gamma^{\sigma^\tau}(f,g)=\Gamma^{\sigma}(\tau f, \tau g), \,\,\Gamma_2^{\sigma^\tau}(f,g)=\Gamma_2^{\sigma}(\tau f, \tau g),
\end{equation}
and
\begin{equation}
\Delta^{\sigma^\tau}=D(\tau)^{-1}\Delta^\sigma D(\tau),
\end{equation}
where $D(\tau)$ \label{notation:diagonal_tau} stands for the diagonal matrix of the function $\tau$.
\end{proposition}

Observe that the curvature function $\K_{G,\sigma,x}(\cdot)$ at a vertex $x$ depends only on the local structure of $B_2(x)$.
Let us denote
\[S_1(x)=\{y_1,\ldots,y_m\},\,\,\text{and}\,\,S_2(x)=\{z_1,\ldots,z_n\}\label{notation:x_S2}.\]
We denote by $\Delta^\sigma(x)$ the following $(m+1)d\times d$ matrix
\begin{equation}\label{eq:Delta_x}\Delta^\sigma(x)^\top:=\begin{pmatrix}
    -\frac{d_{x}}{\mu_x}I_{d} & p_{xy_1}\sigma_{xy_{1}} & p_{xy_2}\sigma_{xy_{2}} & \dots & p_{xy_m}\sigma_{xy_{m}}
\end{pmatrix}.\end{equation}
For any functions $f,g: V\to \mathbb{K}^{d}$, we denote the corresponding local $(m+1)d$-vectors by
\[\vec{f}(x)^\top:=(f(x)^\top, f(y_1)^\top, \ldots, f(y_m)^\top),\,\,\vec{g}(x)^\top:=(g(x)^\top, g(y_1)^\top, \ldots, g(y_m)^\top).\]
Then we have
\[(\Delta^\sigma f(x))^\top\overline{\Delta^\sigma g(x)}=\vec{f}(x)^\top\Delta^\sigma(x)\overline{\Delta^\sigma(x)}^\top\overline{\vec{g}(x)}.\]
That is, the matrix $\Delta^\sigma(x)\overline{\Delta^\sigma(x)}^\top$ is the Hermitian matrix corresponding to the symmetric sesquilinear form $(\Delta^\sigma f(x))^\top\overline{\Delta^\sigma g(x)}$ at $x$.
Let $\Gamma^\sigma(x)$ and $\Gamma_2^\sigma(x)$ be the Hermitian matrices corresponding to the symmetric sesquilinear forms $\Gamma^\sigma$ and $\Gamma_2^\sigma$ at $x$, respectively.

The explicit expressions of the non-trivial blocks of these two matrices for combinatorial graphs have been derived in \cite[Section 3.4]{LMP}. Recall from \cite[Proposition 3.1]{LMP} that for any $f,g:V\to \mathbb{K}^d$
\begin{equation}\label{eq:2Gamma}
    2\Gamma^\sigma(f,g)(x)=\sum_{y\in V}p_{xy}(\sigma_{xy}f(y)-f(x))^\top\overline{(\sigma_{xy}g(y)-g(x))}.
\end{equation}
Therefore, the matrix $\Gamma^{\sigma}(x)$ has a  non-trivial block of size $d|B_1(x)|\times d|B_1(x)|$:
\begin{equation}\label{eq:GammaMatrix}
    2\Gamma^{\sigma}(x)=
\begin{pmatrix}
   \sum_{i=1}^mp_{xy_i}I_{d}                 & -p_{xy_1}\overline{\sigma_{xy_{1}}} & -p_{xy_2}\overline{\sigma_{xy_{2}}} & \dots  & -p_{xy_m}\overline{\sigma_{xy_{m}}}\\
    -p_{xy_1}\sigma_{xy_{1}}^\top      & p_{xy_1}I_{d}                 & \mathbf{0}_d                            & \dots  & \mathbf{0}_d                           \\
    -p_{xy_2}\sigma_{xy_{2}}^\top      & \mathbf{0}_d                            & p_{xy_2}I_{d}                  & \dots  & \mathbf{0}_d                           \\
    \vdots                            & \vdots                       & \vdots                       & \ddots & \vdots                      \\
    -p_{xy_m}\sigma_{xy_{m}}^\top      & \mathbf{0}_d                            & \mathbf{0}_d                            & \dots  & p_{xy_m}I_{d}
\end{pmatrix}.
\end{equation}

The matrix $\Gamma_{2}^{\sigma}(x)$ has a non-trivial block of size $d|B_2(x)|\times d|B_2(x)|$:
\begin{equation}\label{eq:Gamma2}
    \Gamma^{\sigma}_{2}(x)=
\begin{pmatrix}
    \Gamma_2^{\sigma}(x)_{x,x} & \Gamma_2^{\sigma}(x)_{x,S_1} & \Gamma_2^{\sigma}(x)_{x,S_2}\\
    \Gamma_2^{\sigma}(x)_{S_1,x} & \Gamma_2^{\sigma}(x)_{S_1,S_1} & \Gamma_2^{\sigma}(x)_{S_1,S_2}\\
    \Gamma_2^{\sigma}(x)_{S_2,x} & \Gamma_2^{\sigma}(x)_{S_2,S_1} & \Gamma_2^{\sigma}(x)_{S_2,S_2}
\end{pmatrix},
\end{equation}
The explicit expression is given in the Appendix \ref{section:appendix}. We only mention here that the block $(\Gamma_2^{\sigma})_{S_2,S_2}$ is diagonal and positive definite.

Since the block $(\Gamma_2^{\sigma})_{S_2,S_2}$ is diagonal, the matrix $\Gamma_2^\sigma(x)$ and hence the curvature function $\K_{G,\sigma,x}(\cdot)$ are completely determined by the \emph{local connection structure}. Specifically, they depend on the incomplete $2$-ball $B_2^{inc}(x)$ \label{notation:B2_inc_x} around $x$. This structure is obtained from the induced subgraph of $B_2(x)$ by removing edges connecting vertices in $S_2(x)$ and assigning a restricted connection function of $\sigma$.

For the matrix $\Gamma^\sigma(x)$, we have the following observation.
\begin{proposition}\label{prop:Gamma}
Let $(G,\sigma)$ be a connection graph with $d$-dimensional connections. At any vertex $x$, the eigenvalues of the matrix $\Gamma^\sigma(x)$ can be listed with multiplicity as
\[0=\lambda_1=\cdots=\lambda_d<\lambda_{d+1}\leq \cdots \leq \lambda_{(m+1)d}.\]
\end{proposition}
\begin{proof}
It follows directly from (\ref{eq:2Gamma}) that
\[2\Gamma^\sigma(f,f)(x)=\sum_{y\in V}p_{xy}(\sigma_{xy}f(y)-f(x))^2\geq 0,\]
where the equality holds if and only if $f(y)=\sigma_{xy}^{-1}f(x)$ for any $y\in S_1(x)$. This proves the proposition.
\end{proof}

\subsection{Schur complement}
We reformulate the Bakry--\'Emery curvature as the smallest eigenvalue of a family of curvature matrices. A fundamental tool in our reformulation is the so-called \emph{Schur complement}. Let $S$ be an Hermitian matrix
\begin{equation}\label{eq:S}S=\begin{pmatrix}S_{11} & S_{12}\\ S_{21} & S_{22}
\end{pmatrix},\end{equation}
where $S_{11}$ and $S_{22}$ are square blocks. The \emph{Schur complement} $S/S_{11}$ of $S_{11}$ in $S$ is defined as
\begin{equation}\label{eq:schur_complement}
S/S_{11}:=S_{22}-S_{21}S_{11}^{\dagger}S_{12},
\end{equation}
where $S_{11}^{\dagger}$ \label{notation:S_dagger} stands for the pseudoinverse of the matrix $S_{11}$. The following result has been proved in \cite[Theorem 1(i)]{Albert}.
\begin{theorem}[Albert]\label{thm:Albert}
Let $S$ be an Hermitian matrix defined as above. Then $S$ is positive semidefinite (denoted by $S\succeq 0$\label{notation:positive_semidefinite}) if and only if $S_{11}S_{11}^\dagger S_{12}=S_{12}$, $S_{11}\succeq 0$ and $S/S_{11}\succeq 0$.
\end{theorem}
\begin{proof}
In \cite[Theorem 1(i)]{Albert}, the proof was originally established for real symmetric matrices. The same reasoning extends naturally to Hermitian matrices without modification. For the purposes of our subsequent discussion, we provide a slightly different proof in terms of the associated sesquilinear forms.

  Assuming that
  \begin{equation}\label{eq:Schur_condition}
S_{11}\succeq 0, \,\,\text{and}\,\,S_{11}S_{11}^\dagger S_{12}=S_{12},
  \end{equation}
the following relation can be established: For any vectors $v_1, v_2$ whose dimensions coincide with the sizes of rows in $S_{11}, S_{12}$, respectively, we have
  \begin{align}
   & \begin{pmatrix}
      v_1^\top & v_2^\top
    \end{pmatrix}
  \begin{pmatrix}S_{11} & S_{12}\\ S_{21} & S_{22}
\end{pmatrix}
\begin{pmatrix}
  \overline{v_1} \\ \overline{v_2}
\end{pmatrix}
=v_2^\top (S/S_{11})\overline{v_2}+\left|(S_{11}^{1/2})^\dagger S_{12}\overline{v_2}+\overline{S_{11}^{1/2}}^\top \overline{v_1}\right|^2. \label{eq:Schur_sesquilinear_form}
  \end{align}
In the above, we use the identity $S_{11}^{1/2}(S_{11}^{1/2})^\dagger S_{12}=S_{12}$\label{notation:S_square_root}, which is a consequence of the condition $S_{11}(S_{11})^\dagger S_{12}=S_{12}$
and $S_{11}=S_{11}^{1/2}\overline{S_{11}^{1/2}}^\top$.

Given condition (\ref{eq:Schur_condition}), it follows directly from (\ref{eq:Schur_sesquilinear_form}) that if $S/S_{11}\succeq 0$, then $S\succeq 0$.

For the reverse implication, assume that $S\succeq 0$. Then there exists a matrix $H$ such that $S=H\overline{H}^\top$. Let us write
\[H=\begin{pmatrix}
  X \\ Y
\end{pmatrix},\]
where the number of rows in $X$ (in $Y$) coincides with the size of $S_{11}$ (of $S_{12}$). Thus, we have $S_{11}=X\overline{X}^\top$, $S_{12}=\overline{S_{21}}^{\top}=X\overline{Y}^\top$, and $S_{22}=Y\overline{Y}^\top$. Hence, we obtain
\[S_{11}S_{11}^\dagger S_{12}=X\overline{X}^\top\left(X\overline{X}^\top\right)^\dagger S_{12}=XX^\dagger \left(X\overline{Y}^\top\right)=X\overline{Y}^\top=S_{12},\]
by using the relations $\overline{X}^\top\left(X\overline{X}^\top\right)^\dagger=X^\dagger$ and $XX^\dagger X=X$. That is, the fact $S\succeq 0$ leads to condition (\ref{eq:Schur_condition}), and with (\ref{eq:Schur_sesquilinear_form}) applied, we conclude that $S/S_{11}\succeq 0$.
\end{proof}
\begin{remark}
If $S_{11}$ is positive definite (denoted by $S_{11}\succ 0$), then the above Theorem states that $S\succeq 0$ if and only if the Schur complement $S/S_{11}\succeq 0$.
\end{remark}
For later purposes, we prove the following technical lemma.
\begin{lemma}\label{lemma:schur_technical}
Let $S$ be an Hermitian matrix defined in (\ref{eq:S}) such that $S_{11}\succeq 0$ and $S_{11}S_{11}^\dagger S_{12}=S_{12}$. Let $v_1, v_2$ be vectors whose dimensions coincide with the sizes of rows in $S_{11}, S_{12}$, respectively. Then the norm square term  in (\ref{eq:Schur_sesquilinear_form})
\[\left|(S_{11}^{1/2})^\dagger S_{12}\overline{v_2}+\overline{S_{11}^{1/2}}^\top \overline{v_1}\right|^2\]
 vanishes if and only if 
\begin{equation}
S\begin{pmatrix}
  \overline{v_1} \\ \overline{v_2}
\end{pmatrix}=\begin{pmatrix}
  S_{11} & S_{12} \\
  S_{21} & S_{22}
\end{pmatrix}\begin{pmatrix}
  \overline{v_1}\\\overline{v_2}
\end{pmatrix}=\begin{pmatrix}
  \mathbf{0} \\ *
\end{pmatrix}.
\end{equation}
\end{lemma}
\begin{proof}
Since the identity $X^\dagger X X^\dagger=X^\dagger$ holds for any matrix $X$, we have \[(S_{11}^{1/2})^\dagger S_{12}\overline{v_2}+\overline{S_{11}^{1/2}}^\top \overline{v_1}=0\] if and only if \[S_{11}^{1/2}(S_{11}^{1/2})^\dagger S_{12}\overline{v_2}+S_{11}^{1/2}\overline{S_{11}^{1/2}}^\top \overline{v_1}=0.\] Using the assumption $S_{11}S_{11}^\dagger S_{12}=S_{12}$, the latter simplifies to
$S_{11}\overline{v_1}+S_{12}\overline{v_2}=0$.
\end{proof}

\section{Curvature matrices of connection graphs}\label{section:cur_matrix}
By Definition \ref{def:BEcurvature}, the $N$-Bakry--\'Emery curvature $\K_{G,\sigma,x}(N)$ of the vertex $x$ of a connection graph $(G,\sigma)$ is the solution of the following semidefinite programming problem:
\begin{align}
   & \mathrm{maximize}\,\,\, K\notag\\
   & \mathrm{subject\,\, to} \,\,\, \Gamma^\sigma_2(x)-\frac{1}{N}\Delta^\sigma(x)\overline{\Delta^\sigma(x)}^\top-K\Gamma^\sigma(x)\succeq 0.\label{eq:subjectto}
\end{align}
Since the block $\Gamma_2^\sigma(x)_{S_2,S_2}\succ 0$, see \eqref{eq:Gamma2S2S2}, we derive by Theorem \ref{thm:Albert} that the inequality (\ref{eq:subjectto}) holds if and only if
\begin{equation}\label{eq:Q}
    Q(x)-\frac{1}{N}\Delta^\sigma(x)\overline{\Delta^\sigma(x)}^\top-K\Gamma^\sigma(x)\succeq 0,
\end{equation}
where
\begin{equation}\label{eq:Qdef}
      Q(x):=\Gamma_2^\sigma(x)/\Gamma^\sigma_2(x)_{S_2,S_2}=\Gamma^\sigma_2(x)_{B_1,B_1}-\Gamma^\sigma_2(x)_{B_1,S_2}\Gamma^\sigma_2(x)_{S_2,S_2}^{-1}\Gamma^\sigma_2(x)_{S_2,B_1}
\end{equation}
is the Schur complement of $\Gamma^\sigma_2(x)_{S_2,S_2}$ in the matrix $\Gamma_2^\sigma(x)$. Notice that $Q(x)$ is Hermitian.

Let $B$ be an $(m+1)d\times (m+1)d$ nonsingular matrix such that
\begin{equation}\label{eq:B}
    B(2\Gamma^\sigma(x))\overline{B}^\top=\begin{pmatrix}
        \mathbf{0}_{d} & \mathbf{0}_{d\times md} \\
        \mathbf{0}_{md\times d} & I_{md}
    \end{pmatrix},
\end{equation}where we use the notation $I_{md}$ \label{notation_identity_matrix}for a $md\times md$ identity matrix, and $\mathbf{0}_{d}$\label{notation:zero_matrix} for a $d\times d$ zero matrix.

For convenience, we introduce the following notations:
\begin{align}\label{eq:generalB}
B=\begin{pmatrix}
b_0^\top \\
b_1^\top\\
\vdots\\
b_m^\top
\end{pmatrix},
\end{align}
where for each $i=0,1,\ldots, m$ we further denote the $d\times (m+1)d$ matrix $b_{i}^\top$ as below, 
\[b_i^\top=\begin{pmatrix}
    b_{i1}^\top \\
    \vdots\\
    b_{id}^\top
\end{pmatrix}
\]
with $b_{ij}\in \mathbb{K}^{(m+1)d}$, $j=1,\ldots,d$.

Recall from Proposition \ref{prop:Gamma} that the zero eigenvalue of $2\Gamma^\sigma(x)$ has multiplicity $d$ and all the other eigenvalues are positive. Hence, a matrix $B$ satisfying (\ref{eq:B}) does exist. Indeed, such a matrix is not unique.

From equation (\ref{eq:B}), we observe that  $b_0^\top 2\Gamma^\sigma(x)\overline{b_0}=\mathbf{0}_d$. Since $2\Gamma^\sigma(x)$ is a positive semidefinite Hermitian matrix, it follows that $\overline{b_0}$ consists of $d$ mutually independent eigenvectors $\overline{b_{0i}},\ i = 1, \ldots, d$, each associated with the zero eigenvalue of the matrix $2\Gamma^\sigma(x)$. By (\ref{eq:2Gamma}), any $f:V\to \mathbb{K}^d$ such that $\Gamma^{\sigma}(f,f)(x)=0$ satisfies
$$\sigma_{xy}f(y)=f(x),\text{ for }y\in S_{1}(x).$$
Hence, we derive that $b_0^\top=E_d p_0^\top$ for some nonsingular $d\times d$ matrix $E_d$, where
\begin{equation}\label{eq:p_0}
    p_0^\top:=
\begin{pmatrix}
  I_d & \overline{\sigma_{xy_1}} & \ldots & \overline{\sigma_{xy_m}}
\end{pmatrix}.
\end{equation}

The blocks $b_i, i=1,\ldots,m$ can be constructed from the nonzero eigenvalues and the corresponding eigenvectors of $2\Gamma^\sigma(x)$. They can also be given by simply setting $b_i=p_i, i=1,\ldots, m$ with
\begin{equation}\label{eq:specialb_i}
p_1^\top:=\begin{pmatrix}
   \mathbf{0}_d &  \frac{1}{\sqrt{p_{xy_1}}}I_d & \mathbf{0}_d & \cdots & \mathbf{0}_d
\end{pmatrix},\,\,\cdots,\,\, p_m^\top:=\begin{pmatrix}
   \mathbf{0}_d &  \mathbf{0}_d & \mathbf{0}_d & \cdots & \frac{1}{\sqrt{p_{xy_m}}}I_d
\end{pmatrix}.
\end{equation}

The fact that $b_0$ consists of $d$ mutually independent eigenvectors to the zero eigenvalue of $2\Gamma^\sigma(x)$ implies that
\begin{equation}\label{eq:Gamma_Laplacian}
    b_0^\top \Delta^\sigma(x)=\mathbf{0}_{d}.
\end{equation}

The inequality (\ref{eq:Q}) holds if and only if
\begin{equation}\label{eq:BQ}
    2BQ(x)\overline{B}^\top-\frac{2}{N}(B\Delta^\sigma(x))(\overline{B\Delta^\sigma(x)})^\top-K\begin{pmatrix}
        \mathbf{0}_{d} & \mathbf{0}_{d\times md} \\
        \mathbf{0}_{md\times d} & I_{md}
    \end{pmatrix}\succeq 0.
\end{equation}
Due to (\ref{eq:Gamma_Laplacian}), we can denote
\begin{equation}\label{eq:v0}
B\Delta^\sigma(x):=\begin{pmatrix}
  \mathbf{0}_d \\ \mathbf{v}_0(x)
\end{pmatrix},
\end{equation}
where $\mathbf{v}_{0}(x):=\mathbf{v}_0(G,\sigma,x,B)$ is a $md\times d$ matrix. Then, we have
\begin{equation}\label{eq:BLaplacian}
B\Delta^\sigma(x)\overline{B\Delta^\sigma(x)}^\top=\begin{pmatrix}
        \mathbf{0}_{d} & \mathbf{0}_{d\times md} \\
        \mathbf{0}_{md\times d} & \mathbf{v}_0\overline{\mathbf{v}_0}^\top
    \end{pmatrix}.
\end{equation}
 We further introduce the notation
\[2BQ(x)\overline{B}^\top=\begin{pmatrix}
    a & \omega^\top \\
    \overline{\omega} & (2BQ(x)\overline{B}^\top)_{\hat{1}}\label{notation:removing_d_rows}
\end{pmatrix},\]
where
\begin{equation}\label{eq:a}a:=a(G,\sigma,x,B)=2b_0^\top Q(x)\overline{b_0},\end{equation}
\begin{equation}\label{eq:omega}\omega^\top:=\omega(G,\sigma,x,B)^\top=
\begin{pmatrix}
  2b_0^\top Q(x)\overline{b_1} & \cdots & 2b_0^\top Q(x)\overline{b_m}
\end{pmatrix},
\end{equation}
and
\begin{equation}\label{eq:BQB}(2BQ(x)\overline{B}^\top)_{\hat{1}}=\begin{pmatrix}
    b_1^\top \\ \vdots \\b_m^\top
\end{pmatrix}2Q(x)\begin{pmatrix}
    \overline{b_1} & \cdots & \overline{b_m}
\end{pmatrix}
\end{equation}
is the $md\times md$ matrix obtained from $2BQ(x)\overline{B}^\top$ via removing the first $d$ columns and the first $d$ rows.
Therefore, (\ref{eq:BQ}) can be reformulated as
\begin{equation}\label{eq:BQmatrix}
    \begin{pmatrix}
        a & \omega^\top \\
        \overline{\omega} & (2BQ(x)\overline{B}^\top)_{\hat{1}}-\frac{2}{N}\mathbf{v}_0\overline{\mathbf{v}_0}^\top-K I_{md\times md}
    \end{pmatrix}\succeq 0.
\end{equation}

Indeed, we have $a\succeq 0$ and
$a  a^\dagger \omega^\top=\omega^\top$ (It is highly non-trivial to show this fact. For a detailed argument, see Proposition \ref{prop:aadaggeromega} in the Appendix \ref{section:appendix}).
Therefore, we apply Theorem \ref{thm:Albert} again to derive that (\ref{eq:BQmatrix}) holds if and only if
\begin{equation}
    (2BQ(x)\overline{B}^\top)_{\hat{1}}-\frac{2}{N}\mathbf{v}_0\overline{\mathbf{v}_0}^\top-K I_{md\times md}-\overline{\omega}a^\dagger \omega^\top\succeq 0.
\end{equation}
That is, the $N$-Bakry--\'Emery curvature $\K_{G,\sigma,x}(N)$ is equal to the smallest eigenvalue of the matrix \[(2BQ(x)\overline{B}^\top)_{\hat{1}}-\overline{\omega} a^\dagger\omega^\top-\frac{2}{N}\mathbf{v}_0(x)\overline{\mathbf{v}_0(x)}^\top.\]

In conclusion, we introduce the following concept of curvature matrices and reformulate the curvature as an eigenvalue problem.
\begin{definition}[Curvature matrices] \label{def:curvature_matrices}Let $x$ be a vertex in a connection graph $(G,\sigma)$ with the matrix $Q(x)$ defined in (\ref{eq:Qdef}). Let $B$ be a nonsingular matrix satisfying (\ref{eq:B}). The curvature matrix $A_\infty(G,\sigma,x, B)$ associated with the vertex $x$ and the matrix $B$ is defined as an $md\times md$ matrix as follows:
\begin{equation}\label{eq:curvature_matrix}
  A_\infty(G,\sigma,x, B):= (2BQ(x)\overline{B}^\top)_{\hat{1}}-\overline{\omega} a^\dagger\omega^\top,
\end{equation}
where $a$ and $\omega$ are defined by (\ref{eq:a}) and $(\ref{eq:omega})$.
Given $N\in (0,\infty]$, we define
\begin{equation}\label{eq:curvature_matrix_N}
   A_N(G,\sigma,x, B):= A_\infty(G,\sigma,x, B)-\frac{2}{N}\mathbf{v}_0(x)\overline{\mathbf{v}_0(x)}^\top,
\end{equation}
where $\mathbf{v}_0(x)$ is given in (\ref{eq:v0}).
\end{definition}
\begin{theorem}\label{thm:curvature_eigenvalue}
Let $x$ be a vertex in a connection graph $(G,\sigma)$. Let $B$ be a nonsingular matrix satisfying (\ref{eq:B}). Then we have for any $N\in (0,\infty]$ that
\begin{equation}\label{notation:curvature}
    \K_{G,\sigma,x}(N)=\lambda_{min}(A_N(G,\sigma,x,B)).
\end{equation}
\end{theorem}

We first observe that the curvature matrices have the following property under switching the connections.
\begin{proposition}\label{prop:Switching}
Let $x$ be a vertex in a connection graph $(G,\sigma)$. Let $\tau: V\to U(d)$ be a switching function. Then the matrix $A_N(G,\sigma,x,B)$, with $B$ being a nonsingular matrix satisfying (\ref{eq:B}) and  $N\in (0,\infty]$, has the following property
 \[A_N(G,\sigma^\tau,x,\tau_{B_1(x)}^\top B\overline{\tau_{B_1(x)}})=\tau_{S_1(x)}^\top A_{N}(G,\sigma, x, B)\overline{\tau_{S_1(x)}},\]
where \[\tau_{S_1(x)}:=\begin{pmatrix}
    \tau(y_1) & & \\
              & \ddots &\\
              &   & \tau(y_m)
\end{pmatrix} \,\,\text{and}\,\,\tau_{B_1(x)}:=\begin{pmatrix}
    \tau(x)& \\
    & \tau_{S_1(x)}
\end{pmatrix}.\]
\end{proposition}
In order to show the proposition, We prepare the following lemma for the Schur complement.

\begin{lemma}\label{lemma:schur_switch}
Consider a block matrix
\[S=\begin{pmatrix}
    A & B \\
    C& D
\end{pmatrix}\]
with $D$ being an invertible square matrix. Let
\[U=\begin{pmatrix}
    U_1 & \mathbf{0} \\
    \mathbf{0} & U_2
\end{pmatrix}\]
be an invertible matrix with the blocks $U_1, U_2$ of the same sizes as $A, D$, respectively. Then the Schur complement of the bottom right block in $USU^{-1}$ equals
\[U_1(A-BD^{-1}C)U_1^{-1}.\]
\end{lemma}
\begin{proof}
One can directly check that
\[USU^{-1}=\begin{pmatrix}
    U_1AU_1^{-1} & U_1BU_2^{-1}\\
    U_2CU_1^{-1} & U_2DU_2^{-1}
\end{pmatrix}.\]
By definition, the Schur complement of the bottom right block in $USU^{-1}$ is
\[U_1AU_1^{-1}-(U_1BU_2^{-1})(U_2DU_2^{-1})^{-1}(U_2CU_1^{-1})=U_1(A-BD^{-1}C)U_1^{-1}.\]
This finishes the proof.
\end{proof}

\begin{proof}[Proof of Proposition \ref{prop:Switching}]
It is direct to check
for any $f,g: V\to \mathbb{C}^d$ that
\[\Gamma^{\sigma^\tau}(f,g)=\Gamma^\sigma(\tau f, \tau g), \,\,\text{and}\,\,\Gamma_2^{\sigma^\tau}(f,g)=\Gamma_2^\sigma(\tau f, \tau g).\]
Therefore, the matrices $\Gamma^{\sigma}(x)$ and $\Gamma^{\sigma}_{2}(x)$ change under the switching by $\tau:V\rightarrow U(d)$ as follows:
\begin{equation}\label{eq:Gamma}
    \Gamma^{\sigma^{\tau}}(x)=\tau_{B_1(x)}^\top\Gamma^{\sigma}(x)\overline{\tau_{B_1(x)}}\,\,\text{and}\,\,\Gamma^{\sigma^{\tau}}_{2}(x)=\tau_{B_2(x)}^\top\Gamma^{\sigma}_{2}(x)\overline{\tau_{B_2(x)}},
\end{equation}
where
\[\tau_{B_2(x)}=\begin{pmatrix}
    \tau_{B_1(x)}& & \\
     & \tau(z_1)& & \\
     &&\ddots&\\
     &&&\tau(z_n)
\end{pmatrix}.\]
Let us denote by $B^\tau:=\tau_{B_1(x)}^\top B\overline{\tau_{B_1(x)}}$.
 By the first identity in (\ref{eq:Gamma}), we check that
 \[B^\tau2\Gamma^{\sigma^\tau}(x)\overline{B^\tau}^\top=\begin{pmatrix}
        \mathbf{0}_{d\times d} & \mathbf{0}_{d\times md} \\
        \mathbf{0}_{md\times d} & I_{md\times md}
    \end{pmatrix}.\]
By Lemma \ref{lemma:schur_switch} and the second identity in (\ref{eq:Gamma}), we obtain
\begin{equation}\label{eq:Qswitch}
Q(\Gamma_2^{\sigma^\tau}(x))=\tau_{B_1(x)}^\top Q(\Gamma_2^\sigma(x))\overline{\tau_{B_1(x)}},
\end{equation}
which further implies that
\begin{equation}\label{eq:BQswitch}
 B^\tau 2Q(\Gamma_2^{\sigma^\tau}(x))\overline{B^\tau}^\top=\tau_{B_1(x)}^\top B2Q(\Gamma_2^\sigma(x))\overline{B}^\top\overline{\tau_{B_1(x)}}.
\end{equation}

If the local connection structure $B_2^{inc}(x)$ (that is, the induced subgraph of $B_2(x)$ with spherical edges on $S_2(x)$ deleted,) is balanced,  then the first $d$ rows and first $d$ columns of the matrices $2Q(\Gamma_2^{\sigma^\tau}(x))$ and $2Q(\Gamma_2^\sigma(x))$ vanish, and, hence, the curvature matrices satisfy
\[A_{\infty}(G,\sigma^\tau, x, B^\tau)=2Q(\Gamma_2^{\sigma^\tau}(x))_{\hat{1}}=\tau_{S_1(x)}^\top A_\infty(G,\sigma,x, B)\overline{\tau_{S_1(x)}}.\]

If, otherwise, the local connection structure $B_2^{inc}(x)$ is unbalanced, we apply Lemma \ref{lemma:schur_switch} to conclude that
\begin{equation}\label{eq:Ainfty}
A_{\infty}(G,\sigma^\tau, x, B^\tau)=\tau_{S_1(x)}^\top A_\infty(G,\sigma,x, B)\overline{\tau_{S_1(x)}}.\end{equation}

By definition, we calculate straightforwardly that
\begin{equation}\label{eq:Laplacian}
    \Delta^{\sigma^\tau}(x)=\tau_{B_1(x)}^\top\Delta^\sigma(x)\overline{\tau(x)}.
\end{equation}
This leads to
\[\Delta^{\sigma^\tau}(x)\overline{\Delta^{\sigma^\tau}(x)}^\top=\tau_{B_1(x)}^\top\Delta^{\sigma}(x)\overline{\Delta^{\sigma}(x)}^\top\overline{\tau_{B_1(x)}}.\]
We derive that
\[B^\tau\Delta^{\sigma^\tau}(x)\overline{B^\tau\Delta^{\sigma^\tau}(x)}^\top=\tau_{B_1(x)}^\top B\Delta^\sigma(x)\overline{B\Delta^\sigma(x)}^\top\overline{\tau_{B_1(x)}}.\]
Recall that the first $d$ rows and first $d$ columns of both sides vanish. Then we have
\begin{equation}\label{eq:dimension}
(B^\tau\Delta^{\sigma^\tau}(x)\overline{B^\tau\Delta^{\sigma^\tau}(x)}^\top)_{\hat{1}}=\tau_{S_1(x)}^\top (B\Delta^\sigma(x)\overline{B\Delta^\sigma(x)}^\top)_{\hat{1}}\overline{\tau_{S_1(x)}}.
\end{equation}
That is,
\begin{equation}\label{eq:dimensionv}
\mathbf{v}_0(\sigma^\tau)\overline{\mathbf{v}_0(\sigma^\tau)}^\top=\tau_{S_1(x)}^\top \mathbf{v}_0(\sigma)\overline{\mathbf{v}_0(\sigma)}^\top\overline{\tau_{S_1(x)}}.
\end{equation}
Combining (\ref{eq:Ainfty}) and (\ref{eq:dimensionv}) completes the proof.
\end{proof}

Next, we show that different choices of the matrix $B$ lead to unitarily equivalent curvature matrices.
\begin{proposition}\label{prop:unitarily_equivalent}
Let $x$ be a vertex in a connection graph $(G,\sigma)$. Let $B_1, B_2$ be two nonsingular matrices satisfying (\ref{eq:B}). Then the two curvature matrices $A_N(G,\sigma,x,B_1)$ and $A_N(G,\sigma,x,B_2)$ are unitarily equivalent. Indeed, we have
\begin{equation}
    A_N(G,\sigma,x,B_1)=(B_1B_2^{-1})_{\hat{1}}A_N(G,\sigma,x,B_2)\overline{(B_1B_2^{-1})_{\hat{1}}}^\top.
\end{equation}
\end{proposition}
We prepare some lemmas.

\begin{lemma}\label{lemma:b0}
 Let $t_0, b_0, b_1, \ldots, b_m$ be $(m+1)d\times d$ matrices such that both
\begin{equation}
    B_1:=\begin{pmatrix}
t_0^\top \\
b_1^\top\\
\vdots\\
b_m^\top
\end{pmatrix},\,\,\text{and}\,\,B_2:=\begin{pmatrix}
b_0^\top \\
b_1^\top\\
\vdots\\
b_m^\top
\end{pmatrix}
\end{equation}
are nonsingular and satisfy (\ref{eq:B}). Then we have
\begin{equation}\label{eq:Ab}
    A_N(G,\sigma,x, B_1)=A_N(G,\sigma, x, B_2).
\end{equation}
\end{lemma}
\begin{proof}
Recall that
\[A_N(G,\sigma,x, B_2)=(2B_2Q(x)\overline{B_2}^\top)_{\hat{1}}-\overline{\omega} a^\dagger\omega^\top-\frac{2}{N}(B_2\Delta^\sigma(x)\overline{B_2\Delta^\sigma(x)}^\top)_{\hat{1}}.\]
Due to (\ref{eq:BLaplacian}), (\ref{eq:BQB}), we have the two terms
\[(2B_2Q(x)\overline{B_2}^\top)_{\hat{1}}=(2B_1Q(x)\overline{B_1}^\top)_{\hat{1}}\]
and
\[(B_2\Delta^\sigma(x)\overline{B_2\Delta^\sigma(x)}^\top)_{\hat{1}}=(B_1\Delta^\sigma(x)\overline{B_1\Delta^\sigma(x)}^\top)_{\hat{1}}\]
are independent of the choice of $b_0$.

 Recall that both $\overline{b_0}$ and $\overline{t_0}$ consist of linearly independent eigenvectors corresponding to the zero eigenvalue of $2\Gamma^\sigma(x)$. There exists a nonsingular $d\times d$ matrix $E_d$ such that $b_0^\top=E_dt_0^\top$. Then, we derive
\begin{align*}\overline{\omega(x,B_2)}a(x,B_2)^{\dagger}\omega(x,B_2)^\top=\overline{\omega(x,B_1)}a(x,B_1)^{\dagger}\omega(x,B_1)^\top.
\end{align*}
This completes the proof of (\ref{eq:Ab}).
\end{proof}

\begin{lemma}\label{lemma:B1B2}
Let $B_1,B_2$ be two nonsingular matrices satisfying (\ref{eq:B}). Then the $md\times md$ matrix $(B_1B_2^{-1})_{\hat{1}}$ is unitary.
\end{lemma}
\begin{proof}
By direct calculation, we have
\begin{align*}
\begin{pmatrix}
    \mathbf{0}_d & \mathbf{0}_{d\times md} \\
    \mathbf{0}_{md\times d} & I_{md}
\end{pmatrix}=&B_12\Gamma^\sigma(x)\overline{B_1}^\top=(B_1B_2^{-1})B_22\Gamma^\sigma(x)\overline{B_2}^\top\overline{B_1B_2^{-1}}^\top\\
=&(B_1B_2^{-1})\begin{pmatrix}
    \mathbf{0}_d & \mathbf{0}_{d\times md} \\
    \mathbf{0}_{md\times d} & I_{md}
\end{pmatrix}\overline{B_1B_2^{-1}}^\top.
\end{align*}
Therefore, we have $(B_1B_2^{-1})_{\hat{1}}\overline{(B_1B_2^{-1})_{\hat{1}}}^\top=I_{md}$.
\end{proof}

\begin{lemma}\label{lemma:B1B2inverse}
For $i=1,2$, let
\[B_i=\begin{pmatrix}
    b_{0,i}^\top \\
    b_{1,i}^\top \\
    \vdots \\
    b_{m,i}^\top
\end{pmatrix}\]
be a nonsingular matrix satisfying (\ref{eq:B}). Then we have
\begin{equation}
    B_1B_2^{-1}=\begin{pmatrix}
        E_d & \mathbf{0}_{d\times md} \\
        \eta & (B_1B_2^{-1})_{\hat{1}}
    \end{pmatrix},
\end{equation}
for some $md\times d$ matrix $\eta$, where $E_d$ is the unique $d\times d$ matrix with $b_{0,1}^\top=E_d b_{0,2}^\top$.
\end{lemma}
\begin{proof}
By (\ref{eq:B}), we derive
\begin{equation}\label{eq:B1GammaB2}
    B_1B_2^{-1}\begin{pmatrix}
        \mathbf{0}_d & \mathbf{0}_{d\times md} \\
        \mathbf{0}_{md\times d} & I_{md}
    \end{pmatrix}=B_12\Gamma^\sigma(x)\overline{B_2}^\top
    =\begin{pmatrix}
        \mathbf{0}_d & \mathbf{0}_{d\times md} \\
        \mathbf{0}_{md\times d} & \begin{pmatrix}
        b_{1,1}^\top \\ \vdots \\ b_{m,1}^\top
        \end{pmatrix}2\Gamma^\sigma(x)\begin{pmatrix}
            \overline{b_{1,2}} & \cdots & \overline{b_{m,2}}
        \end{pmatrix}
    \end{pmatrix}.
\end{equation}
In the second equality above, we use the fact $2\Gamma^\sigma(x)\overline{b_{0,i}}=0, \,i=1,2.$
Let us write
\[B_1B_2^{-1}=\begin{pmatrix}
    *_1 & *_2 \\
    \eta & (B_1B_2^{-1})_{\hat{1}}
\end{pmatrix}.\]
Then we have
\[B_1B_2^{-1}\begin{pmatrix}
        \mathbf{0}_d & \mathbf{0}_{d\times md} \\
        \mathbf{0}_{md\times d} & I_{md}
    \end{pmatrix}=\begin{pmatrix}
        \mathbf{0}_d & *_2 \\
        \mathbf{0}_{md\times d} & (B_1B_2^{-1})_{\hat{1}}
    \end{pmatrix}.\] Hence, we solve from (\ref{eq:B1GammaB2}) that $*_2=\mathbf{0}$.
    Furthermore, we solve from
    \[B_1=\begin{pmatrix}
    *_1 & \mathbf{0}_{d\times md} \\
    \eta & (B_1B_2^{-1})_{\hat{1}}
\end{pmatrix}B_2\]
that $b_{0,1}^\top=*_1 b_{0,2}^\top$. Therefore, we have $*_1=E_d$.
\end{proof}

\begin{corollary}
\label{coro:transitive}
Let $B_1, B_2, B_3$ be three nonsingular matrices satisfying (\ref{eq:B}). Then, we have
\begin{equation}
    (B_1B_2^{-1})_{\hat{1}}(B_2B_3^{-1})_{\hat{1}}=(B_1B_3^{-1})_{\hat{1}}.
\end{equation}
\end{corollary}
\begin{proof}
This follows directly from Lemma \ref{lemma:B1B2inverse} and the identity below:
\[\begin{pmatrix}
    E_d & \mathbf{0}_{d\times md} \\
    \eta & C
\end{pmatrix}\begin{pmatrix}
    E'_d & \mathbf{0}_{d\times md} \\
    \eta' & C'
\end{pmatrix}=\begin{pmatrix}
    \star & \mathbf{0}_{d\times md} \\
    * & CC'
\end{pmatrix}.\]
\end{proof}


Now, we are prepared for the proof of Proposition \ref{prop:unitarily_equivalent}.
\begin{proof}[Proof of Proposition \ref{prop:unitarily_equivalent}]
Let us denote by $C:=B_1B_2^{-1}$. By Lemma \ref{lemma:B1B2inverse}, we have
$$C=\begin{pmatrix}
     I_{d}          & \mathbf{0}_{d\times md}  \\
      \eta &                       C_{\hat{1}}
\end{pmatrix},$$
where $\eta$ is a $md\times d$ matrix and $C_{\hat{1}}$ is the matrix obtained from $C$ by deleting the first $d$ rows and first $d$ columns.

Now we compare the two curvature matrices \[A_{\infty}(G,\sigma,x,B_i):=(B_i2Q(x)\overline{B_i}^{\top})_{\hat{1}}-\overline{\omega}_ia_i^\dagger\omega_i^{\top},\,\,i=1,2,\]
where $a_i$ and $\omega_i$ are defined by (\ref{eq:a}) and (\ref{eq:omega}) using $B_i$, for $i=1,2$.

By Lemma \ref{lemma:b0}, we can assume $a_1=a_2$ without loss of generality. Let us write $a:=a_1=a_2$ for short.

We start with
\begin{align*}
B_12Q(x)\overline{B_1}^{\top}=&CB_22Q(x)\overline{B_2}^{\top}\overline{C}^{\top}\\
=&\begin{pmatrix}
     I_{d}          & \mathbf{0}_{d\times md}  \\
      \eta &                       C_{\hat{1}}
\end{pmatrix}
\begin{pmatrix}
    a                       &  \omega_{2}^{\top}  \\
    \overline{\omega_{2}}   &  (B_22Q(x)\overline{B_2}^{\top})_{\hat{1}}
\end{pmatrix}
\begin{pmatrix}
    I_{d}                     &  \overline{\eta}^{\top}       \\
    \mathbf{0}_{md\times d}  &  \overline{C_{\hat{1}}}^{\top}
\end{pmatrix}\\
=&\begin{pmatrix}
      a   &   a\overline{\eta}^{\top}+\omega_{2}^{\top}\overline{C_{\hat{1}}}^{\top}  \\
      \eta a+C_{\hat{1}}\overline{\omega_{2}} &  \eta a \overline{\eta}^{\top}+C_{\hat{1}}\overline{\omega_{2}}\overline{\eta}^{\top}+\eta\omega_2^{\top}\overline{C_{\hat{1}}}^{\top}+C_{\hat{1}}(B_22Q(x)\overline{B_2}^{\top})_{\hat{1}}\overline{C_{\hat{1}}}^{\top}
\end{pmatrix}.
\end{align*}
Suppose the local connection structure $B_2^{inc}(x)$ is unbalanced. It is direct to check that the Schur complement of $a$ in $B_{1}2Q(x)\overline{B}^\top_{1}$ satisfies
\begin{align*}
(B_12Q(x)\overline{B_1}^{\top})_{\hat{1}}-\overline{{\omega}_1}a^{\dagger}\omega_1^\top
=C_{\hat{1}}\left[(B_22Q(x)\overline{B_2}^{\top})_{\hat{1}}-\overline{\omega_{2}}a^{\dagger}\omega_{2}^{\top}\right]\overline{C_{\hat{1}}}^{\top}
\end{align*}
That is, we have 
\[A_{\infty}(G,\sigma,x,B_1)=C_{\hat{1}}A_{\infty}(G,\sigma,x,B_{2})\overline{C_{\hat{1}}}^{\top}.\]
For the case that the local connection structure $B_2^{inc}(x)$ is balanced, we have $a=\mathbf{0}_d$ and $\omega_1=\omega_2=\mathbf{0}_{md\times d}$. Hence, $A_{\infty}(G,\sigma,x,B_1)=C_{\hat{1}}A_{\infty}(G,\sigma,x,B_2)\overline{C_{\hat{1}}}^{\top}$ still holds true.

Moreover, we derive that
\begin{align*}
\begin{pmatrix}
  \mathbf{0}_d\\\mathbf{v}_0(B_1)
\end{pmatrix}=B_1\Delta^\sigma(x)=CB_2\Delta^\sigma(x)=\begin{pmatrix}
  I_d & \mathbf{0}_{d\times md}\\
  \eta & C_{\hat{1}}
\end{pmatrix}\begin{pmatrix}
  \mathbf{0}_d \\\mathbf{v}_0(B_2)
\end{pmatrix}=\begin{pmatrix}
  \mathbf{0}_d \\ C_{\hat{1}}\mathbf{v}_0(B_2)
\end{pmatrix}.
\end{align*}
Then, we obtain $\mathbf{v}_0(B_1)=C_{\hat{1}}\mathbf{v}_0(B_2)$, and, hence
\begin{align*}
\mathbf{v}_0(B_1)\overline{\mathbf{v}_0(B_1)}^\top=C_{\hat{1}}\mathbf{v}_0(B_2)\overline{\mathbf{v}_0(B_2)}^\top\overline{C_{\hat{1}}}^\top.
\end{align*}
This completes the proof.
\end{proof}

\begin{remark}\label{remark:B0}
We can choose the following matrix $B_0$ as a canonical choice in calculating the curvature matrices:
\begin{equation}
B_0:=\begin{pmatrix}
  p_0^\top \\ p_1^{\top}\\\vdots \\ p_m^{\top}
\end{pmatrix}=
\begin{pmatrix}\label{eq:canonicalB}
    I_{d} & \overline{\sigma_{xy_1}}                            & \cdots  & \overline{\sigma_{xy_m}} \\
          & \frac{1}{\sqrt{p_{xy_{1}}}}I_{d} &        &       \\
          &                                  & \ddots &       \\
          &                                  &        & \frac{1}{\sqrt{p_{xy_{m}}}}I_{d}
\end{pmatrix}.
\end{equation}
It is straightforward to verify that $B_0$ is a nonsingular matrix that meets condition (\ref{eq:B}). In addition, we have that
\begin{equation}\label{eq:canonicalv0}
\mathbf{v}_0(B_0)=\begin{pmatrix}
  \sqrt{p_{xy_1}}\sigma_{xy_1}^\top \\\vdots\\ \sqrt{p_{xy_m}}\sigma_{xy_m}^\top
\end{pmatrix}.
\end{equation}
\end{remark}
\begin{remark}
Indeed, the matrix $a$ vanishes if the local connection structure $B_2^{inc}(x)$ is balanced (see Proposition \ref{prop:aadaggeromega} in the Appendix \ref{section:appendix}). Whenever each edge is assigned with the connection $1\in O(1)$, our curvature matrix $A_{N}(G,\sigma,x, B_0)$ coincides with the curvature matrix for Bakry--\'Emery curvature of graph Laplacians introduced in \cite{CKLP22} and \cite{Siconolfi}.
\end{remark}
To conclude this section, we present an example that demonstrates how to determine the curvature using the approach discussed above. 
\begin{example}
Consider an unbalanced $U(2)$-connection graph $(G_1,\sigma_1)$ depicted in Figure \ref{fig:local structure G1}, where we have $(\sigma_1)_{23}=
\begin{pmatrix}
    0  & i \\
    -i & 0
\end{pmatrix}$, and $p_{ij}=1$ whenever $i\sim j$ and $0$ otherwise.
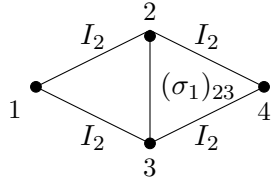
\begin{figure}[!htp]
\centering
\tikzset{vertex/.style={circle, draw, fill=black!20, inner sep=0pt, minimum width=4pt}}
\begin{tikzpicture}[scale=3.0]

\draw (0,0) -- (0.5,0.25) node[midway, above, black]{$I_2$}
		 -- (1,0) node[midway, above, black]{$I_2$};
\draw (0,0) -- (0.5,-0.25) node[midway, below, black]{$I_2$}
		 -- (1,0) node[midway, below, black]{$I_2$};
\draw (0.5,0.25) -- (0.5,-0.25) node[midway, right, black]{$(\sigma_1)_{23}$};

\node at (0,0) [vertex, label={[label distance=0mm]225: \small $1$}, fill=black] {};
\node at (0.5,0.27) [vertex, below, label={[label distance=0mm]90: \small $2$}, fill=black] {};
\node at (0.5,-0.25) [vertex, label={[label distance=0mm]270: \small $3$}, fill=black] {};
\node at (1,0) [vertex, label={[label distance=0mm]270: \small $4$}, fill=black] {};
\end{tikzpicture}
\caption{An unbalanced connection graph $(G_{1},\sigma_1)$.}
\label{fig:local structure G1}
\end{figure}
\par  Let us derive the curvature matrix $A_\infty(G_1,\sigma_1,1,B_0)$ using Definition \ref{def:curvature_matrices} and (\ref{eq:canonicalB}). By \eqref{eq:GammaMatrix} and \eqref{eq:gamma 2 matrix}-\eqref{eq:Gamma_2_last}, we obtain the matrices 
$$2\Gamma^{\sigma_{1}}(1)=\begin{pmatrix}
                          2 & 0 & -1 & 0 & -1 & 0 \\
                          0 & 2 & 0 & -1 & 0 & -1 \\
                          -1 & 0 & 1 & 0 & 0 & 0 \\
                          0 & -1 & 0 & 1 & 0 & 0\\
                          -1 & 0 & 0 & 0 & 1 & 0\\
                          0 & -1 & 0 & 0 & 0 & 1
\end{pmatrix},$$
and
$$2\Gamma^{\sigma_{1}}_{2}(1)=\frac{1}{2}\begin{pmatrix}
                              10 & 0 & -7 & -i & -7 & -i & 2 & 0 \\
                              0 & 10 & i & -7 & i  & -7 & 0 & 2 \\
                              -7 & -i & 10 & 0 & 2 & 4i & -2 & 0 \\
                              i & -7 & 0 & 10 & -4i & 2 & 0 & -2 \\
                              -7 & -i & 2 & 4i & 10 & 0 & -2 & 0 \\
                              i & -7 & -4i & 2 & 0 & 10 & 0 & -2 \\
                              2 & 0 & -2 & 0 & -2 & 0 & 2 & 0\\
                              0 & 2 & 0 & -2 & 0 & -2 & 0 & 2
\end{pmatrix}.$$
Then, the Schur complement $2Q(1)$ of  $2(\Gamma^{\sigma_{1}}_{2}(1))_{S_{2}(1),S_{2}(1)}$ in $2\Gamma^{\sigma_{1}}_{2}(1)$ is
$$2Q(1)=\frac{1}{2}\begin{pmatrix}
           8 & 0 & -5 & -i & -5 & -i \\
           0 & 8 & i & -5 & i & -5 \\
           -5 & -i & 8 & 0 & 0 & 4i \\
           i & -5 & 0 & 8 & -4i & 0 \\
           -5 & -i & 0 & 4i & 8 & 0 \\
           i & -5 & -4i & 0 & 0 & 8
\end{pmatrix}.$$
With the canonical choice of
$$B_{0}=\begin{pmatrix}
        1 & 0 & 1 & 0 & 1 & 0 \\
        0 & 1 & 0 & 1 & 0 & 1 \\
        0 & 0 & 1 & 0 & 0 & 0 \\
        0 & 0 & 0 & 1 & 0 & 0 \\
        0 & 0 & 0 & 0 & 1 & 0 \\
        0 & 0 & 0 & 0 & 0 & 1
\end{pmatrix},$$
we deduce the matrix
$$B_{0}2Q(1)\overline{B_{0}}^\top=\frac{1}{2}
\begin{pmatrix}
    4 & 4i & 3 & 3i & 3 & 3i \\
    -4i & 4 & -3i & 3 & -3i & 3 \\
    3 & 3i & 8 & 0 & 0 & 4i \\
    -3i & 3 & 0 & 8 & -4i & 0 \\
    3 & 3i & 0 & 4i & 8 & 0 \\
    -3i & 3 & -4i & 0 & 0 & 8
\end{pmatrix}.$$
Then, the curvature matrix is the Schur complement of $\frac{1}{2}\begin{pmatrix}
    4 & 4i \\
    -4i &4
\end{pmatrix}$ in $B_02Q(1)\overline{B_0}^\top$.  That is, we have
$$A_{\infty}(G_{1},\sigma_{1},1,B_{0})=\frac{1}{8}\begin{pmatrix}
    23 & -9i & -9 & 7i \\
    9i & 23 & -7i  & -9 \\
    -9 & 7i & 23 & -9i \\
    -7i & -9 & 9i & 23
\end{pmatrix}.$$
The eigenvalues of $A_{\infty}(G_{1},\sigma_{1},1,B_{0})$ are $\{\frac{3}{2},2,2,6\}$. Therefore, we obtain the curvature
\[\K_{G_1,\sigma_1,1}(\infty)=\frac{3}{2}.\]
\end{example}

\section{Curvature tensors of connection graphs}\label{section:tensor}
In order to motivate the definition of discrete Bakry--\'Emery curvature tensors of connection graphs, we first present an alternative proof of Theorem \ref{thm:curvature_eigenvalue}. This alternative proof aligns more closely with the work by Siconolfi \cite{Siconolfi}.
\begin{lemma}\label{lemma:GammaRayleigh}
Let $x$ be a vertex in a connection graph $(G,\sigma)$. Then we have for any $N\in (0,\infty]$ that
\begin{equation}
  \K_{G,\sigma,x}(N)=\inf_{\substack{f:V\to \mathbb{K}^d\\\text{with}\,\,\Gamma^\sigma(f)(x)\neq 0}}\frac{\Gamma_2^\sigma(f)(x)-(1/N)|\Delta^\sigma f(x)|^2}{\Gamma^\sigma(f)(x)}.
\end{equation}
\end{lemma}
\begin{proof}
It is enough to show for any $f:V\to \mathbb{K}^d$ with $\Gamma^\sigma(f)(x)=0$ that \[\Gamma_2^\sigma(f)(x)-(1/N)|\Delta^\sigma f(x)|^2\geq 0.\]
By Remark \ref{remark:B0}, it holds for any $f$ with $\Gamma^\sigma(f)(x)=0$ that $f(y_i)=\sigma_{xy_i}^{-1}f(x)$ for any $i=1,\ldots,m$. Then we have
$\Delta^\sigma f (x)=\mathbf{0}$. Next, we show $\Gamma_2^\sigma(f)(x)\geq  0$ for such an $f$.

Let us denote by
\begin{equation*}
b_{01}^\top:=\begin{pmatrix}
  f(x)^\top & f(x)^\top\overline{\sigma_{xy_1}} & \cdots & f(x)^\top\overline{\sigma_{xy_m}}
\end{pmatrix}.
\end{equation*}
By the identity (\ref{eq:Schur_sesquilinear_form}), we estimate
\begin{align*}
2\Gamma_2^\sigma(f)(x)\geq b_{01}^\top 2Q(x)\overline{b_{01}}.
\end{align*}
If $f(x)=\mathbf{0}$, then we have $2\Gamma_2^\sigma(f)(x)\geq 0$. Otherwise, we have $f(x)\neq \mathbf{0}$. Let $\overline{b_{01}},\ldots,\overline{b_{0d}}$ be a basis for the eigenspace of $\Gamma^\sigma(x)$ to its zero eigenvalue. This basis provides
\[b_0^\top=\begin{pmatrix}
  b_{01}^\top \\
  \vdots \\
  b_{0d}^\top
\end{pmatrix}.\]
Recall that we have the matrix $a=b_0^\top2Q(x)\overline{b_0}\succeq 0$ (see Proposition \ref{prop:aadaggeromega} in Appendix \ref{section:appendix}). Observe that $b_{01}^\top 2Q(x)\overline{b_{01}}$ is a diagonal entry of the matrix $a$, and hence, is nonnegative. Therefore, $\Gamma_2^\sigma(f)(x)\geq 0$.
\end{proof}
\begin{proof}[An alternative proof of Theorem \ref{thm:curvature_eigenvalue}]
For a given function $f: V\to \mathbb{K}^d$, we denote by
\begin{equation}\label{eq:f0f1f2}
f_0:=f(x),\,\,f_1:=\begin{pmatrix}
  f(y_1) \\ \vdots \\ f(y_m)
\end{pmatrix}, \,\,\text{and}\,\,f_2:=\begin{pmatrix}
  f(z_1) \\ \vdots \\ f(z_n)
\end{pmatrix}.\end{equation}
Applying Lemma \ref{lemma:GammaRayleigh}, we derive
\begin{align}
&\K_{G,\sigma,x}(N)\notag\\
=&\inf_{\substack{f:V\to \mathbb{K}^d\\\text{with}\,\,\Gamma^\sigma(f)(x)\neq 0}}\frac{\begin{pmatrix}
f_0^\top & f_1^\top & f_2^\top
\end{pmatrix}\left(2\Gamma^\sigma_2(x)-\frac{2}{N}\Delta^\sigma(x)\overline{\Delta^\sigma(x)}^\top\right)\begin{pmatrix}
\overline{f_0} \\ \overline{f_1} \\ \overline{f_2}
\end{pmatrix}}{\begin{pmatrix}
f_0^\top & f_1^\top
\end{pmatrix}2\Gamma^\sigma(x)\begin{pmatrix}
\overline{f_0} \\ \overline{f_1}
\end{pmatrix}}\notag\\
=& \inf_{\substack{f_0,f_1\\\text{with}\,\,\Gamma^\sigma(f)(x)\neq 0}}\inf_{f_2}\frac{\begin{pmatrix}
f_0^\top & f_1^\top
\end{pmatrix}\left(2Q(x)-\frac{2}{N}\Delta^\sigma(x)\overline{\Delta^\sigma(x)}^\top\right)\begin{pmatrix}
\overline{f_0} \\ \overline{f_1}
\end{pmatrix}+\left|T(f_0,f_1,f_2)\right|^2}{\begin{pmatrix}
f_0^\top & f_1^\top
\end{pmatrix}2\Gamma^\sigma(x)\begin{pmatrix}
\overline{f_0} \\ \overline{f_1}
\end{pmatrix}},\label{eq:K_Rayleigh_1}
\end{align}
where \[T(f_0,f_1,f_2)=\left((2\Gamma_2^\sigma(x)_{S_2,S_2})^{1/2}\right)^{-1}(2\Gamma_2^\sigma(x)_{S_2,B_1})\begin{pmatrix}
  \overline{f_0} \\ \overline{f_1}
\end{pmatrix}+\overline{(2\Gamma_2^\sigma(x)_{S_2,S_2})^{1/2}}^\top\overline{f_2}\in \mathbb{K}^{nd}\] is the vector determined by (\ref{eq:Schur_sesquilinear_form}). Applying Lemma \ref{lemma:schur_technical} yields that
\begin{equation}\label{eq:inf_T}
\inf_{f_2}\left|T(f_0,f_1,f_2)\right|^2=0.
\end{equation}

We continue the calculation in (\ref{eq:K_Rayleigh_1}) to derive
\begin{align}
&\K_{G,\sigma,x}(N)\notag\\
=& \inf_{\substack{f_0,f_1\\\text{with}\,\,\Gamma^\sigma(f)(x)\neq 0}}\frac{\begin{pmatrix}
f_0^\top & f_1^\top
\end{pmatrix}\left(2Q(x)-\frac{2}{N}\Delta^\sigma(x)\overline{\Delta^\sigma(x)}^\top\right)\begin{pmatrix}
\overline{f_0} \\ \overline{f_1}
\end{pmatrix}}{\begin{pmatrix}
f_0^\top & f_1^\top
\end{pmatrix}2\Gamma^\sigma(x)\begin{pmatrix}
\overline{f_0} \\ \overline{f_1}
\end{pmatrix}}\notag\\
=& \inf_{\substack{f_0,f_1\\\text{with}\,\,\Gamma^\sigma(f)(x)\neq 0}}\frac{\begin{pmatrix}
f_0^\top & f_1^\top
\end{pmatrix}B^{-1}B\left(2Q(x)-\frac{2}{N}\Delta^\sigma(x)\overline{\Delta^\sigma(x)}^\top\right)\overline{B}^\top\overline{B^{-1}}^\top\begin{pmatrix}
\overline{f_0} \\ \overline{f_1}
\end{pmatrix}}{\begin{pmatrix}
f_0^\top & f_1^\top
\end{pmatrix}B^{-1}B2\Gamma^\sigma(x)\overline{B}^\top\overline{B^{-1}}^\top\begin{pmatrix}
\overline{f_0} \\ \overline{f_1}
\end{pmatrix}},\label{eq:K_Rayleigh_2}
\end{align}
where $B$ is the nonsingular matrix satisfying (\ref{eq:B}). We denote by
\begin{equation}\label{eq:fB}\begin{pmatrix}
f_{B,0}\\f_{B,1}
\end{pmatrix}:=\left(B^{-1}\right)^\top\begin{pmatrix}
f_0 \\f_1
\end{pmatrix}.\end{equation}
Then we derive from (\ref{eq:K_Rayleigh_2}) that
\begin{align}
&\K_{G,\sigma,x}(N)\notag\\
=&\inf_{\substack{f_0,f_1\\\text{with}\,\,f_{B,1}\neq 0}}\frac{\begin{pmatrix}
f_{B,0}^\top & f_{B,1}^\top
\end{pmatrix}B\left(2Q(x)-\frac{2}{N}\Delta^\sigma(x)\overline{\Delta^\sigma(x)}^\top\right)\overline{B}^\top\begin{pmatrix}
\overline{f_{B,0}} \\ \overline{f_{B,1}}
\end{pmatrix}}{\left|f_{B,1}\right|^2},\notag\\
=&\inf_{f_{B,1}\neq 0}\inf_{f_{B,0}}\frac{f_{B,1}^\top A_N(G,\sigma,x,B)\overline{f_{B,1}}+|T_B(f_{B,0},f_{B,1})|^2}{\left|f_{B,1}\right|^2},\label{eq:4.6}
\end{align}
where $T_B(f_{B,0},f_{B,1})=\overline{a^{1/2}}^\top\overline{f_{B,0}}+(a^{1/2})^\dagger\omega^\top\overline{f_{B,1}}\in\mathbb{K}^d$ is the vector determined by the identity (\ref{eq:Schur_sesquilinear_form}). Applying Lemma \ref{lemma:schur_technical} yields that
\begin{equation}\label{eq:inf_TB}\inf_{f_{B,0}}|T_B(f_{B,0},f_{B,1})|^2=0.\end{equation}

Now we continue the calculation in (\ref{eq:4.6}) as follows
\begin{align}
&\K_{G,\sigma,x}(N)=\inf_{f_{B,1}\neq 0}\frac{f_{B,1}^\top A_N(G,\sigma,x,B)\overline{f_{B,1}}}{\left|f_{B,1}\right|^2}=\lambda_{\min}(A_N(G,\sigma,x,B)).
\end{align}
This completes the proof.
\end{proof}

We have the following key observations about the vectors achieving the infimums in (\ref{eq:inf_T}) and (\ref{eq:inf_TB}):

\textbf{Observation 1}: By Lemma \ref{lemma:schur_technical}, the infimum in (\ref{eq:inf_T}) is attained by $f_2$ satisfying
\[2\Gamma^\sigma_2(x)\begin{pmatrix}
  \overline{f_0} \\ \overline{f_1} \\ \overline{f_2}
\end{pmatrix}=\begin{pmatrix}
  *_0 \\ *_1 \\ \mathbf{0}
\end{pmatrix}.\]
Explicitly, the infimum is attained by
\begin{align}\label{eq:f2=f0f1}
f_2=&-(2\Gamma_2^\sigma(x)_{S_2,S_2})^{-1}\overline{2\Gamma_2^\sigma(x)_{S_2,B_1}}\begin{pmatrix}
  f_0 \\ f_1
\end{pmatrix}\notag\\
=&\begin{pmatrix}
 -\frac{\sum_yp_{xy}p_{yz_1}\overline{\sigma_{yz_1}}^\top\overline{\sigma_{xy}}^\top}{\sum_{y}p_{xy}p_{yz_1}} &   \frac{2p_{xy_1}p_{y_1z_1}\overline{\sigma_{y_1z_1}}^\top}{\sum_{y}p_{xy}p_{yz_1}} & \cdots &  \frac{2p_{xy_m}p_{y_mz_1}\overline{\sigma_{y_mz_1}}^\top}{\sum_{y}p_{xy}p_{yz_1}}\\
  \vdots & \vdots & \ddots & \vdots\\
  -\frac{\sum_yp_{xy}p_{yz_n}\overline{\sigma_{yz_n}}^\top\overline{\sigma_{xy}}^\top}{\sum_{y}p_{xy}p_{yz_n}} &  \frac{2p_{xy_1}p_{y_1z_n}\overline{\sigma_{y_1z_n}}^\top}{\sum_{y}p_{xy}p_{yz_n}} & \cdots &  \frac{2p_{xy_m}p_{y_mz_n}\overline{\sigma_{y_mz_n}}^\top}{\sum_{y}p_{xy}p_{yz_n}}
\end{pmatrix}\begin{pmatrix}
  f_0 \\ f_1
\end{pmatrix}.
\end{align}
Notice that, in the above we make use of the fact that $\Gamma_2^\sigma(x)_{S_2,S_2}$ is a real matrix. Indeed, it is a positive diagonal matrix.

\textbf{Observation 2}: The infimum in (\ref{eq:inf_TB}) is attained by a vector $f_{B,0}$ satisfying
\begin{equation}\label{eq:fB0=fB1}B\left(2Q(x)-\frac{2}{N}\Delta^\sigma(x)\overline{\Delta^\sigma(x)}^\top\right)\overline{B}^\top\begin{pmatrix}
  \overline{f_{B,0}}\\
  \overline{f_{B,1}}
\end{pmatrix}=\begin{pmatrix}
  \mathbf{0} \\ *
\end{pmatrix}.\end{equation}
Let us denote the submatrix of $B$ composed of the first $d$ rows by $b_0^\top$. Then combining (\ref{eq:fB}) and (\ref{eq:fB0=fB1}) leads to
\begin{equation}\label{eq:f_0=f_1}
  \mathbf{0}=b_0^\top\left(2Q(x)-\frac{2}{N}\Delta^\sigma(x)\overline{\Delta^\sigma(x)}^\top\right)
  \begin{pmatrix}
    \overline{f_0}\\ \overline{f_1}
  \end{pmatrix}=b_0^\top(2Q(x)) \begin{pmatrix}
    \overline{f_0}\\ \overline{f_1}
  \end{pmatrix}.
\end{equation}
Recall that there exists a nonsingular $d\times d$ matrix $E_d$ such that
\[b_0^\top=E_d p_0^\top:=E_d\begin{pmatrix}
  I_d & \overline{\sigma_{xy_1}} & \cdots &\overline{\sigma_{xy_m}}
\end{pmatrix}.\]
Then we derive from (\ref{eq:f_0=f_1}) that
\begin{equation}\label{eq:f_0_formula_f_1}
p_0^\top2Q(x)\begin{pmatrix}
  \overline{f_0} \\ \overline{f_1}
\end{pmatrix}=\mathbf{0}.
\end{equation}
We point out that the matrix $p_0^\top2Q(x)$ is a zero matrix when the local connection structure $B_2^{inc}(x)$ is balanced (see (\ref{eq:B_0Q(x)}) and (\ref{eq:B_0Q(x)i}) in Appendix \ref{section:appendix}). In that case, the equation (\ref{eq:f_0_formula_f_1}) holds true for any $f_0,f_1$.
%

Based on the above two observations, we introduce the following linear maps.
\begin{definition}\label{def:PhiPsi}
Let $x$ be a given vertex in a connection graph $(G,\sigma)$.
We define a linear map $\Psi_x:\mathbb{K}^{(m+1)d}\to \mathbb{K}^{(m+n+1)d}$ via
\begin{equation}
 \Psi_x(w)=\begin{pmatrix}
  w \\
-(2\Gamma_2^\sigma(x)_{S_2,S_2})^{-1}\overline{2\Gamma_2^\sigma(x)_{S_2,B_1}}w
\end{pmatrix}
\end{equation}
for any $w\in \mathbb{K}^{(m+1)d}$.
\end{definition}
\begin{definition}\label{def:tangent_space}
Let $x$ be a given vertex in a connection graph $(G,\sigma)$. Any two functions $f, f': B_1(x)\to \mathbb{K}^d$ are defined to be equivalent if
\[\begin{pmatrix}
             \sigma_{xy_{1}}f(y_{1})-f(x) \\
             \vdots \\
             \sigma_{xy_{m}}f(y_{m})-f(x)
\end{pmatrix}=\begin{pmatrix}
             \sigma_{xy_{1}}f'(y_{1})-f'(x) \\
             \vdots \\
             \sigma_{xy_{m}}f'(y_{m})-f'(x)
\end{pmatrix}.\]
We denote by $[f]$ the equivalent class of the function $f: B_1(x)\to \mathbb{K}^d$. We define the following linear operations: For any equivalent classes $[f], [f']$ and $\lambda\in \mathbb{K}$,
\[[f]+[f']:=[f+f'],\,\,\lambda[f]:=[\lambda f].\]
We define the \emph{tangent space} $T_x(G,\sigma)$ at the vertex $x$ as the linear space of the equivalent classes $[f]$ of any function $f: B_1(x)\to \mathbb{K}^d$.
\end{definition}
Notice that $T_x(G,\sigma)$ is isomorphic to $\mathbb{K}^{md}$ as a linear space. Indeed, the isomorphic map is given by
\begin{equation*}
\Pi: T_x(G,\sigma)\to \mathbb{K}^{md},
\end{equation*}
such that for any $f:B_1(x)\to \mathbb{K}^d$
\begin{equation*}
  \Pi([f]):=\begin{pmatrix}
    \sigma_{xy_1}f(y_1)-f(x)\\
    \vdots\\
    \sigma_{xy_m}f(y_m)-f(x)
  \end{pmatrix}.
\end{equation*}
In the following, we identify each tangent space $T_x(G,\sigma)$ with $\Pi(T_x(G,\sigma))=\mathbb{K}^{md}$.
\begin{proposition}\label{prop:Phi}
Let $x$ be a given vertex in a connection graph $(G,\sigma)$. Then there exists a linear map
\[\phi_x: T_x(G,\sigma)\to \mathbb{K}^{d},\]
such that
\begin{equation}\label{eq:phiv}
  p_0^\top 2Q(x)\begin{pmatrix}
    \overline{\phi_x(v)}\\
    \sigma_{xy_1}^{\top}(\overline{v_1+\phi_x(v)})\\
    \vdots \\
    \sigma_{xy_m}^{\top}(\overline{v_m+\phi_x(v)})
  \end{pmatrix}=\mathbf{0}
\end{equation}
holds true for any $v=\begin{pmatrix}
v_1\\ \vdots \\ v_m
\end{pmatrix}\in T_x(G,\sigma)$ with each $v_i\in \mathbb{K}^d$.
%
\end{proposition}
\begin{proof}
We derive from (\ref{eq:phiv}) that
\begin{align}\label{eq:p_02Q(x)}
\mathbf{0}=p_0^\top 2Q(x)\begin{pmatrix}
  \mathbf{0} \\
  \sigma_{xy_1}^{\top}\overline{v_1}\\
  \vdots\\
  \sigma_{xy_m}^\top\overline{v_m}
\end{pmatrix}+p_0^\top 2Q(x)\overline{p_0}\overline{\phi_x(v)}.
\end{align}
By Proposition \ref{prop:aadaggeromega}, we have $a:=p_0^\top 2Q(x)\overline{p_0}\succeq 0$. Hence, there exists a nonsingular matrix $P$ such that
\[a=P\diag (1,\ldots,1,0,\ldots,0) \overline{P}^\top.\]
Then, we have
\begin{align}\label{eq:Pinverse_a}
  \diag (1,\ldots,1,0,\ldots,0) \overline{P}^\top\overline{\phi_x(v)}=-P^{-1}p_0^\top 2Q(x)\begin{pmatrix}
  \mathbf{0} \\
  \sigma_{xy_1}^{\top}\overline{v_1}\\
  \vdots\\
  \sigma_{xy_m}^\top\overline{v_m}
\end{pmatrix}.
\end{align}
By (\ref{eq:Appendix_a}), there exists $d\times d$ matrices $X_{ij,k}$ and $X_{ij}$, $i,j=1,\ldots,m,\,k=1,\ldots,n$ such that
\[a=\sum_{i,j,k}X_{ij,k}\overline{X_{ij,k}}^\top+\sum_{i,j}X_{ij}\overline{X_{ij}}^\top.\]
Observing from (\ref{eq:B_0Q(x)}) and (\ref{eq:B_0Q(x)i}) that there exists $d\times d$ matrices $Y_{ij,k}, Y'_{ij,k}$ and $Y_{ij}, Y_{ij}'$ such that
\[p_0^\top 2Q(x)=\begin{pmatrix}
(p_0^\top 2Q(x))_0 & (p_0^\top 2Q(x))_1 & \cdots & (p_0^\top 2Q(x))_m
\end{pmatrix}\]
with
\[(p_0^\top 2Q(x))_i=\sum_{j,k}X_{ij,k}Y_{ij,k}+\sum_j X_{ij}Y_{ij},\,\,i=1,\ldots,m\]
and
\[(p_0^\top 2Q(x))_0=\sum_{i,j,k}X_{ij,k}Y'_{ij,k}+\sum_{i,j} X_{ij}Y'_{ij}.\]
Then we read from
\begin{align*}
\diag(1,\ldots,1,0,\ldots,0)=P^{-1}a\left(\overline{P}^{\top}\right)^{-1}=\sum_{i,j,k}(P^{-1}X_{ij,k})\overline{P^{-1}X_{ij,k}}^\top+\sum_{i,j}(P^{-1}X_{ij})\overline{P^{-1}X_{ij}}^\top
\end{align*}
that the rows of $P^{-1}X_{ij,k}$ and $P^{-1}X_{ij}$ corresponding to the zero rows in $\diag(1,\ldots,1,0,\ldots,0)$ vanish. Hence, the corresponding rows in the matrix $P^{-1}p_0^\top 2Q(x)$ also vanish. Therefore, the equation (\ref{eq:Pinverse_a}) has solutions. The solution of the following linear equation provides a particular solution of (\ref{eq:Pinverse_a}):
\begin{equation*}
 \overline{P}^\top\overline{\phi_x(v)}=-P^{-1}p_0^\top 2Q(x)\begin{pmatrix}
  \mathbf{0} \\
  \sigma_{xy_1}^{\top}\overline{v_1}\\
  \vdots\\
  \sigma_{xy_m}^\top\overline{v_m}
\end{pmatrix}.
\end{equation*}
That is, we can set
\begin{equation*}
\overline{\phi_x(v)}=-\left(P\overline{P}^\top\right)^{-1}p_0^\top 2Q(x)\begin{pmatrix}
  \mathbf{0} \\
  \sigma_{xy_1}^{\top}\overline{v_1}\\
  \vdots\\
  \sigma_{xy_m}^\top\overline{v_m}
\end{pmatrix},
\end{equation*}
which is a linear map from $T_x(G,\sigma)$ to $\mathbb{K}^{d}$.
\end{proof}

Next, we define discrete Bakry--\'Emery Ricci curvature tensor and the metric tensor.
\begin{definition}[Curvature tensor and metric tensor]\label{def:curvature_metric_tensors}
Let $x$ be a given vertex in a connection graph $(G,\sigma)$.  For any linear map $\phi_x: T_x(G,\sigma)\to \mathbb{K}^d$ satisfying (\ref{eq:phiv}), we define a discrete Bakry--\'Emery Ricci curvature tensor $\mathrm{Ric}^x_\infty$  as a $(0,2)$-tensor, i.e., a sesquilinear map
\begin{equation*}
 \mathrm{Ric}^x_\infty: T_x(G,\sigma)\times T_x(G,\sigma) \to \mathbb{K},
\end{equation*}
by
\begin{equation}
 \mathrm{Ric}^x_\infty(v_1,v_2):=\Psi_x\circ\Phi_x(v_1)^\top 2\Gamma_2^\sigma(x)\overline{\Psi_x\circ\Phi_x(v_2)},
\end{equation}
for any $v_1,v_2\in T_x(G,\sigma)$, where $\Phi_x: T_x(G,\sigma)\to \mathbb{K}^{(m+1)d}$ is the linear map defined as
\begin{equation}\label{eq:Phix}
  \Phi_x(v)=\begin{pmatrix}
    \phi_x(v)\\
    \sigma_{xy_1}^{-1}(v_1+\phi_x(v))\\
    \vdots \\
    \sigma_{xy_m}^{-1}(v_m+\phi_x(v))
  \end{pmatrix},\,\,v=\begin{pmatrix}
    v_1\\\vdots\\v_m
  \end{pmatrix}\in T_x(G,\sigma),
\end{equation}
and $\Psi_x$ is the linear map in Definition \ref{def:PhiPsi}.
For $N\in (0,\infty]$, the discrete $N$-Bakry--\'Emery Ricci curvature tensor is defined by
\begin{equation}
  \mathrm{Ric}^x_N(v_1,v_2):=\mathrm{Ric}^x_\infty(v_1,v_2)-\frac{2}{N}\Phi_x(v_1)^\top\Delta^\sigma(x)\overline{\Delta^\sigma(x)}^\top\overline{\Phi_x(v_2)}.
\end{equation}
 We define a metric tensor $g_x$ as a $(0,2)$-tensor, i.e., a sesquilinear map
\begin{equation*}
 g_x: T_x(G,\sigma)\times T_x(G,\sigma) \to \mathbb{K},
\end{equation*}
by
\begin{equation}
 g_x(v_1,v_2):=\Phi_x(v_1)^\top 2\Gamma^\sigma(x)\overline{\Phi_x(v_2)}.
\end{equation}
for any $v_1,v_2\in T_x(G,\sigma)$.
\end{definition}

\begin{remark}
Notice that the metric tensor is independent of the choices of the linear map $\phi_x$. Indeed, we have
\[g_x(v_1,v_2)=\sum_{i=1}^mp_{xy_i}v_{1i}^\top\overline{v_{2i}},\]
for any $v_t=(v_{t1}^\top,\ldots,v_{tm}^\top)^\top\in T_x(G,\sigma),\,t=1,2$. The metric tensor $g_x$ provides an inner product to the linear space $T_x(G,\sigma)$.
\end{remark}
\begin{proposition}
Let $x$ be a given vertex in a connection graph $(G,\sigma)$ with a balanced local connection structure $B_2^{inc}(x)$. Then the definition of the tensor $\mathrm{Ric}^x_N$ is independent of the choices of the linear map $\phi_x$ satisfying (\ref{eq:phiv}).
\end{proposition}
\begin{proof}
In the case of a balanced local connection structure $B^{inc}_{2}(x)$, the signature of any $3$-cycle $x\sim y_{i}\sim y_{j}\sim x$ or any $4$-cycle $x\sim y_{i}\sim z_{k} \sim y_{j} \sim x$ is (the conjugacy class of) $I_{d}$. By Appendix (\ref{eq:appendixa}) and (\ref{eq:B_0Q(x)})-(\ref{eq:B_0Q(x)i}), the coefficient matrices in (\ref{eq:p_02Q(x)}) are $$a=\mathbf{0}_d,\,\,\text{ and} \,\,(p_{0}^{\top}2Q(x))=\begin{pmatrix}
  \mathbf{0}_d & \mathbf{0}_d & \cdots & \mathbf{0}_d
\end{pmatrix}.$$
 And the equation (\ref{eq:p_02Q(x)})
becomes $$\mathbf{0}_{d}\phi_{x}(v)=\mathbf{0}_{d}.$$
Therefore, the vector $\phi_{x}(v)$ can be arbitrarily chosen.

For any vector $v\in \mathbb{K}^{md}$, we have by (\ref{eq:Phix}) that $\Phi_x(v)=u_1+u_2$ with
$$u_{1}:=
\begin{pmatrix}
     \mathbf{0} \\
     \sigma_{xy_{1}}^{-1}v_{1} \\
     \vdots \\
     \sigma_{xy_{m}}^{-1}v_{m}
\end{pmatrix},\,\,u_2:=
\begin{pmatrix}
    \phi_{x}(v) \\
    \sigma_{xy_{1}}^{-1}\phi_{x}(v) \\
    \vdots \\
    \sigma_{xy_{m}}^{-1}\phi_{x}(v)
\end{pmatrix}=p_0 \phi_x(v).$$
By (\ref{eq:Delta_x}), we observe that
\begin{equation*}
u_2^\top\Delta^\sigma(x)=\phi_x(v)^\top p_0^\top \Delta^\sigma(x)=\mathbf{0}_{1\times d}.
\end{equation*}
Since $(p_{0}^{\top}2Q(x))=\mathbf{0}_{d\times md}$, we have that
\begin{equation*}
u_2^\top 2Q(x)=\phi_x(v)^\top(p_{0}^{\top}2Q(x))=\mathbf{0}_{1\times md}.
\end{equation*}
Therefore, we derive by definition that
\begin{align*}
Ric^{x}_{N}(v,v)=&\Phi_{x}(v)^\top(2Q(x)-\frac{2}{N}\Delta^{\sigma}(x)\overline{\Delta^{\sigma}(x)}^\top)\overline{\Phi_{x}(v)}\\
=&(u_{1}+u_{2})^\top(2Q(x)-\frac{2}{N}\Delta^{\sigma}(x)\overline{\Delta^{\sigma}(x)}^\top)\overline{(u_{1}+u_{2})}\\
=&u_{1}^\top(2Q(x)-\frac{2}{N}\Delta^{\sigma}(x)\overline{\Delta^{\sigma}(x)}^\top)\overline{u_{1}}.
\end{align*}
In conclusion, the definition of the tensor $Ric^{x}_{N}$ is independent of the choices of $\phi_x$ in the locally balanced case.
\end{proof}
\begin{proposition}
Let $x$ be a given vertex in a $U(1)$-connection graph $(G,\sigma)$ with an unbalanced local connection structure $B_2^{inc}(x)$.  Then the linear map $\phi_x$ satisfying (\ref{eq:phiv}) is unique.
\end{proposition}
\begin{proof}
For the case of $U(1)$-connection, the signature of each edge is a complex number of norm one, i.e. $e^{i\theta}$ for some $\theta\in [0,2\pi)$.
Recall from (\ref{eq:Appendix_a}) that the coefficient $a:=p_{0}^{\top}2Q(x)\overline{p_{0}}$ of $\phi_{x}(v)$ in (\ref{eq:p_02Q(x)}) has the form
$$a=\sum_{i,j,k}X_{ij,k}\overline{X_{ij,k}}+\sum_{i,j}X_{ij}\overline{X_{ij}},$$
where $X_{ij,k}=C_{ij,k}(1-e^{i\theta_{ij,k}})$ and $X_{ij}=C_{ij}(1-e^{i\theta_{ij}}), \,\,i,j=1,\dots,m,\,\,k=1,\dots,n$ and $C_{ij,k}, C_{ij}$ are positive coefficients.

If the local connection structure of the vertex $x$ is unbalanced, then there exists some $\{ij,k\}$ or $\{ij\}$ such that $$e^{i\theta_{ij,k}}=\cos(\theta_{ij,k})+i\sin(\theta_{ij,k})\neq 1,\,\,\text{ or}\,\, e^{i\theta_{ij}}=\cos(\theta_{ij})+i\sin(\theta_{ij})\neq 1,$$
and
$$1-e^{i\theta_{ij,k}}=(1-\cos(\theta_{ij,k}))-i\sin(\theta_{ij,k})\neq 0,\,\,\text{ or}\,\, 1-e^{i\theta_{ij}}=(1-\cos(\theta_{ij}))-i\sin(\theta_{ij})\neq 0.$$

We derive that either $\Re(X_{ij,k})>0$ and $X_{ij,k}\overline{X_{ij,k}}>0$\label{notation:real_part}, or $\Re(X_{ij})>0$ and $X_{ij}\overline{X_{ij}}>0$. Hence, $a>0$
and $\phi_{x}(v)$ can be uniquely determined by the equation (\ref{eq:p_02Q(x)}).
In other words, the linear map $\phi_{x}$ is unique in the case of $U(1)$-connection.
\end{proof}

\begin{theorem}\label{thm:curvature_tensor}
Let $x$ be a vertex in a connection graph $(G,\sigma)$. Then we have for any $N\in (0,\infty]$
\begin{equation}\label{eq:curvature_tensor_identity}
    \K_{G,\sigma,x}(N)=\inf_{v\in T_x(G,\sigma)\setminus\{0\}}\frac{\mathrm{Ric}^x_N(v,v)}{g_x(v,v)}.
\end{equation}
\end{theorem}
\begin{proof}
 Recall from Lemma \ref{lemma:GammaRayleigh} that
\[\K_{G,\sigma,x}(N)=\inf_{\substack{f:V\to \mathbb{K}^d\\\text{with}\,\,\Gamma^\sigma(f)(x)\neq 0}}\frac{\Gamma_2^\sigma(f)(x)-(1/N)|\Delta^\sigma f(x)|^2}{\Gamma^\sigma(f)(x)}.\]
By the calculations in the alternative proof of Theorem \ref{thm:curvature_eigenvalue}, it is enough to take the infimum above over a sub-class of functions:
\[\{f: V\to \mathbb{K}^{d}: \Gamma^\sigma(f)(x)\neq 0, \begin{pmatrix}
  f_0 \\ f_1 \\f_2
\end{pmatrix}=\Psi_x\circ\Phi_x(f_1) \},\]
where we adopt the notations $f_0, f_1, f_2$ given in (\ref{eq:f0f1f2}). By Definition \ref{def:curvature_metric_tensors}, we obtain
\[\K_{G,\sigma,x}(N)=\inf_{\substack{f_1\in \mathbb{K}^{md}\\\text{with}\,\,g_x(f_1,f_1)\neq 0}}\frac{\mathrm{Ric}^x_N(f_1,f_1)}{g_x(f_1,f_1)}.\]
Since $g_x(f_1,f_1)\neq 0$ if and only if $f_1\neq 0$, we arrive at (\ref{eq:curvature_tensor_identity}).
\end{proof}


Notice that $\mathbb{K}^{md}=\{(v^1,\ldots, v^{md})^\top: v_i\in \mathbb{K}, i=1,\ldots,md\}$ is the space of $md$-tuples. It can also be considered as the linear space spanned by the basis $\epsilon_1,\ldots,\epsilon_{md}$, where 
\begin{equation*}\label{eq:standard_coordinate}
\epsilon_i=(0,\ldots,1, \ldots,0)
\end{equation*}
with the only $1$ appearing in the $i$-th position.
Given any orthonormal basis of the inner product space $(\mathbb{K}^{md}, g_x(\cdot,\cdot))$, we have a unique matrix corresponding to the curvature tensor $\mathrm{Ric}_N^x$.

\begin{proposition}\label{prop:tensor_matrix}
 For any nonsingular matrix $B$ satisfying (\ref{eq:B}), there exists an orthonormal basis $e_1,\ldots, e_{md}$ of the inner product space $(\mathbb{K}^{md}, g_x(\cdot, \cdot))$, such that for any $v=\sum_{i=1}^{md}v^i\epsilon_i=\sum_{i=1}^{md}v_{B}^ie_i$,
\begin{equation}\label{eq:tensor_matrix}
 \mathrm{Ric}_N^x(v,v)=\begin{pmatrix}
  v_{B}^1 & \cdots &v_{B}^{md}
\end{pmatrix}A_N(G,\sigma,x,B)\begin{pmatrix}
  \overline{v_{B}^1} \\ \vdots \\\overline{v_{B}^{md}}
\end{pmatrix}.
\end{equation}
\end{proposition}
\begin{proof}
  For any $v=\sum_{i=1}^{md}v^i\epsilon_i$ and the given $B$, we define
\begin{equation}\label{eq:vB_defn}
v_{B}:=\begin{pmatrix}
 v_{B}^1 \\ \vdots \\v_{B}^{md}
\end{pmatrix}=\begin{pmatrix}
  \mathbf{0}_{md\times d} & I_{md}
\end{pmatrix}(B^{-1})^\top \Phi_x(v),\end{equation}
where $\Phi_x$ is the linear map given in (\ref{eq:Phix}).
Then we verify that
\begin{align}
 g_x(v,v)=&\Phi_x(v)^\top B^{-1}B2\Gamma^\sigma(x)\overline{B}^\top\overline{B^{-1}}^\top\overline{\Phi_x(v)}\notag\\
=&\Phi_x(v)^\top B^{-1}\begin{pmatrix}
  \mathbf{0}_{d\times d} & \\
& I_{md}
\end{pmatrix}\overline{B^{-1}}^\top\overline{\Phi_x(v)}\notag\\
=& v_{B}^\top\overline{v_B}.\label{eq:vB}
\end{align}
Let $\xi_B: \mathbb{K}^{md}\to \mathbb{K}^{md}$ be the linear map  defined by $\xi_B(v):=v_{B}$ for any $v\in \mathbb{K}^{md}$. If $\xi_B(v)=v_B=0$, we derive from (\ref{eq:vB}) that $g_x(v,v)=0$. Hence, $v=0$ as $g_x$ is an inner product. That is, the linear map $\xi_B$ is injective. Since $\xi_B$ is a linear map between two linear spaces of the same dimension, it is surjective. Thus, there exists a nonsingular matrix $M_B$ such that
\begin{equation}
\begin{pmatrix}
 v_{B}^1 \\ \vdots \\v_{B}^{md}
\end{pmatrix}=M_B \begin{pmatrix}
 v^1 \\ \vdots \\v^{md}
\end{pmatrix}.
\end{equation}

Then
\begin{equation*}
  \begin{pmatrix}
  e_1 \\ \vdots \\ e_{md}
\end{pmatrix}:=\left(M_B^{-1}\right)^\top\begin{pmatrix}
  \epsilon_1 \\ \vdots \\ \epsilon_{md}
\end{pmatrix}
\end{equation*}
is a basis of $\mathbb{K}^{md}$ such that $v=\sum_{i=1}^{md}v^i\epsilon_i=\sum_{i=1}^{md}v_B^ie_i$. Due to (\ref{eq:vB}), the basis $e_1, \ldots e_{md}$ is an orthonormal basis of the inner product space $(\mathbb{K}^{md}, g_x(\cdot, \cdot))$.

Finally, we check by definition
\begin{align*}
  \mathrm{Ric}_N^x(v,v)=&\Psi_x\circ\Phi_x(v)^\top (2\Gamma_2^\sigma(x)-\frac{2}{N}\Delta^\sigma(x)\overline{\Delta^\sigma(x)}^\top)\overline{\Psi_x\circ\Phi_x(v)}\\
=&\Phi_x(v)^\top (2Q(x)-\frac{2}{N}\Delta^\sigma(x)\overline{\Delta^\sigma(x)}^\top)\overline{\Phi_x(v)}\\
=&\Phi_x(v)^\top B^{-1} B(2Q(x)-\frac{2}{N}\Delta^\sigma(x)\overline{\Delta^\sigma(x)}^\top)\overline{B}^\top\overline{B^{-1}}^\top\overline{\Phi_x(v)}\\
=&v_B^\top A_N(G,\sigma,x,B)\overline{v_B}.
\end{align*}
That is, (\ref{eq:tensor_matrix}) holds true.
\end{proof}
\begin{remark}
In the identity (\ref{eq:tensor_matrix}), the tensor $\mathrm{Ric}^x_N(v,v)$ and the linear map $\xi_B$ sending $v$ to $v_B$ depend on the choices of $\phi_x$, and the matrix $A_N(G,\sigma,x,B)$ is independent of the choices of $\phi_x$.
\end{remark}

In view of Proposition \ref{prop:tensor_matrix}, we give an alternative proof of Proposition \ref{prop:unitarily_equivalent}.
\begin{proof}[An alternative proof of Proposition \ref{prop:unitarily_equivalent}]
For any two nonsingular matrices $B_1$ and $B_2$ satisfying (\ref{eq:B}), the curvature matrices $A_N(G,\sigma,x,B_1)$ and $A_N(G,\sigma,x,B_2)$ are the matrices of the discrete Bakry--\'Emery Ricci curvature tensor with respect to different orthonormal basis of the inner product space $(\mathbb{K}^{md},g_x(\cdot,\cdot))$. Hence, they are unitarily equivalent.

Next, we derive the unitary matrix explicitly. For any $v\in \mathbb{K}^{md}$, let $v_{B_1}$ and $v_{B_2}$ be defined as in (\ref{eq:vB_defn}) with respect to $B_1$ and $B_2$. Then, we calculate
\begin{align*}
 v_{B_2}=&\begin{pmatrix}
  \mathbf{0}_{md\times d} & I_{md}
\end{pmatrix}\left(B_2^{-1}\right)^\top\Phi_x(v)\\
=& \begin{pmatrix}
  \mathbf{0}_{md\times d} & I_{md}
\end{pmatrix}\left(B_1B_2^{-1}\right)^\top(B_1^{-1})^\top\Phi_x(v)\\
=&\begin{pmatrix}
  \mathbf{0}_{md\times d} & \left(B_1B_2^{-1}\right)^\top_{\hat{1}}
\end{pmatrix}\left(B_1^{-1}\right)^\top\Phi_x(v)\\
=& \left(B_1B_2^{-1}\right)^\top_{\hat{1}}v_{B_1}.
\end{align*}
In the third equality above, we have applied Lemma \ref{lemma:B1B2inverse}. Proposition \ref{prop:tensor_matrix} tells that
\begin{equation*}
  \mathrm{Ric}_N^x(v,v)=v_{B_1}^\top A_N(G,\sigma,x,B_1)\overline{v_{B_1}}=v_{B_2}^\top A_N(G,\sigma,x,B_2)\overline{v_{B_2}}.
\end{equation*}
This leads to $A_N(G,\sigma,x,B_1)=(B_1B_2^{-1})_{\hat{1}}A_N(G,\sigma,x,B_2)\overline{(B_1B_2^{-1})_{\hat{1}}}^\top$.
\end{proof}
\begin{remark}
  From the above proof, we see that the matrix $(B_1B_2^{-1})_{\hat{1}}^{-1}$ is actually the transformation matrix between two orthonormal basis of the inner product space $(\mathbb{K}^{md}, g_x(\cdot, \cdot))$. Therefore, it is quite natural to see $(B_1B_2^{-1})_{\hat{1}}$  is unitary (Lemma \ref{lemma:B1B2}) and the transitive property in Corollary \ref{coro:transitive} holds true.
\end{remark}
\section{Properties of curvature functions}\label{section:function}
In this section, we study the properties of curvature functions with the help of curvature matrices.
\begin{theorem}\label{thm:curvature_function}
Let $(G,\sigma)$ be a connection graph and $x$ be a given vertex. Then the curvature function $\K_{G,\sigma,x}: (0,\infty]\to \mathbb{R}$ satisfying the following properties.
\begin{itemize}
  \item [(i)] It is continuous, monotone non-decreasing and concave with \[\lim_{N\to 0}\K_{G,\sigma,x}(N)=-\infty\,\,\text{ and} \,\,\lim_{N\to \infty}\K_{G,\sigma,x}(N)<\infty.\]
  \item [(ii)] If there exist $0<N_1<N_2\leq \infty$ such that $\K_{G,\sigma,x}(N_1)=\K_{G,\sigma,x}(N_2)$, then the function $\K_{G,\sigma,x}$ is constant on the interval $[N_1,\infty]$.
  \item [(iii)] If the multiplicity of $\lambda_{\min}(A_{N_0}(G,\sigma,x,B_0))$ is larger than $d$ for some $N_0\in (0,\infty]$, then the function $\K_{G,\sigma,x}$ is constant on the interval $[N_0,\infty]$.
\end{itemize}
\end{theorem}

\begin{proof} The continuity of the curvature function is due to the fact that the zeros of a polynomial are continuous functions of the coefficients of the polynomial. (See the proof of \cite[Proposition 3.1]{CKLP22}.)

For two Hermitian matrices $M_1$ and $M_2$, we recall that
\begin{align}\lambda_{\min}(M_1+M_2)=&\inf_{v\neq 0}\frac{v^\top (M_1+M_2) \overline{v}}{v^\top \overline{v}}\notag\\
\geq &\inf_{v\neq 0}\frac{v^\top M_1 \overline{v}}{v^\top \overline{v}}+\inf_{v\neq 0}\frac{v^\top M_2 \overline{v}}{v^\top \overline{v}}=\lambda_{\min}(M_1)+\lambda_{\min}(M_2),\label{eq:sum_trick}\end{align}
where the inequality holds with equality if and only if $M_1$ and $M_2$ share an eigenvector corresponding to $\lambda_{\min}(M_1)$ and $\lambda_{\min}(M_2)$.

By Proposition \ref{prop:unitarily_equivalent} and Theorem \ref{thm:curvature_eigenvalue}, we consider the minimal eigenvalue of the matrix $A_N(G,\sigma,x,B_0)$ with respect to the canonical choice $B_0$ given in (\ref{eq:canonicalB}). Observe the following facts from (\ref{eq:canonicalv0}) that
\[\mathbf{v}_0(B_0)\overline{\mathbf{v}_0(B_0)}^\top\mathbf{v}_0(B_0)=\sum_{y}p_{xy} \mathbf{v}_0(B_0).\]
Therefore, the matrix $\mathbf{v}_0(B_0)\overline{\mathbf{v}_0(B_0)}^\top$ has eigenvalue $\sum_{y}p_{xy}$ with multiplicity $d$ and eigenvalue $0$ with multiplicity $(m-1)d$.

The monotonicity follows from (\ref{eq:sum_trick}) as follows. For any $0<N_1<N_2\leq \infty$, we have
\begin{align}\label{eq:monotonicity}
\lambda_{\min}(A_{N_2})\geq \lambda_{\min}(A_{N_1})+\lambda_{\min}\left(\left(\frac{2}{N_1}-\frac{2}{N_2}\right)\mathbf{v}_0(B_0)\overline{\mathbf{v}_0(B_0)}^\top\right)=\lambda_{\min}(A_{N_1}),
\end{align}
where we write $A_{N_1}=A_{N_1}(G,\sigma,x,B_0)$ for short.

The concavity can also be derived from (\ref{eq:sum_trick}). Indeed, for any $0<N_1<N_2\leq \infty$ and $\gamma\in (0,1)$, we have
\begin{align*}
  &\lambda_{\min}(A_{\gamma N_1+(1-\gamma)N_2})\\
  =&\lambda_{\min}\left(\gamma A_{N_1}+(1-\gamma)A_{N_2}+\left(\frac{2\gamma}{N_1}+\frac{2(1-\gamma)}{N_2}-\frac{2}{\gamma N_1+(1-\gamma)N_2}\right)\mathbf{v}_0(B_0)\overline{\mathbf{v}_0(B_0)}^\top\right)\\
  \geq &\gamma\lambda_{\min}(A_{N_1})+(1-\gamma)\lambda_{\min}(A_{N_2}).
\end{align*}
We further calculate that
\[\lim_{N\to \infty}\K_{G,\sigma,x}(N)=\K_{G,\sigma,x}(\infty)=\lambda_{\min}(A_\infty)<\infty,\]
and
\begin{align*}\lim_{N\to 0}\K_{G,\sigma,x}(N)=\lim_{N\to 0}\lambda_{\min}(A_N)\leq & \lim_{N\to 0}\left(\lambda_{\max}(A_{\infty})+\lambda_{\min}(-\frac{2}{N}\mathbf{v}_0(B_0)\overline{\mathbf{v}_0(B_0)}^\top)\right)\\
=& \lambda_{\max}(A_{\infty})-\lim_{N\to 0}\frac{2}{N}\sum_{y}p_{xy}=-\infty.
\end{align*}
This completes the proof of the property $(i)$.

Regarding property $(ii)$, the condition $\K_{G,\sigma,x}(N_1)=\K_{G,\sigma,x}(N_2)$ implies that the inequality in (\ref{eq:monotonicity}) is actually tight. By (\ref{eq:sum_trick}), this means that the matrices $A_{N_1}$ and $\mathbf{v}_0(B_0)\overline{\mathbf{v}_0(B_0)}^\top$ share an eigenvector corresponding to their minimal eigenvalues. This further implies that $\lambda_{\min}(A_N)=\lambda_{\min}(A_{N_1})$ for any $N\in [N_1,\infty]$.

Next, we show property $(iii)$. Let us denote by $\mathbf{v}_0(B_0)=\begin{pmatrix}
  v_{01} & \cdots & v_{0d}
\end{pmatrix}$ with each $v_{0i}\in \mathbb{K}^{md}$. By definition, the vectors $v_{0i}$, $i=1,\ldots,d$ are pairwise orthogonal.

We claim that there exists a nonzero vector $w$ in the minimal eigenspace $E_{\min}(A_{N_0})$ \label{notation:eigenspace}which is orthogonal to the space $\mathrm{span}\{v_{01},\ldots,v_{0d}\}$. Indeed, such a vector $w$ can be constructed as follows. Pick $d+1$ linearly independent vectors $\xi_{j}$, $j=1,\ldots,d+1$ in the space $E_{\min}(A_{N_0})$. Then there exist numbers $a_j^i\in \mathbb{K}$, $i=1,\ldots, d$, $j=1,\ldots,d+1$ and vectors $w_j$, $j=1,\ldots,d+1$ orthogonal to the space $\mathrm{span}\{v_{01},\ldots,v_{0d}\}$ such that
\[\xi_{j}=\sum_{i=1}^da_j^iv_{0i}+w_j,\,\,j=1,\ldots, d+1.\]
If there exists a $1\leq j_0\leq d+1$ with $\sum_{i=1}^da_{j_0}^iv_{0i}=0$, we can set $w=w_{j_0}$.
Otherwise, the vectors $\sum_{i=1}^da_j^iv_{0i}, j=1,\ldots,d+1$ are $d+1$ nonzero vectors in the $d$-dimensional space $\mathrm{span}\{v_{01},\ldots,v_{0d}\}$, and hence are linearly dependent. There exist $c_1,\ldots,c_{d+1}\in \mathbb{K}$ such that $\sum_{j=1}^{d+1}c_j\sum_{i=1}^da_j^iv_{0i}=0$. Then we can set
$w=\sum_{j=1}^{d+1}c_j\xi_j=\sum_{j=1}^{d+1}c_jw_j$.

Notice that such a vector $w$ is an eigenvector corresponding to both $\lambda_{\min}(A_{N_0})$ and $\lambda_{\min}\left(\mathbf{v}_0(B_0)\overline{\mathbf{v}_0(B_0)}^\top\right)$. Therefore, the curvature $\K_{G,\sigma,x}(N)=\lambda_{\min}(A_N)$ is constant for $N\in [N_0,\infty]$ by (\ref{eq:sum_trick}).
\end{proof}

\begin{remark}
The proof presented above is largely adapted from the approach in \cite[Section 3]{CKLP22}. In particular, when $d=1$---that is, for connection graphs involving $O(1)$ or $U(1)$ connections---the argument used in \cite[Theorem 1.3]{CKLP22} can be applied directly to establish the following fact:
There exists a unique threshold $N_1\in (0,\infty]$ such that
$\K_{G,\sigma,x}$ is analytic, \emph{strictly} monotone increasing and \emph{strictly} concave on $(0,N_1)$ and constant on $[N_1,\infty]$.
\end{remark}
When the connections are $d$-dimensional with $d\geq 2$, the behavior of the system is far less understood, and several fundamental questions remain unresolved:
\begin{itemize}
  \item [(1)] Is the multiplicity condition in Theorem \ref{thm:curvature_function} (iii) necessary?
  \item [(2)] Does the multiplicity of $\lambda_{\min}(A_N)$ necessarily grow monotonically as $N$ increases?
  \item [(3)] Does there exist a threshold $N_1\in (0,\infty]$ such that the curvature function $\K_{G,\sigma,x}$ is analytic on $(0,N_1)$ and constant on $[N_1,\infty]$?
\end{itemize}
\section{Curvature matrices of Cartesian products}\label{section:product}



 In this section, we study the curvature matrices of Cartesian products of connection graphs.
 \begin{definition}
 Let $(G,\sigma)$ be a connection graph. The signature group $\Sigma$ is defined to be the group generated by the elements of the set $\{\sigma_{xy}: (x,y)\in E^{or}\}$.
 \end{definition}

\begin{definition}\label{def:Cartesian}
Let $(G,\sigma)$ and $(G',\sigma')$ be two connection graphs with \[G=(V, w, \mu), G'=(V', w', \mu')\ \  \text{and}\ \ \sigma:E^{or}\to U(d),\,\sigma': (E')^{or}\to U(d).\] Let $\alpha, \beta$ be two positive real numbers. We define the Cartesian product of these two connection graphs as a connection graph
\begin{align}\label{notation:cartesian_product}
(G,\sigma)\times_{\alpha,\beta}(G',\sigma'):=(G\times_{\alpha,\beta} G', \sigma^\times).
\end{align}
Here, $G\times_{\alpha,\beta} G'$ is the Cartesian product of the two graphs $G$ and $G'$ assigned with the following edge weights and vertex measure: For $x,y\in V$ and $x',y'\in V'$,
\begin{align*}
    w_{(x,x')(y,x')}&=\alpha w_{xy}\mu_{x'},\\
    w_{(x,x')(x,y')}&=\beta w'_{x'y'}\mu_{x},\\
    \mu_{(x,x')}&=\mu_x\mu_{x'}.
\end{align*}
 The connection $\sigma^\times$ \label{notation:cartesian_connection}of each oriented edge in the Cartesian product of $G$ and $G'$ is given as follows:
 \begin{align*}
     \sigma^\times_{(x,x')(y,x')}&= \sigma_{xy},\\
    \sigma^\times_{(x,x')(x,y')}&= \sigma'_{x'y'}.
 \end{align*}
Here we use the same notation $w,\mu$ for the product graph and the graph $G$. The associated graphs can be determined by the vertex input.
\end{definition}
\begin{theorem}\label{thm:Adecomp}
Let $(G,\sigma)$ and $(G',\sigma')$ be two $U(d)$-connection graphs. Suppose that the two signature groups $\Sigma$ and $\Sigma'$ commute. For any $x\in V$ and $x'\in V'$, the curvature matrices with respect to the canonical choice $B_0$ in (\ref{eq:canonicalB}) satisfy
\begin{align}\label{eq:Adecomp}A_{\infty}&(x,x')=
\begin{pmatrix}
    \alpha A_{\infty}(x) &   \\
                      & \beta A_{\infty}(x')
\end{pmatrix}+R(x,x'),
\end{align}
where
\begin{equation}\label{eq:R}
R(x,x'):=\begin{pmatrix}
    \alpha^{3}\overline{\omega}\left((\alpha^2a)^{\dagger}-(\alpha^2a+\beta^2a')^{\dagger}\right)\omega^{\top} & -\alpha^{\frac{3}{2}}\beta^{\frac{3}{2}}\overline{\omega} (\alpha^2a+\beta^2a')^{\dagger}(\omega ')^{\top} \\
    -\alpha^{\frac{3}{2}}\beta^{\frac{3}{2}}\overline{\omega '}(\alpha^2a+\beta^2a')^{\dagger}\omega^{\top}   &   \beta^3\overline{\omega'}\left((\beta^2a')^{\dagger}-(\alpha^2a+\beta^2a')^{\dagger}\right)(\omega')^{\top}
\end{pmatrix}
\end{equation}
is positive semidefinite. Here, $a,a', \omega, \omega'$ are given by (\ref{eq:a}), (\ref{eq:omega}), and we use the notation $A_\infty(x)=A_\infty(G,\sigma,x,B_0)$ for short.
In particular, if one of the two local connection structures $B_2^{inc}(x)$ and $B_2^{inc}(x')$ is balanced, we have $R(x,x')=\mathbf{0}$.
\end{theorem}

Before the proof of Theorem \ref{thm:Adecomp}, we first establish several lemmas.
\begin{lemma}\label{lemma:commut_switch}
Consider two connection graphs $(G,\sigma)$ and $(G',\sigma')$. Suppose that the two signature groups $\Sigma$ and $\Sigma'$ commute. Let $\tau: E^{or}\to \Sigma$ and $\tau': (E')^{or}\to \Sigma'$ be two switching functions of $(G,\sigma)$ and $(G',\sigma')$, respectively. Then
$$(G,\sigma^{\tau})\times_{\alpha,\beta} (G',\sigma'^{\tau'})=(G\times_{\alpha,
\beta} G',(\sigma^\times)^{\tau^\times}),$$
where $\tau^\times (x,x'):=\tau(x)\tau'(x')$ for any $x\in V$ and $x'\in V'$.
That is, the Cartesian product $(G,\sigma)$ and $(G',\sigma')$ is switching equivalent to the Cartesian product of $(G,\sigma^\tau)$ and $(G', \sigma'^{\tau'})$.
\end{lemma}
\begin{proof}
For any $x,y\in V$ such that $\{x,y\}\in E$ and any $x'\in V'$, we calculate that
\begin{align*}
    \tau^\times(x,x')\sigma^\times_{(x,x')(y,x')} \tau^\times(y,x')^{-1}
    = \tau(x)\tau'(x')\sigma_{xy}\tau'(x')^{-1}\tau(y)^{-1}
    =\tau(x)\sigma_{xy}\tau(y)^{-1},
\end{align*}
where we use the community of $\Sigma$ and $\Sigma'$ in the second equality. Similarly, we have, for any $x\in V$ and any $x',y'\in V'$ such that $\{x',y'\}\in E'$, 
\[\tau^\times(x,x')\sigma^\times_{(x,x')(x,y')} \tau^\times(x,y')^{-1}=\tau(x)\tau'(x')\sigma'_{x'y'}\tau'(y')^{-1}\tau(x)^{-1}=\tau'(x')\sigma'_{x'y'}\tau'(y')^{-1}.\]
This completes the proof.
\end{proof}
\begin{lemma}\label{lemma:Qdecomp}
Let $(G,\sigma)$ and $(G',\sigma')$ be two connection graphs. Consider two vertices $x\in V$ and $x'\in V'$. Suppose that
\begin{equation}\label{eq:localId}
    \sigma_{xy}=I_d, \,\,\text{and}\,\,\sigma'_{x'y'}=I_d,\,\,\text{for any}\,\,y\in S_1(x), y'\in S_1(x').
\end{equation}
 Then we have the following direct sum decomposition for $Q(x,x'):=Q(\Gamma_2^{\sigma^\times}(x,x'))$:
 \begin{equation}\label{eq:Qdecomp}
   Q(x,x')=
\begin{pmatrix}
    \alpha^{2}Q(x)_{x,x}+\beta^{2}Q(x')_{x',x'} & \alpha^{2}Q(x)_{x,S_{1}(x)} & \beta^{2}Q(x')_{x',S_{1}(x')} \\
    \alpha^{2}Q(x)_{S_{1}(x),x} & \alpha^{2}Q(x)_{S_{1}(x),S_{1}(x)} & \mathbf{0} \\
    \beta^{2}Q(x')_{S_{1}(x'),x'} & \mathbf{0} & \beta^{2}Q(x')_{S_{1}(x'),S_{1}(x')}
\end{pmatrix}.
 \end{equation}
In particular, we have
$$Q(x,x')_{\hat{1}}=\alpha^2 Q(x)_{\hat{1}}\oplus\beta^2 Q(x')_{\hat{1}}.$$
\end{lemma}
\begin{proof}
It is obtained by a straightforward computation using definitions. The details are presented in the Appendix.
\end{proof}

\begin{corollary}\label{coro:Qdecomp}
Let $(G,\sigma)$ and $(G',\sigma')$ be two connection graphs. If the two signature groups $\Sigma$ and $\Sigma'$ commute, then the matrix $Q(x,x')$ satisfies (\ref{eq:Qdecomp}) for any $x\in V$ and $x'\in V'$.
\end{corollary}
\begin{proof}
Let $\tau, \tau'$ be two switching functions such that the connections $\sigma^\tau$ and $\sigma'^{\tau'}$ satisfy (\ref{eq:localId}). Applying Lemma \ref{lemma:commut_switch} and Lemma \ref{lemma:Qdecomp} yields the decomposition (\ref{eq:Qdecomp}) for the matrix $Q(\Gamma_2^{(\sigma^\times)^{\tau^\times}}(x,x'))$. Due to (\ref{eq:Qswitch}), we have
\begin{equation}
    Q(\Gamma_2^{(\sigma^\times)^{\tau^\times}}(x,x'))=(\tau^\times_{B_1(x,x')})^\top Q(\Gamma_2^{(\sigma^\times)}(x,x'))\overline{\tau^\times_{B_1(x,x')}}.
\end{equation}
By community of $\Sigma$ and $\Sigma'$ and applying (\ref{eq:Qswitch}) again, we derive the decomposition (\ref{eq:Qdecomp}) for the matrix
$Q(x,x')=Q(\Gamma_2^{\sigma^\times}(x,x'))$.
\end{proof}
\begin{remark}
If both the local connection structures $B_2^{inc}(x)$ and $B_2^{inc}(x')$ are balanced, then the first $d$ rows and $d$ columns of the matrix $Q(x,x')$ vanish. That is, we have
$$\begin{pmatrix}
    \mathbf{0}_{d} & \mathbf{0} & \mathbf{0} \\
    \mathbf{0} & \alpha^{2}Q(x)_{\hat{1}} & \mathbf{0} \\
    \mathbf{0} & \mathbf{0} & \beta^{2}Q(x')_{\hat{1}}
\end{pmatrix}.$$
\end{remark}

Furthermore, We need the following linear algebraic lemma.
\begin{lemma}\label{lemma:diagonalization}
For two positive semidefinite Hermitian matrices $A$ and $B$, there exists a nonsingular matrix $P$ such that $PA\overline{P}^{\top}$ and $PB\overline{P}^{\top}$ are both diagonal matrices.
\end{lemma}
\begin{proof}
Before the proof, we prepare two facts about positive semidifinite matrices.
\begin{itemize}
  \item Every diagonal entry of a positive semidifinite matrix is non-negative.
  \item If a diagonal entry is $0$, then all entries in the same row and same column are $0$.
\end{itemize}
Since all principal minors of a positive semidefinite matrix are non-negative, it follows in particular that each diagonal element---being a $1\times 1$ principal minor---must also be non-negative. This shows the first fact.

Next, we show the second fact. Any positive semidifinite matrix $D$ can be written as
\begin{align}\label{eq:positive D}
D=X\overline{X}^\top
\end{align}
for some square matrix $X$. We denote $X$ as
\[X=
\begin{pmatrix}
  x_{1}^\top \\
  x_{2}^\top \\
  \vdots  \\
  x_{n}^\top
\end{pmatrix},
\]
where each $x_{i}^\top$ is a row vector. Then the identity (\ref{eq:positive D}) can be rewritten as
\[\begin{pmatrix}
  x_{1}^\top \\
  x_{2}^\top \\
  \vdots  \\
  x_{n}^\top
\end{pmatrix}
\begin{pmatrix}
   \overline{x_{1}} & \overline{x_{2}} & \cdots & \overline{x_{n}}
\end{pmatrix}=D.
\]
Notice that the diagonal entries are exactly
\[
\begin{pmatrix}
  x_{1}^\top\overline{x_{1}} &  &  &  \\
                            & x_{2}^\top\overline{x_{2}} &           &  \\
                            &                            & \ddots    &   \\
                            &                            &           & x_{n}^\top\overline{x_{n}}
\end{pmatrix}=
\begin{pmatrix}
          d_{11}   &          &        &      \\
                   & d_{22}   &        &      \\
                   &          & \ddots &      \\
                   &          &        & d_{nn}
\end{pmatrix}.\]
If there is a zero diagonal entry, say, $d_{ii}=0=x_{i}^\top\overline{x_{i}}$, we get $x_{i}^\top=
\begin{pmatrix}
  0 & 0 & \cdots & 0
\end{pmatrix}$
and all entries in the $i$-th row and $i$-th column of $D=X\overline{X}^\top$ are zeros.

Now we prove this lemma. Since the matrices $A$ and $B$ are both positive semidefinite, $A+B$ is again positive semidefinite. There exists a nonsingular matrix $P_{1}$ such that $P_{1}(A+B)\overline{P_{1}}^\top=\diag(I_r,\mathbf{0})$, here $r$ is the rank of $A+B$. Then
$$P_{1}A\overline{P_{1}}^\top=\diag(S_{1},\mathbf{0}),\,\,P_{1}B\overline{P_{1}}^\top=\diag(S_{2},\mathbf{0}),$$
where $S_{1}$ and $S_{2}$ are positive semidifinite and satisfy $S_{1}+S_{2}=I_{r}$. Hence, there exists a unitary matrix $T$ such that the matrix $TS_{1}\overline{T}^\top$ is a diagonal matrix. Moreover, the matrix $TS_{2}\overline{T}^\top=I_{r}-TS_{1}\overline{T}^\top$ is also a diagonal matrix. Set $P:=\diag{(T,I_{n-r})}P_{1}$. Then, both $PA\overline{P}^\top$ and $PB\overline{P}^\top$ are diagonal matrices.
\end{proof}


\begin{proof}[Proof of Theorem \ref{thm:Adecomp}]
First observe that for any $x,y\in V$ and any $x',y'\in V'$
\[p_{(x,x')(y,x')}=\alpha p_{xy}, \,\,\text{and}\,\,p_{(x,x')(x,y')}=\beta p_{x'y'}.\]
By (\ref{eq:canonicalB}), the canonical choice of matrix $B$ at $(x,x')$ in the Cartesian product is
\begin{equation*}
    B_0=\begin{pmatrix}
        I_d & \overline{\sigma_{xy_1}} & \cdots &\overline{\sigma_{xy_m}} &  \overline{\sigma'_{x'y'_1}} & \cdots &\overline{\sigma'_{x'y'_{m'}}} \\
        &\frac{1}{\sqrt{\alpha p_{xy_1}}}I_d & & & & & \\
        &&\ddots&&&&\\
        &&&\frac{1}{\sqrt{\alpha p_{xy_m}}}I_d&&&\\
        &&&&\frac{1}{\sqrt{\beta p_{x'y'_1}}}I_d&&\\
        &&&&&\ddots&\\
        &&&&&&\frac{1}{\sqrt{\beta p_{x'y'_{m'}}}}I_d
    \end{pmatrix}.
\end{equation*}
Using Corollary \ref{coro:Qdecomp}, we derive that
\[B_02Q(x,x')\overline{B_0}^\top=\begin{pmatrix}
    \alpha^2 a+\beta^2 a' & \alpha^{\frac{3}{2}}\omega^\top & \beta^{\frac{3}{2}}\omega'^\top \\
    \alpha^{\frac{3}{2}}\overline{\omega}&\alpha \left(B_02Q(x)\overline{B_0}^\top\right)_{\hat{1}} & \\
    \beta^{\frac{3}{2}}\overline{\omega'}&&\beta\left(B_02Q(x')\overline{B_0}^\top\right)_{\hat{1}}
\end{pmatrix}.\]
Here, we use the same notations $B_0, Q$ in $B_02Q(x)\overline{B_0}^\top$, $B_02Q(x')\overline{B_0}^\top$ and $B_02Q(x,x')\overline{B_0}^\top$ for the graphs $(G,\sigma)$, $(G',\sigma')$ and their Cartesian product graph. Those matrices are distinguished by the vertex input and are determined in the corresponding graphs.

By definition, we have the curvature matrix
\begin{align*}
    A_\infty(x,x')=&\begin{pmatrix}
   \alpha (B_02Q(x)\overline{B_0}^\top)_{\hat{1}} & \\
   &\beta(B_02Q(x')\overline{B_0}^\top)_{\hat{1}}
\end{pmatrix}\\
&\ \ \ \ \ \quad \quad -\begin{pmatrix}
 \alpha^{\frac{3}{2}}\overline{\omega}\\   \beta^{\frac{3}{2}}\overline{\omega'}
\end{pmatrix}(\alpha^2a+\beta^2a')^\dagger\begin{pmatrix}
    \alpha^{\frac{3}{2}}\omega^\top & \beta^{\frac{3}{2}}\omega'^\top
\end{pmatrix}\\
=&\begin{pmatrix}
   \alpha A_{\infty}(x) & \\
   &\beta A_{\infty}(x')
\end{pmatrix}+\begin{pmatrix}
   \alpha \overline{\omega} a^\dagger \omega^\top & \\
   &\beta \overline{\omega'} a'^{\dagger}\omega^{\top}
\end{pmatrix}\\
&\ \ \ \ \ \quad \quad -\begin{pmatrix}
 \alpha^{\frac{3}{2}}\overline{\omega}\\   \beta^{\frac{3}{2}}\overline{\omega'}
\end{pmatrix}(\alpha^2a+\beta^2a')^\dagger\begin{pmatrix}
    \alpha^{\frac{3}{2}}\omega^\top & \beta^{\frac{3}{2}}\omega'^\top
\end{pmatrix}.
\end{align*}
Summing up the last two matrices yields the matrix $R(x,x')$.

Next, we show that $R(x,x')$ is positive semidefinite. If one of the two local connection structures $B_2^{inc}(x)$ and $B_2^{inc}(x')$ is balanced,  one of $(a,\omega^\top)$ and $(a',\omega'^\top)$ vanishes (see Proposition \ref{prop:aadaggeromega} in Appendix \ref{section:appendix}), and hence $R(x,x')=\mathbf{0}$. In the following, we assume that both local connection structures $B_2^{inc}(x)$ and $B_2^{inc}(x')$  are unbalanced. In this case, both $a$ and $a'$ in (\ref{eq:Adecomp}) are positive semidifinite (see Proposition \ref{prop:aadaggeromega} in Appendix \ref{section:appendix}). By Lemma \ref{lemma:diagonalization}, there exists a nonsingular matrix $P$ such that
$$Pa\overline{P}^{\top}=\diag(\mu_{1},\ldots,\mu_{d}),\,\,\text{for some}\,\,\mu_{i}\in \mathbb{R}_{\geq 0},i=1,\dots,d.$$
and $$Pa'\overline{P}^{\top}=\diag(\lambda_1,\ldots,\lambda_d), \,\,\text{for some}\,\,\lambda_i\in \mathbb{R}_{\geq 0}, i=1,\ldots,d.$$
Then, we obtain
\begin{align}\label{eq:ainverse}
    a^{\dagger}=\overline{P}^\top \diag(\mu^\dagger_{1},\ldots,\mu^\dagger_{d})P, \,\,\,\,(a')^{\dagger}=\overline{P}^\top \diag(\lambda_1^{\dagger},\ldots,\lambda_d^{\dagger})P
\end{align}
and
\begin{equation}\label{eq:aplusinverse}
  (\alpha^2 a+\beta^2 a')^{\dagger}=\overline{P}^\top \diag((\alpha^2\mu_{1}+\beta^2\lambda_1)^{\dagger},\ldots,(\alpha^2\mu_{d}+\beta^2\lambda_d)^{\dagger})P,
\end{equation}
where we use the notation $\dagger$ for
$$\lambda^\dagger=\left\{
\begin{aligned}
&\frac{1}{\lambda},&\textrm{if} \,\, \lambda\neq 0,\\
&0,&\textrm{otherwise.}
\end{aligned}
\right.
$$
Let us write $R=R(x,x')$ for short. Notice that we have
\begin{equation*}
    R=\begin{pmatrix}
        \alpha^{\frac{3}{2}}\overline{\omega} & \\
        & \beta^{\frac{3}{2}}\overline{\omega'}
    \end{pmatrix}
   \begin{pmatrix}
   (\alpha^2a)^{\dagger}-(\alpha^2a+\beta^2a')^{\dagger} & - (\alpha^2a+\beta^2a')^{\dagger} \\
    -(\alpha^2a+\beta^2a')^{\dagger}   &   (\beta^2a')^{\dagger}-(\alpha^2a+\beta^2a')^{\dagger}
\end{pmatrix}
    \begin{pmatrix}
        \alpha^{\frac{3}{2}}\omega^\top & \\
        & \beta^{\frac{3}{2}}\omega'^\top
    \end{pmatrix}.
\end{equation*}
Employing (\ref{eq:ainverse}) and (\ref{eq:aplusinverse}), we can reformulate $R$ as below:
\begin{equation}\label{eq:Rleftright}
    R=\overline{W}^\top
   \begin{pmatrix}
    \diag((\alpha^{2}\mu_{i})^\dagger-(\alpha^2\mu_{i}+\beta^2\lambda_i)^\dagger)  &   -\diag((\alpha^2\mu_{i}+\beta^2\lambda_i)^\dagger)\\
    -\diag((\alpha^2\mu_{i}+\beta^2\lambda_i)^\dagger)&   \diag((\beta^2\lambda_i)^\dagger-(\alpha^2\mu_{i}+\beta^2\lambda_i)^\dagger)
\end{pmatrix}
    W,
\end{equation}
where we use the notation $\diag(\lambda_i):=\diag(\lambda_1,\ldots,\lambda_d)$ and
\[W:= \begin{pmatrix}
        \alpha^{\frac{3}{2}}P\omega^\top & \\
        & \beta^{\frac{3}{2}}P\omega'^\top
    \end{pmatrix}.\]
We further introduce the notations
\begin{equation*}
   D_1:= \diag\left(\sqrt{(\alpha^2\mu_i)^\dagger-(\alpha^2\mu_i+\beta^2\lambda_i)^\dagger}\right),\,\,\text{and}\,\,
        D_2:=\diag\left(\sqrt{(\beta^2\lambda_i)^\dagger-(\alpha^2\mu_i+\beta^2\lambda_i)^\dagger}\right).
    \end{equation*}
Observe that the middle matrix in (\ref{eq:Rleftright}) can be decomposed as
\begin{equation*}
    \begin{pmatrix}
    \diag((\alpha^{2}\mu_{i})^\dagger-(\alpha^2\mu_{i}+\beta^2\lambda_i)^\dagger)  &   -\diag((\alpha^2\mu_{i}+\beta^2\lambda_i)^\dagger)\\
    -\diag((\alpha^2\mu_{i}+\beta^2\lambda_i)^\dagger)&   \diag((\beta^2\lambda_i)^\dagger-(\alpha^2\mu_{i}+\beta^2\lambda_i)^\dagger)
\end{pmatrix}=\begin{pmatrix}
        D_1\\
        -D_2
    \end{pmatrix}  \begin{pmatrix}
        D_1 &
        -D_2
    \end{pmatrix},
\end{equation*}
and, hence, is positive semidefinite. This implies that
\begin{equation}\label{eq:D}
R=\overline{W}^\top\begin{pmatrix}
        D_1\\
        -D_2
    \end{pmatrix}  \begin{pmatrix}
        D_1 &
        -D_2
    \end{pmatrix}W
\end{equation} is positive semidefinite.
\end{proof}

\begin{corollary}\label{coro:AN}
Let $(G,\sigma)$ and $(G',\sigma')$ be two $U(d)$-connection graphs. Suppose that the two signature groups $\Sigma$ and $\Sigma'$ commute. For any $x\in V$, $x'\in V'$ and any $N,N'\in (0,\infty]$, the matrices $A_{N}(x), A_{N'}(x')$ and $A_{N+N'}(x,x')$ with respect to the canonical choice $B_0$ of matrix $B$ in (\ref{eq:canonicalB}) satisfy
\begin{align}\label{eq:ANdecomp}A_{N+N'}&(x,x')=
\begin{pmatrix}
    \alpha A_{N}(x) &   \\
                      & \beta A_{N'}(x')
\end{pmatrix}+R(x,x')+J(x,x'),
\end{align}
where $R(x,x')$ is given in (\ref{eq:R}) and
\begin{equation}\label{eq:J}
J(x,x'):=\frac{2}{N+N'}\begin{pmatrix}
\alpha\frac{N'}{N}\mathbf{v}_0(x)\overline{\mathbf{v}_0(x)}^\top & -\sqrt{\alpha\beta}\mathbf{v}_0(x)\overline{\mathbf{v}_0(x')}^\top\\
-\sqrt{\alpha\beta}\mathbf{v}_0(x')\overline{\mathbf{v}_0(x)}^\top & \beta\frac{N}{N'}\mathbf{v}_0(x')\overline{\mathbf{v}_0(x')}^\top
\end{pmatrix},
\end{equation}
with $\mathbf{v}_0$ defined in (\ref{eq:canonicalv0}).
Moreover, $J(x,x')$ is positive semidefinite.
\end{corollary}
\begin{proof}
The decomposition formula (\ref{eq:ANdecomp}) is a direct consequence of Theorem \ref{thm:Adecomp} and the definition of $A_N(x)=A_{\infty}(x)-\frac{2}{N}\mathbf{v}_0(x)\overline{\mathbf{v}_0(x)}^{\top}$. It remains to show that $J(x,x')$ is positive semidefinite. This follows from the following observation:
\begin{align}
&(N+N')J(x,x')\notag\\=&2\begin{pmatrix}
  \mathbf{v}_0(x)& \\
  & \mathbf{v}_0(x')
\end{pmatrix}\begin{pmatrix}
  \sqrt{\alpha\frac{N'}{N}}I_d\\
  -\sqrt{\beta\frac{N}{N'}}I_d
\end{pmatrix}\begin{pmatrix}
    \sqrt{\alpha\frac{N'}{N}}I_d &
  -\sqrt{\beta\frac{N}{N'}}I_d
\end{pmatrix}\begin{pmatrix}
   \overline{\mathbf{v}_0(x)}^\top& \\
  & \overline{\mathbf{v}_0(x')}^\top
\end{pmatrix}.\label{eq:Jdecomp}
\end{align}
\end{proof}

\begin{corollary}\label{coro:Cartesian}
Let $(G,\sigma)$ and $(G',\sigma')$ be two connection graphs. Suppose that the two signature groups $\Sigma$ and $\Sigma'$ commute. For any $x\in V$, $x'\in V'$, and $N, N'\in (0,\infty]$, we have
\begin{equation}\label{eq:coroCartesian}
\K_{G\times_{\alpha,\beta} G',\sigma^\times, (x,x')}(N+N')\geq\min\{\alpha\mathcal{K}_{G,\sigma,x}(N),\beta\mathcal{K}_{G',\sigma',x'}(N')\}.
\end{equation}
Moreover, if one of the two local connection structures $B_2^{inc}(x)$ and $B_2^{inc}(x')$ is balanced, we have
\begin{equation}\label{eq:coroCartesianEquality}
    \K_{G\times_{\alpha,\beta} G',\sigma^\times, (x,x')}(\infty)=\min\{\alpha\K_{G,\sigma,x}(\infty),\beta\K_{G',\sigma',x'}(\infty)\}.
\end{equation}
\end{corollary}
\begin{proof}
 By Corollary \ref{coro:AN}, we have the decomposition formula (\ref{eq:ANdecomp}).
Then, we derive from the fact that the matrices $R(x,x')$ and $J(x,x')$ are positive semidefinite and Theorem \ref{thm:curvature_eigenvalue} that
\begin{equation*}
\K_{G\times_{\alpha,\beta} G',\sigma^\times, (x,x')}(N+N')\geq \min\{\alpha\mathcal{K}_{G,\sigma,x}(N),\beta\mathcal{K}_{G',\sigma',x'}(N')\}.
\end{equation*}

If one of the two local connection structures $B_2^{inc}(x)$ and $B_2^{inc}(x')$ is balanced, we have
\[A_{\infty}(x,x')=
\begin{pmatrix}
   \alpha A_{\infty}(x) &   \\
                      &\beta A_{\infty}(x')
\end{pmatrix}.\]
Applying Theorem \ref{thm:curvature_eigenvalue} yields
\begin{equation*}
    \K_{G\times_{\alpha,\beta} G',\sigma^\times, (x,x')}(\infty)=\min\{\alpha\K_{G,\sigma,x}(\infty),\beta\K_{G',\sigma',x'}(\infty)\}.
\end{equation*}
\end{proof}

\begin{remark}
For two connection graphs $(G,\sigma)$ and $(G',\sigma')$ with connections of (possibly) different dimensions, we can define the Cartesian product of them as follows. Suppose that
\[\sigma: E^{or}\to O(d_1)\,\,\text{or}\,\,U(d_1), \,\,\text{and}\,\,\sigma':(E')^{or}\to O(d_2)\,\,\text{or}\,\,U(d_2)\]
for two positive integers $d_1$ and $d_2$. The Cartesian product of $(G,\sigma)$ and $(G',\sigma')$ is the connection graph $(G\times_{\alpha,\beta}G', \sigma^{\otimes})$ with a $d_1d_2$-dimensional connection $\sigma^\otimes$.
 The connection $\sigma^\otimes$ of each oriented edge in the Cartesian product of $G$ and $G'$ is given as below:
 \begin{align*}
     \sigma^\otimes_{(x,x')(y,x')}&= \sigma_{xy}\otimes I_{d_2},\\
    \sigma^\otimes_{(x,x')(x,y')}&=I_{d_1}\otimes\sigma'_{x'y'}.
 \end{align*}
 Notice that $(G\times_{\alpha,\beta}G', \sigma^{\otimes})$ coincides with the Cartesian product of $(G,\sigma\otimes I_{d_2})$ and $(G',I_{d_1}\otimes \sigma')$. This correspondence follows the sense of Definition \ref{def:Cartesian}.  The signature groups of $(G,\sigma\otimes I_{d_2})$ and $(G',I_{d_1}\otimes \sigma')$ always commute. Moreover, we have \cite[Corollary 3.4]{LMP}
 \[\mathcal{K}_{G,\sigma\otimes I_{d_2},x}=\mathcal{K}_{G,\sigma,x},\,\,\mathcal{K}_{G',I_{d_1}\otimes\sigma',x'}=\mathcal{K}_{G',\sigma',x'}.\]
 Therefore, the two estimates in Corollary \ref{coro:Cartesian} holds true for $(G\times_{\alpha,\beta}G', \sigma^{\otimes})$.
\end{remark}

We present an example involving two unbalanced connection graphs whose signature groups, $\Sigma$ and $\Sigma'$, do not commute. In this case, the curvature of their Cartesian product turns out to be smaller than the minimum of the two individual curvatures. This demonstrates that the condition on community of signature groups $\Sigma$ and $\Sigma'$ in Corollary \ref{coro:Cartesian} is indeed essential.
\begin{example}
Let $(G,\sigma)$ and $(G',\sigma')$ be given in Figure \ref{fig:local structure G} and Figure \ref{fig:local structure G'} with \[\sigma_{AC}=\begin{pmatrix}
    0 & i \\
    -i & 0
\end{pmatrix}\,\,\text{ and }\,\,\sigma'_{23}=
\begin{pmatrix}
    1 & 0 \\
    0 & i
\end{pmatrix},\] respectively. The edge weight of each edge and vertex measure at each vertex are set to be $1$. We will assume this choice of edge weights and vertex measure in all examples considered in this article.
\begin{figure}[!htp]
\begin{minipage}[t]{0.5\textwidth}
\centering
\tikzset{vertex/.style={circle, draw, fill=black!20, inner sep=0pt, minimum width=4pt}}
\begin{tikzpicture}[scale=3.0]

\draw (0,0) -- (0.75,0) node[midway, below, black]{$\sigma_{AC}$}
		 -- (0.375,0.375) node[midway, above, black]{$I_{2}$}
         -- (0,0) node[midway, above, black]{$I_{2}$};

\node at (0,0) [vertex, label={[label distance=0mm]225: \small $A$}, fill=black] {};
\node at (0.375,0.375) [vertex, label={[label distance=0mm]90: \small $B$}, fill=black] {};
\node at (0.75,0) [vertex, label={[label distance=0mm]270: \small $C$}, fill=black] {};

\end{tikzpicture}
\caption{The connection graph $(G,\sigma)$.}
\label{fig:local structure G}
\end{minipage}
\begin{minipage}[t]{0.5\textwidth}
\centering
\tikzset{vertex/.style={circle, draw, fill=black!20, inner sep=0pt, minimum width=4pt}}
\begin{tikzpicture}[scale=3.0]

\draw (0,0) -- (0.5,0.25) node[midway, above, black]{$I_2$}
		 -- (1,0) node[midway, above, black]{$I_2$};
\draw (0,0) -- (0.5,-0.25) node[midway, below, black]{$I_2$}
		 -- (1,0) node[midway, below, black]{$I_2$};
\draw (0.5,0.25) -- (0.5,-0.25) node[midway, right, black]{$\sigma'_{23}$};

\node at (0,0) [vertex, label={[label distance=0mm]225: \small $1$}, fill=black] {};
\node at (0.5,0.27) [vertex, below, label={[label distance=0mm]90: \small $2$}, fill=black] {};
\node at (0.5,-0.25) [vertex, label={[label distance=0mm]270: \small $3$}, fill=black] {};
\node at (1,0) [vertex, label={[label distance=0mm]270: \small $4$}, fill=black] {};
\end{tikzpicture}
\caption{The connection graph $(G',\sigma')$.}
\label{fig:local structure G'}
\end{minipage}
\end{figure}
We calculate that \[\mathcal{K}_{G,\sigma,A}(\infty)=\frac{1}{2},\,\,\text{ and }\,\,\K_{G',\sigma',1}=\frac{3}{2}.\] It is clearly that $\sigma_{AC}\sigma'_{23}\neq \sigma'_{23}\sigma_{AC}$. We check the curvature $\K_{G\times G',\sigma^\times, (A,1)}(\infty)<0$ (with $\alpha=\beta=1$). Indeed, the matrix $4\Gamma^{\sigma^\times}_{2}(A,1)$ has a negative eigenvalue $-0.7660<0$.
\end{example}

We now illustrate a case with two connection graphs that are each locally unbalanced, yet the curvature of their Cartesian product turns out to be the minimum of the two individual curvatures. This demonstrates that the criterion stated in Corollary \ref{coro:Cartesian}—namely, that at least one of the graphs must be locally balanced—is a sufficient condition for equality in (\ref{eq:coroCartesianEquality}), but not a required one.
\begin{example}\label{example:Cartesian_small}
Consider again the graphs $G$ and $G'$ depicted in Figure \ref{fig:local structure G} and Figure \ref{fig:local structure G'}, respectively. We assign $O(1)$-connections $\overline{\sigma}$ and $\overline{\sigma}'$ to $G$ and $G'$ by replacing each $I_2$ in $\sigma$ and $\sigma'$ by $+1$ and assigning $\overline{\sigma}_{AC}=-1$ and $\overline{\sigma}'_{23}=-1$.
\par We calculate the curvature
\[\K_{G,\overline{\sigma},A}(\infty)=\frac{1}{2},\,\,\text{and}\,\,\K_{G',\overline{\sigma}',1}(\infty)=\frac{3}{2}.\]
The local connection structure of $(A,1)$ of the Cartesian product of $(G,\overline{\sigma})$ and $(G',\overline{\sigma}')$ is depicted in Figure \ref{fig:cartesian product GG'}, in which edges with signature $+1$ (signature $-1$) are drawn as solid lines (dashed lines).
\begin{figure}[!htp]
\centering
\tikzset{vertex/.style={circle, draw, fill=black!20, inner sep=0pt, minimum width=4pt}}
\begin{tikzpicture}[scale=3.0]
\draw[ line width=0.025cm,
    dash pattern=on 0.1cm off 0.05cm] (1,0.375)  -- (1,0.125);
\draw (0,0) -- (1,0.375) node[midway, below, black]{};
\draw (1,0.125) -- (0,0) node[midway, above, black]{}
         -- (1,-0.125) node[midway, above, black]{}
         -- (1,-0.375) node[midway, above, black]{};
 \draw[ line width=0.025cm,
    dash pattern=on 0.1cm off 0.05cm] (1,-0.375) -- (0,0);
\draw (1,0.375) -- (2,0.5) node[midway, above, black]{}
         -- (1,0.125) node[midway, above, black]{};
 \draw[ line width=0.025cm,
    dash pattern=on 0.1cm off 0.05cm]  (1,0.125) -- (2,-0.5);
\draw (2,-0.5) -- (1,-0.375) node[midway, above, black]{}
         -- (2,0) node[midway, below, black]{};
\draw[ line width=0.025cm,
    dash pattern=on 0.1cm off 0.05cm] (2,0) -- (1,0.375);
\draw (1,0.375) -- (2,0.25) node[midway, above, black]{}
         -- (1,-0.125) node[midway, above, black]{}
         -- (2,-0.25) node[midway, above, black]{};

\node at (0,0) [vertex, label={[label distance=0mm]180: \small $(A,1)$}, fill=black] {};
\node at (1,0.375) [vertex, label={[label distance=0mm]180: \small $(A,2)$}, fill=black] {};
\node at (1,0.125) [vertex, label={[label distance=0mm]180: \small $(A,3)$}, fill=black] {};
\node at (1,-0.125) [vertex, label={[label distance=0mm]180: \small $(B,1)$}, fill=black] {};
\node at (1,-0.375) [vertex, label={[label distance=0mm]180: \small $(C,1)$}, fill=black] {};
\node at (2,0.5) [vertex, label={[label distance=0mm]0: \small $(A,4)$}, fill=black] {};
\node at (2,0.25) [vertex, label={[label distance=0mm]0: \small $(B,2)$}, fill=black] {};
\node at (2,0) [vertex, label={[label distance=0mm]0: \small $(C,2)$}, fill=black] {};
\node at (2,-0.25) [vertex, label={[label distance=0mm]0: \small $(B,3)$}, fill=black] {};
\node at (2,-0.5) [vertex, label={[label distance=0mm]0: \small $(C,3)$}, fill=black] {};

\end{tikzpicture}
\caption{The local connection structure of $(A,1)$ in $(G,\overline{\sigma})\times(G',\overline{\sigma}')$.}
\label{fig:cartesian product GG'}
\end{figure}
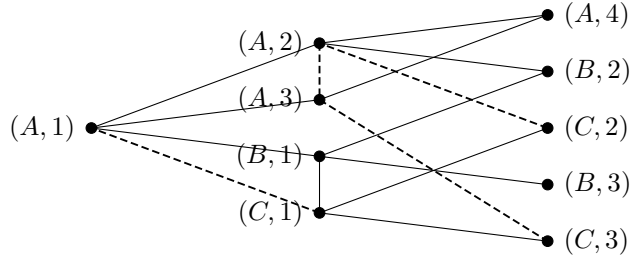

We check that the curvature at the vertex $(A,1)$ is $\frac{1}{2}$, which is equal to the minimum of $\K_{G,\overline{\sigma}, A}(\infty)$ and $ \K_{G',\overline{\sigma}',1}(\infty)$.
\end{example}
In the following, we present an example of two signed graphs, for which the curvature of the Cartesian product lies strictly in between the two individual curvatures.
\begin{example}
Consider the signed graph $(G,\overline{\sigma})$ given in Example \ref{example:Cartesian_small}, and the signed graph $(G_2,\sigma_2)$ given in Figure \ref{fig:local structure G2}. Notice that we only label the negative edges in the Figure \ref{fig:local structure G2} for simplicity.
\begin{figure}[!htp]
\centering
\tikzset{vertex/.style={circle, draw, fill=black!20, inner sep=0pt, minimum width=4pt}}
\begin{tikzpicture}[scale=3.0]

\draw (0,0) -- (0.5,0.5) node[midway, below, black]{}
		 -- (1,0.5) node[midway, above, black]{}
         -- (0.5,-0.5) node[midway, above, black]{}
         -- (0,0) node[midway, above, black]{};
\draw (0,0) -- (0.5,0) node[midway, above, black]{}
		 -- (0.5,0.5) node[midway, right, black]{$-$};
\draw (0.5,0) -- (1,-0.5) node[midway, right, black]{};

\node at (0,0) [vertex, label={[label distance=0mm]225: \small $1$}, fill=black] {};
\node at (0.5,0.5) [vertex, above, label={[label distance=0mm]90: \small $2$}, fill=black] {};
\node at (0.5,0) [vertex, label={[label distance=0mm]135: \small $3$}, fill=black] {};
\node at (0.5,-0.5) [vertex, label={[label distance=0mm]90: \small $4$}, fill=black] {};
\node at (1,0.5) [vertex, label={[label distance=0mm]90: \small $5$}, fill=black] {};
\node at (1,-0.5) [vertex, label={[label distance=0mm]225: \small $6$}, fill=black] {};

\end{tikzpicture}
\caption{The signed graph $(G_{2},\sigma_{2})$.}
\label{fig:local structure G2}
\end{figure}
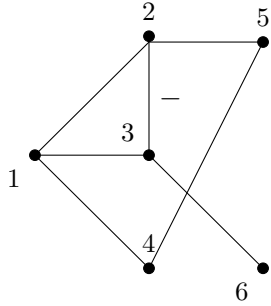

Then the curvature matrix at the vertex $1$ is \[A_{\infty}(G_2,\sigma_2,1)=\frac{1}{4}\begin{pmatrix}
    5 & 3 & 0\\
    3 & 1 & 4 \\
    0 & 4 & 6
\end{pmatrix}.\]Thus, the curvature is $\K_{G_{2},\sigma_{2},1}(\infty)\approx -0.5502$. Recall that the curvature at the vertex $A$ is $\K_{G_,\overline{\sigma}, A}=\frac{1}{2}$.

The local structure of the Cartesian product at the vertex $(1,A)$ is shown in Figure \ref{fig:cartesian product G2G}. Again, for clarity, the edges with signature $+1$ (signature $-1$) are drawn as solid lines (dashed lines).
\begin{figure}[!ht]
\centering
\tikzset{vertex/.style={circle, draw, fill=black!20, inner sep=0pt, minimum width=4pt}}
\begin{tikzpicture}[scale=3.0]

\draw (0,0) -- (1,0.5) node[midway, below, black]{}
		 -- (1,0.25) node[midway, left, black]{};
\draw[ line width=0.025cm,
    dash pattern=on 0.1cm off 0.05cm] (1,0.25) -- (0,0);
\draw (0,0) -- (1,0) node[midway, above, black]{};
\draw[ line width=0.025cm,
    dash pattern=on 0.1cm off 0.05cm] (1,0) -- (1,-0.25);
\draw (1,-0.25) -- (0,0) node[midway, above, black]{};
\draw (0,0) -- (1,-0.5) node[midway, above, black]{};
\draw[ line width=0.025cm,
    dash pattern=on 0.1cm off 0.05cm](1,-0.5)  -- (2,-0.875);
\draw (2,-0.875) -- (1,0.25) node[midway,below, black]{}
         -- (2,0.875) node[midway, above, black]{};
\draw[ line width=0.025cm,
    dash pattern=on 0.1cm off 0.05cm] (2,0.875)  -- (1,-0.25);
\draw (1,0.5) -- (2,0.625) node[midway, above, black]{}
         -- (1,-0.25) node[midway, above, black]{};
 \draw[ line width=0.025cm,
    dash pattern=on 0.1cm off 0.05cm]  (1,0) -- (2,0.375);
\draw (2,0.375) -- (1,0.25) node[midway, above, black]{}
            -- (2,-0.875) node[midway, above, black]{};
\draw (1,0.5) -- (2,0.125) node[midway, above, black]{}
              -- (1,0) node[midway, above, black]{};
\draw (1,0) -- (2,-0.125) node[midway, above, black]{}
            -- (1,-0.5) node[midway, above, black]{};
\draw (1,-0.25) -- (2,-0.375) node[midway, above, black]{};
\draw (1,0.5) -- (2,-0.625) node[midway, above, black]{}
              -- (1,-0.5) node[midway, above, black]{};

\node at (0,0) [vertex, label={[label distance=0mm]180: \small $(1,A)$}, fill=black] {};
\node at (1,0) [vertex, label={[label distance=0mm]180: \small $(2,A)$}, fill=black] {};
\node at (1,0.25) [vertex, label={[label distance=0mm]180: \small $(1,C)$}, fill=black] {};
\node at (1,0.5) [vertex, label={[label distance=0mm]180: \small $(1,B)$}, fill=black] {};
\node at (1,-0.25) [vertex, label={[label distance=0mm]180: \small $(3,A)$}, fill=black] {};
\node at (1,-0.5) [vertex, label={[label distance=0mm]180: \small $(4,A)$}, fill=black] {};
\node at (2,0.875) [vertex, label={[label distance=0mm]0: \small $(3,C)$}, fill=black] {};
\node at (2,0.625) [vertex, label={[label distance=0mm]0: \small $(3,B)$}, fill=black] {};
\node at (2,0.375) [vertex, label={[label distance=0mm]0: \small $(2,C)$}, fill=black] {};
\node at (2,0.125) [vertex, label={[label distance=0mm]0: \small $(2,B)$}, fill=black] {};
\node at (2,-0.125) [vertex, label={[label distance=0mm]0: \small $(5,A)$}, fill=black] {};
\node at (2,-0.375) [vertex, label={[label distance=0mm]0: \small $(6,A)$}, fill=black] {};
\node at (2,-0.625) [vertex, label={[label distance=0mm]0: \small $(4,B)$}, fill=black] {};
\node at (2,-0.875) [vertex, label={[label distance=0mm]0: \small $(4,C)$}, fill=black] {};
\end{tikzpicture}
\caption{The local connection structure of $(1,A)$ in $(G_2,\sigma_2)\times (G,\overline{\sigma})$.}
\label{fig:cartesian product G2G}
\end{figure}
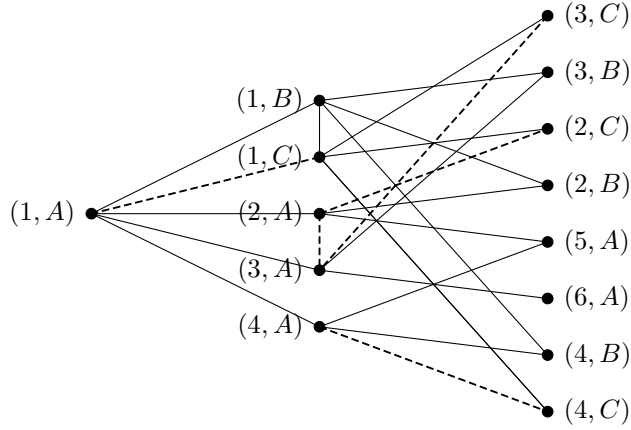

The curvature matrix at the vertex $(1,A)$ is \[A_{\infty}(1,A)=\frac{1}{8}
\begin{pmatrix}
    19 & -15 & -9 & -9 & 0 \\
    -15 & 19 & 9 & 9 & 0 \\
    -9 & 9 & 19 & 15 & 0 \\
    -9 & 9 & 15 & 11 & 8 \\
    0 & 0 & 0 & 8 & 12
\end{pmatrix}.
 \]Hence the curvature is about $-0.454$, which satisfies $$-0.5502\approx \textrm{min}\{\K_{G_2,\sigma_2,1},\K_{G,\overline{\sigma},A}\}< -0.454<\textrm{max}\{\K_{G_2,\sigma_2,1},\K_{G,\overline{\sigma},A}\}=\frac{1}{2}.$$
\end{example}
Next, we restrict to the case that the connections lie in $O(1)$ or $U(1)$.
\begin{corollary}\label{coro:U1infty}
Let $(G,\sigma)$ and $(G',\sigma')$ be two connection graphs with $1$-dimensional connections.  For $x\in V$ and $x'\in V'$, we have
\begin{equation*}
\min\{\alpha\mathcal{K}_{G,\sigma,x}(\infty),\beta\mathcal{K}_{G',\sigma',x'}(\infty)\}\leq\K_{G\times_{\alpha,\beta} G',\sigma^\times, (x,x')}(\infty)\leq\max\{\alpha\mathcal{K}_{G,\sigma,x}(\infty),\beta\mathcal{K}_{G',\sigma',x'}(\infty)\}.
\end{equation*}
\end{corollary}
\begin{proof}
Since the signature groups always commute in the case of $d=1$, the lower bound estimate follows directly from Corollary \ref{coro:Cartesian}. Next, we show the upper bound estimate. Let $u$ and $u'$ be the eigenvectors with unit norm corresponding to  $\lambda_{\min}(A_{\infty}(x))$ and $\lambda_{\min}(A_{\infty}(x'))$, respectively. Let  $c$ and $c'$ be two numbers to be determined. Applying Theorem \ref{thm:Adecomp}, we estimate
$$\lambda_{\min}(A_{\infty}(x,x'))\leq\frac{
\begin{pmatrix}
    \overline{cu}^{\top} & \overline{c'u'}^{\top}
\end{pmatrix}
\begin{pmatrix}
    \alpha A_{\infty}(x) &   \\
                      & \beta A_{\infty}(x')
\end{pmatrix}
\begin{pmatrix}
    cu \\
    c'u'
\end{pmatrix}
+
\begin{pmatrix}
    \overline{cu}^{\top} & \overline{c'u'}^{\top}
\end{pmatrix}
R
\begin{pmatrix}
    cu \\
    c'u'
\end{pmatrix}
}{|c|^2+|c'|^2}.
$$
Applying (\ref{eq:D}), we derive
\begin{equation*}
 \begin{pmatrix}
    \overline{cu}^{\top} & \overline{c'u'}^{\top}
\end{pmatrix}
R
\begin{pmatrix}
    cu \\
    c'u'
\end{pmatrix}=   \begin{pmatrix}
    \overline{cu}^{\top} & \overline{c'u'}^{\top}
\end{pmatrix}
\overline{W}^\top\begin{pmatrix}
        D_1\\
        -D_2
    \end{pmatrix}  \begin{pmatrix}
        D_1 &
        -D_2
    \end{pmatrix}W
\begin{pmatrix}
    cu \\
    c'u'
\end{pmatrix}=\overline{\zeta}^\top\zeta,
\end{equation*}
where 
\[\zeta:=\begin{pmatrix}
        D_1 &
        -D_2
    \end{pmatrix}W
\begin{pmatrix}
    cu \\
    c'u'
\end{pmatrix}=c\alpha^{\frac{3}{2}}D_1P\omega^\top u-c'\beta^{\frac{3}{2}}D_2P\omega'^\top.\]
Notice that $\zeta$ is a number due to $d=1$. We choose $c$ and $c'$ such that $\zeta=0$.

Then we estimate
\begin{align*}
\lambda_{\min}(A_{\infty}(x,x'))
\leq &\frac{1}{|c|^{2}+|c'|^{2}}\left(|c|^{2}\alpha\lambda_{\min}(A_{\infty}(x))+|c'|^{2}\beta\lambda_{\min}(A_{\infty}(x'))\right)\\
\leq&\max\{\alpha\lambda_{\min}(A_{\infty}(x)),\beta\lambda_{\min}(A_{\infty}(x'))\}.
\end{align*}
This completes the proof.
\end{proof}
%

Let us recall the following definition of \emph{star product} of functions from \cite{CLP}.
\begin{definition}\cite[Definition 7.1]{CLP}\label{def:star_product}
Let $f_1,f_2: (0,\infty]\to \mathbb{R}$ be continuous and monotone increasing functions with $\lim_{t\to 0}f_i(t)=-\infty$, $i=1,2$. Then the star product of $f_1$ and $f_2$ is defined as a function $f_1*f_2:(0,\infty]\to \mathbb{R}$ \label{notation:star_product}given by
\[f_1*f_2(t):=f_1(t_1)=f_2(t_2),\]
where $t_1+t_2=t$ such that $f_1(t_1)=f_2(t_2)$.
\end{definition}

Notice that the star product is commutative and associative \cite[Propositions 7.5 and 7.6]{CLP}.

\begin{corollary}\label{coro:U1N}
Let $(G,\sigma)$ and $(G',\sigma')$ be two connection graphs with $1$-dimensional connections. Let $x\in V$, $x'\in V'$ be two vertices. Assume that one of the two local connection structures $B_2^{inc}(x)$ and $B_2^{inc}(x')$ is balanced. For any $N,N'\in (0,\infty]$, we have
\begin{equation*}
\min\{\alpha\mathcal{K}_{G,\sigma,x}(N),\beta\mathcal{K}_{G',\sigma',x'}(N')\}\leq\K_{G\times_{\alpha,\beta} G',\sigma^\times, (x,x')}(N+N')\leq\max\{\alpha\mathcal{K}_{G,\sigma,x}(N),\beta\mathcal{K}_{G',\sigma',x'}(N')\}.
\end{equation*}
Consequently, the curvature functions in this case satisfy
\begin{equation*}
\K_{G\times_{\alpha,\beta} G',\sigma^\times, (x,x')}=\left(\alpha\mathcal{K}_{G,\sigma,x}\right)* \left(\beta\mathcal{K}_{G',\sigma',x'}\right).
\end{equation*}
\end{corollary}
\begin{proof}
  The lower bound follows from Corollary \ref{coro:Cartesian} since the signature groups commute. Due to one of the two local connection structures $B_2^{inc}(x)$ and $B_2^{inc}(x')$ being balanced, we have $R(x,x')=\mathbf{0}$ in (\ref{eq:ANdecomp}). Let $u$ and $u'$ be the eigenvectors with unit norm corresponding to  $\lambda_{\min}(A_{N}(x))$ and $\lambda_{\min}(A_{N'}(x'))$, respectively. Let $c$ and $c'$ be two numbers such that
  \[\begin{pmatrix}
    \overline{cu}^{\top} & \overline{c'u'}^{\top}
\end{pmatrix}
J(x,x')
\begin{pmatrix}
    cu \\
    c'u'
\end{pmatrix}=0.\]
Due to the decomposition formula (\ref{eq:Jdecomp}) for $J(x,x')$, the existence of $c$ and $c'$ can be shown by a similar argument as in the proof of Corollary \ref{coro:U1infty}.
Then we estimate
\begin{align*}
\lambda_{\min}(A_{N+N'}(x,x'))
\leq &\frac{1}{|c|^{2}+|c'|^{2}}\left(|c|^{2}\alpha\lambda_{\min}(A_{N}(x))+|c'|^{2}\beta\lambda_{\min}(A_{N'}(x'))\right)\\
\leq&\max\{\alpha\lambda_{\min}(A_{N}(x)),\beta\lambda_{\min}(A_{N'}(x'))\}.
\end{align*}
This proves the curvature upper bound estimate. The identity of curvature functions follows by applying Theorem \ref{thm:curvature_function} $(i)$ and \cite[Proposition 7.3]{CLP}.
\end{proof}
\section{How the curvature changes under local operations}\label{section:operation}
In this section, we study how the curvature at a given vertex changes under operations on the local connection structure $B_2^{inc}$. We discuss two operations:
\begin{itemize}
\item[(i)] Add a new spherical edge in $S_1(x)$;
\item[(ii)] Merge two vertices in $S_2(x)$ which have no common neighbors.
\end{itemize}
\begin{theorem}\label{thm:sphericaloperation}
Let $(G,\sigma)$ be a connection graph and $x\in V$ be a given vertex. Assume that $x$ is $S_1$-in regular, i.e., $p_{yx}$ is independent of $y\in S_1(x)$. Suppose $y_1,y_2\in S_1(x)$ are non-adjacent, i.e., $w_{y_1y_2}=0$. Denote by $\widetilde{G}$ the weighted graph obtained from $G$ by assigning a positive $\tilde{w}_{y_1y_2}>0$ and keeping other edge weights. The new connection $\tilde{\sigma}$ coincides with $\sigma$ except that
\begin{equation}\label{eq:sigmaassumption}\tilde{\sigma}_{y_1y_2}:=\sigma_{y_1x}\sigma_{xy_2},\,\,\text{and}\,\,\sigma_{y_2y_1}:=\sigma_{y_1y_2}^{-1}.\end{equation}
That is, we add a spherical edge in $S_1(x)$ which produces a balanced triangle.
Then, we have for any $N\in (0,\infty]$,
\[\mathcal{K}_{\widetilde{G},\tilde{\sigma},x}(N)\geq \mathcal{K}_{G,\sigma,x}(N).\]
\end{theorem}
\begin{proof}
First, observe that the matrices $\Delta^{\sigma}(x)$ and $\Gamma^{\sigma}(x)$ stay put under the operation of adding spherical edges. We only need to consider how the matrix $\Gamma^{\sigma}_{2}(x)$ changes. By Proposition \ref{prop:Switching}, the curvature is unchanged under switching the connections. We can switch all the connections of the edges $\{(x,y_i): i=1,\ldots,m\}$ to be $I_d$ by a switching function $\tau$ with $\tau(x)=I_d$ and $\tau(y_i)=\sigma_{xy_i}^{-1}$, $i=1,\ldots, m$.
By \eqref{eq:gamma 2 matrix}-\eqref{eq:Gamma_2_last}, the difference matrix is
$$4\Gamma^{\tilde{\sigma}}_{2}(x)-4\Gamma^{\sigma}_{2}(x)=
\begin{pmatrix}
    \mathbf{0}_d      & c_2\overline{\tilde{\sigma}_{y_2y_1}}-c_1I_d      &  c_1\overline{\tilde{\sigma}_{y_1y_2}}-c_2I_d     & \dots  & \mathbf{0}_d        \\
    c_2\tilde{\sigma}_{y_2y_1}^\top-c_1I_d      & 3c_1I_d+c_2I_d      & -2(c_1+c_2)\overline{\tilde{\sigma}_{y_1y_2}}     & \dots  & \mathbf{0}_d        \\
    c_1\tilde{\sigma}_{y_1y_2}^\top-c_2I_d      & -2(c_1+c_2)\overline{\tilde{\sigma}_{y_2y_1}}      &  3c_2I_d+c_1I_d      & \dots  & \mathbf{0}_d        \\
    \vdots & \vdots & \vdots & \ddots & \vdots \\
    \mathbf{0}_d       & \mathbf{0}_d        & \mathbf{0}_d       & \dots  & \mathbf{0}_d
\end{pmatrix}
,$$
where we use the notations $c_1:=p_{xy_1}\tilde{p}_{y_1y_2}$ and $c_2:=p_{xy_2}\tilde{p}_{y_2y_1}$.
By the $S_1(x)$-in regularity at $x$, we have
\[c_1=p_{xy_1}\tilde{p}_{y_1y_2}=\frac{w_{xy_1}\tilde{w}_{y_1y_2}}{\mu(x)\mu(y_1)}=\frac{\tilde{w}_{y_1y_2}}{\mu(x)}p_{y_1x}=\frac{\tilde{w}_{y_1y_2}}{\mu(x)}p_{y_2x}=c_2.\]

 By our assumption (\ref{eq:sigmaassumption}), the switching function $\tau$ will switch $\tilde{\sigma}_{y_1y_2}$ to be $I_d$ too. Then, the difference matrix becomes
$$4\Gamma^{\tilde{\sigma}}_{2}(x)-4\Gamma^{\sigma}_{2}(x)=
c_1\begin{pmatrix}
    \mathbf{0}_d      & \mathbf{0}_d     &  \mathbf{0}_d     & \dots  & \mathbf{0}_d        \\
    \mathbf{0}_d      & 4I_d      & -4I_d    & \dots  & \mathbf{0}_d        \\
    \mathbf{0}_d      & -4I_d    &  4I_d      & \dots  & \mathbf{0}_d        \\
    \vdots & \vdots & \vdots & \ddots & \vdots \\
    \mathbf{0}_d       & \mathbf{0}_d        & \mathbf{0}_d       & \dots  & \mathbf{0}_d
\end{pmatrix},$$
which is clearly positive semidefinite.
\end{proof}
If, instead, we add a new spherical edge between $y_1$ and $y_2$ with a connection $\tilde{\sigma}_{y_1y_2}$ which produces an \emph{unbalanced} triangle, the difference matrix, after switching all connections $\sigma_{xy_i}$, $i=1,\ldots,m$ to be $I_d$, is
\begin{equation}\label{eq:unbalacneddifference}4\Gamma^{\tilde{\sigma}}_{2}(x)-4\Gamma^{\sigma}_{2}(x)=
c_1\begin{pmatrix}
    \mathbf{0}_d      & \overline{\tilde{\sigma}_{y_2y_1}}-I_d      &  \overline{\tilde{\sigma}_{y_1y_2}}-I_d     & \dots  & \mathbf{0}_d        \\
    \tilde{\sigma}_{y_2y_1}^\top-I_d      & 4I_d      & -4\overline{\tilde{\sigma}_{y_1y_2}}     & \dots  & \mathbf{0}_d        \\
    \tilde{\sigma}_{y_1y_2}^\top-I_d      & -4\overline{\tilde{\sigma}_{y_2y_1}}      &  4I_d      & \dots  & \mathbf{0}_d        \\
    \vdots & \vdots & \vdots & \ddots & \vdots \\
    \mathbf{0}_d       & \mathbf{0}_d        & \mathbf{0}_d       & \dots  & \mathbf{0}_d
\end{pmatrix},\end{equation}
which is indefinite. Indeed, the unbalancedness of the new triangle is equivalent to the fact that $\tilde{\sigma}_{y_1y_2}\neq I_d$ in the above matrix. That is, there exists $v\in \mathbb{K}^d$ such that $\tilde{\sigma}_{y_2y_1}^\top\overline{v}=\lambda\overline{v}$ where $\lambda\in \mathbb{K}$ satisfying $\mathrm{Re}(\lambda)<1$. Then, we calculate for $c\in\mathbb{R}$
\[\begin{pmatrix}
  cv^\top & v^\top
\end{pmatrix}\begin{pmatrix}
  \mathbf{0}_d & \overline{\tilde{\sigma}_{y_2y_1}}-I_d \\
  \tilde{\sigma}_{y_2y_1}^\top-I_d      & 4I_d
\end{pmatrix}\begin{pmatrix}
  c\overline{v}\\\overline{v}
\end{pmatrix}=2\left(c(\mathrm{Re}(\lambda)-1)+2\right)v^\top\overline{v},\]
which is negative for large enough $c$. This implies that the matrix (\ref{eq:unbalacneddifference}) is indefinite.

 Below, we give three examples of adding a spherical edge which produces an unbalanced triangle, for which the curvature increases, stays put and decreases, respectively. The connection graphs in those examples are signed graphs, i.e., graphs with $O(1)$-connections. For simplicity, we only label the edges with $-1$ in the figures and the unlabeled edges have the connection $+1$.
\begin{example}
Let $(G_{3},\sigma_3)$ be the signed graph given in Figure \ref{fig:G3}.
\begin{figure}[!htp]
\begin{minipage}[t]{0.5\textwidth}
\centering
\tikzset{vertex/.style={circle, draw, fill=black!20, inner sep=0pt, minimum width=4pt}}
\begin{tikzpicture}[scale=3.0]

\draw (0,0) -- (0.5,0.5) node[midway, below, black]{}
		 -- (1,0.5) node[midway, above, black]{$-$}
         -- (0.5,-0.5) node[midway, above, black]{}
         -- (0,0) node[midway, above, black]{};
\draw (0,0) -- (0.5,0) node[midway, above, black]{}
		 -- (0.5,0.5) node[midway, right, black]{$-$};
\draw (0.5,0) -- (1,-0.5) node[midway, right, black]{};

\node at (0,0) [vertex, label={[label distance=0mm]225: \small $1$}, fill=black] {};
\node at (0.5,0.5) [vertex, above, label={[label distance=0mm]90: \small $2$}, fill=black] {};
\node at (0.5,0) [vertex, label={[label distance=0mm]135: \small $3$}, fill=black] {};
\node at (0.5,-0.5) [vertex, label={[label distance=0mm]90: \small $4$}, fill=black] {};
\node at (1,0.5) [vertex, label={[label distance=0mm]90: \small $5$}, fill=black] {};
\node at (1,-0.5) [vertex, label={[label distance=0mm]225: \small $6$}, fill=black] {};

\end{tikzpicture}
\caption{The signed graph $(G_{3},\sigma_{3})$.}
\label{fig:G3}
\end{minipage}
\begin{minipage}[t]{0.5\textwidth}
\centering
\tikzset{vertex/.style={circle, draw, fill=black!20, inner sep=0pt, minimum width=4pt}}
\begin{tikzpicture}[scale=3.0]

\draw (0,0) -- (0.5,0.5) node[midway, below, black]{}
		 -- (1,0.5) node[midway, above, black]{$-$}
         -- (0.5,-0.5) node[midway, above, black]{}
         -- (0,0) node[midway, above, black]{};
\draw (0,0) -- (0.5,0) node[midway, above, black]{}
		 -- (0.5,0.5) node[midway, right, black]{$-$};
\draw (0.5,0) -- (1,-0.5) node[midway, right, black]{};
\draw (0.5,0) -- (0.5,-0.5) node[midway, left, black]{$-$};

\node at (0,0) [vertex, label={[label distance=0mm]225: \small $1$}, fill=black] {};
\node at (0.5,0.5) [vertex, label={[label distance=0mm]90: \small $2$}, fill=black] {};
\node at (0.5,0) [vertex, label={[label distance=0mm]135: \small $3$}, fill=black] {};
\node at (0.5,-0.5) [vertex, above, label={[label distance=0mm]270: \small $4$}, fill=black] {};
\node at (1,0.5) [vertex, label={[label distance=0mm]90: \small $5$}, fill=black] {};
\node at (1,-0.5) [vertex, label={[label distance=0mm]225: \small $6$}, fill=black] {};

\end{tikzpicture}
\caption{The signed graph $(\widetilde{G}_3,\tilde{\sigma}_3)$.}
\label{fig:G3tilde}
\end{minipage}
\end{figure}
%
%
%

By definition, the curvature matrix at the vertex $1$ is \[A_{\infty}(G_3,\sigma_3,1)=\frac{1}{5}
\begin{pmatrix}
    3 & 6 & 12 \\
    6 & 7 & 4 \\
    12 & 4 & 13
\end{pmatrix}.\] Hence the curvature is $\mathcal{K}_{G_3,\sigma_{3},1}(\infty)\approx -0.569$.
\par We add a negative edge between the vertices $3$ and $4$ and obtain the signed graph $(\widetilde{G}_3,\tilde{\sigma}_3)$ shown in Figure \ref{fig:G3tilde}.
%
%
%
Then we have the curvature matrix \[A_{\infty}(\widetilde{G}_3,\tilde{\sigma}_3,1)=\frac{1}{9}
\begin{pmatrix}
    31 & 6 & 4 \\
    6 & 9 & 6 \\
    4 & 6 & 31
\end{pmatrix},\]
and, therefore, the curvature $\mathcal{K}_{\widetilde{G}_{3},\tilde{\sigma}_{3},1}(\infty)\approx 0.36>\mathcal{K}_{G_3,\sigma_{3},1}(\infty)$.
\end{example}

\begin{example}
Consider the signed graph $(G_{4},\sigma_{4})$ given in Figure \ref{fig:G4}. The signed graph $(\widetilde{G}_4,\tilde{\sigma}_4)$ in Figure \ref{fig:G4tilde} is obtained by adding a negative edge between the vertices $2$ and $3$.
\begin{figure}[!htp]
\begin{minipage}[t]{0.5\textwidth}
\centering
\tikzset{vertex/.style={circle, draw, fill=black!20, inner sep=0pt, minimum width=4pt}}
\begin{tikzpicture}[scale=3.0]

\draw (0,0) -- (0.5,0.5) node[midway, below, black]{}
		 -- (1,0.5) node[midway, above, black]{};
\draw (0,0) -- (0.5,0) node[midway, above, black]{}
		 -- (1,0) node[midway, right, black]{};

\node at (0,0) [vertex, label={[label distance=0mm]225: \small $1$}, fill=black] {};
\node at (0.5,0.5) [vertex, label={[label distance=0mm]90: \small $2$}, fill=black] {};
\node at (0.5,0) [vertex, label={[label distance=0mm]270: \small $3$}, fill=black] {};
\node at (1,0.5) [vertex, label={[label distance=0mm]90: \small $4$}, fill=black] {};
\node at (1,0) [vertex, label={[label distance=0mm]270: \small $5$}, fill=black] {};

\end{tikzpicture}
\caption{The signed graph $(G_4,\sigma_4)$}
\label{fig:G4}
\end{minipage}
\begin{minipage}[t]{0.5\textwidth}
\centering
\tikzset{vertex/.style={circle, draw, fill=black!20, inner sep=0pt, minimum width=4pt}}
\begin{tikzpicture}[scale=3.0]

\draw (0,0) -- (0.5,0.5) node[midway, below, black]{}
		 -- (1,0.5) node[midway, above, black]{};
\draw (0,0) -- (0.5,0) node[midway, above, black]{}
		 -- (1,0) node[midway, right, black]{};
\draw (0.5,0) -- (0.5,0.5) node[midway, right, black]{$-$};

\node at (0,0) [vertex, label={[label distance=0mm]225: \small $1$}, fill=black] {};
\node at (0.5,0.5) [vertex, label={[label distance=0mm]90: \small $2$}, fill=black] {};
\node at (0.5,0) [vertex, label={[label distance=0mm]270: \small $3$}, fill=black] {};
\node at (1,0.5) [vertex, label={[label distance=0mm]90: \small $4$}, fill=black] {};
\node at (1,0) [vertex, label={[label distance=0mm]270: \small $5$}, fill=black] {};

\end{tikzpicture}
\caption{The signed graph $(\widetilde{G}_4,\tilde{\sigma}_4)$}
\label{fig:G4tilde}
\end{minipage}
\end{figure}
%
We check directly that
$\K_{G_{4},\sigma_{4},1}(\infty)=0=\mathcal{K}_{\widetilde{G}_4,\tilde{\sigma}_4,1}(\infty)$. That is, the curvature stays put.
%
%
%
\end{example}

\begin{example}
Consider the signed graph $(G_{5},\sigma_{5})$ given in Figure \ref{fig:G5}.
\begin{figure}[!htp]
\begin{minipage}[t]{0.5\textwidth}
\centering
\tikzset{vertex/.style={circle, draw, fill=black!20, inner sep=0pt, minimum width=4pt}}
\begin{tikzpicture}[scale=3.0]

\draw (0,0) -- (0.5,0.5) node[midway, below, black]{}
		 -- (1,0) node[midway, above, black]{}
         -- (0.5,-0.5) node[midway, above, black]{}
         -- (0,0) node[midway, above, black]{};

\node at (0,0) [vertex, label={[label distance=0mm]225: \small $1$}, fill=black] {};
\node at (0.5,0.5) [vertex, label={[label distance=0mm]90: \small $2$}, fill=black] {};
\node at (1,0) [vertex, label={[label distance=0mm]270: \small $4$}, fill=black] {};
\node at (0.5,-0.5) [vertex, label={[label distance=0mm]270: \small $3$}, fill=black] {};

\end{tikzpicture}
\caption{The signed graph $(G_5,\sigma_5)$}
\label{fig:G5}
\end{minipage}
\begin{minipage}[t]{0.5\textwidth}
\centering
\tikzset{vertex/.style={circle, draw, fill=black!20, inner sep=0pt, minimum width=4pt}}
\begin{tikzpicture}[scale=3.0]

\draw (0,0) -- (0.5,0.5) node[midway, below, black]{}
		 -- (1,0) node[midway, above, black]{}
         -- (0.5,-0.5) node[midway, above, black]{}
         -- (0,0) node[midway, above, black]{};
\draw (0.5,0.5) -- (0.5,-0.5) node[midway, right, black]{$-$};

\node at (0,0) [vertex, label={[label distance=0mm]225: \small $1$}, fill=black] {};
\node at (0.5,0.5) [vertex, label={[label distance=0mm]90: \small $2$}, fill=black] {};
\node at (1,0) [vertex, label={[label distance=0mm]270: \small $4$}, fill=black] {};
\node at (0.5,-0.5) [vertex, label={[label distance=0mm]270: \small $3$}, fill=black] {};

\end{tikzpicture}
\caption{The signed graph $(\widetilde{G}_5,\tilde{\sigma}_5)$}
\end{minipage}
\end{figure}
%
%
%
We have the curvature $\K_{G_{5},\sigma_{5},1}=2$. We add a negative edge between the vertices $2$ and $3$.
%
%
%
The curvature becomes $\K_{\widetilde{G}_{5},\tilde{\sigma}_{5},1}(\infty)=\frac{3}{2}<\K_{G_5,\sigma_5,1}(\infty)$.
\end{example}

\par In the following, we study the operation of \emph{merging two vertices} in the 2-sphere $S_{2}(x)$ which have no common neighbors. We define the operation of merging two vertices $z_k,z_\ell$ with no common neighbors in $(G,\sigma)$ as follows. Derive from $(G,\sigma)$ a new connection graph $(G',\sigma')$: Identify the two vertices $z_k$ and $z_\ell$ as a new vertex $z$, i.e., set $V':=(V\setminus\{z_k,z_\ell\})\bigcup\{z\}$. We further set
$w'_{uz}:=w_{uz_k}+w_{uz_\ell}$, $w'_{uv}=w_{uv}$, $\mu'(z)=\mu(z_k)+\mu(z_\ell)$ and $\mu'(u)=\mu(u)$ for any $u,v\in V\setminus\{z_k,z_\ell\}$. The connection $\sigma'$ is given by
\[(\sigma'_{zu})^{-1}=\sigma'_{uz}:=\sigma_{uz_k}, \,\,\text{for any}\,\,u\in V\setminus\{z_k,z_\ell\}\,\,\text{with}\,\,w_{uz_k}\neq 0,\]
\[(\sigma'_{zu})^{-1}=\sigma'_{uz}:=\sigma_{uz_\ell}, \,\,\text{for any}\,\,u\in V\setminus\{z_k,z_\ell\}\,\,\text{with}\,\,w_{uz_\ell}\neq 0,\]
and $\sigma'_{uv}=\sigma_{uv}$ for any $u,v\in V\setminus\{z_k,z_\ell\}$ with $w'_{uv}\neq 0$.
\begin{theorem}\label{thm:mergeoperation}
Let $(G,\sigma)$ be a connection graph and $x$ be a given vertex. Let $z_k,z_\ell$ be two vertices in $S_2(x)$ without common neighbors and $(G',\sigma')$ be the new connection graph obtained by merging the two vertices $z_{k}$ and $z_{\ell}$. Then, we have for any $N\in (0,\infty]$
$$\K_{G',\sigma',x}(N)\geq\K_{G,\sigma,x}(N).$$
\end{theorem}
\begin{proof}
Without loss of generality, we assume that the two vertices to be merged are indexed as $z_{n}$ and $z_{n-1}$ in $S_2(x)$. Recall we denote $S_2(x)=\{z_1,\ldots,z_n\}$, where $n$ is the number of vertices in the 2-sphere $S_{2}(x)$. Observe that the merging operation keeps the matrices $\Delta^{\sigma}(x)$ and $\Gamma^{\sigma}(x)$ unchanged.

Notice that the new matrix $\Gamma_2^{\sigma'}(x)$ has a smaller size than the matrix $\Gamma_2^{\sigma}(x)$. Indeed, we have by \eqref{eq:gamma 2 matrix}-\eqref{eq:Gamma_2_last},
$$\Gamma^{\sigma'}_{2}(x)=(C')\Gamma^{\sigma}_{2}(x)(C')^{\top},$$
 where
$$
C'=\begin{pmatrix}
    I_{d(|B_{2}(x)|-2)}           & \mathbf{0}_{d(|B_2(x)|-2)\times d} & \mathbf{0}_{d(|B_2(x)|-2)\times d} \\
    \mathbf{0}_{d\times d(|B_2(x)|-2)} & I_d                       & I_d
\end{pmatrix}.
$$
is a matrix of the size $d(|B_{2}(x)|-1)\times d|B_{2}(x)|$.

It is straightforward to verify that
$$(C')\Delta^{\sigma}(x)\overline{\Delta^{\sigma}(x)}^\top(C')^{\top}=\Delta^{\sigma'}(x)\overline{\Delta^{\sigma'}(x)}^\top,$$
and
$$(C')\Gamma^{\sigma}(x)(C')^{\top}=\Gamma^{\sigma'}(x),$$
where we extend the sizes of $\Delta^{\sigma}(x)\overline{\Delta^{\sigma}(x)}^\top$ and $\Gamma^\sigma(x)$ to be $d|B_2(x)|\times d|B_2(x)|$ by $zeros$.

Therefore, we have for any $N\in(0,\infty]$,
\begin{align*}
&\Gamma^{\sigma'}_{2}(x)-\frac{1}{N}\Delta^{\sigma'}(x)\overline{\Delta^{\sigma'}(x)}^\top-\K_{G,\sigma,x}(N)\Gamma^{\sigma'}(x)\\
=&C'\left(\Gamma^{\sigma}_{2}(x)-\frac{1}{N}\Delta^{\sigma}(x)\overline{\Delta^{\sigma}(x)}^\top-\K_{G,\sigma,x}(N)\Gamma^{\sigma}(x)\right)(C')^{\top}\geq 0.
\end{align*}
That is, we have $\K_{G',\sigma',x}(N)\geq \K_{G,\sigma,x}(N)$.
\end{proof}
\begin{remark}
  Notice that the operation of merging two vertices in $S_2(x)$ which have no common neighbors produces a new $4$-cycle. In contrast to the $3$-cycle case in Theorem \ref{thm:sphericaloperation}, the curvature never decreases no matter this new $4$-cycle is balanced or not.
\end{remark}

\section{An infinite connection graph with positive Bakry--\'Emery curvature lower bound}\label{section:counterexample}
In closing, we provide a striking example that contrasts sharply with the setting of balanced connections. It is known that for a locally finite graph with bounded vertex degrees, a positive lower bound on curvature forces the graph to be finite \cite{LMP18}. However, the example below demonstrates the existence of an infinite, 4-regular signed graph—unbalanced—whose curvature remains a positive constant.

\begin{example}
Consider the infinite signed graph in Figure \ref{fig:infinitepositive}.
\begin{figure}[!htp]
\centering
\tikzset{vertex/.style={circle, draw, fill=black!20, inner sep=0pt, minimum width=4pt}}
\begin{tikzpicture}[scale=3.0]

\draw (-1.5,0) -- (-1,0) node[midway, below, black]{$-$}
		 -- (-0.5,0) node[midway, below, black]{$-$}
         -- (0,0) node[midway, below, black]{$-$}
         -- (0.5,0) node[midway, below, black]{$-$}
         -- (1,0) node[midway, below,black]{$-$};
\draw (-1.75,0.5) -- (-1.25,0.5) node[midway, above, black]{$-$}
                    -- (-0.75,0.5) node[midway, above, black]{$-$}
                    -- (-0.25,0.5) node[midway, above, black]{$-$}
                    -- (0.25,0.5) node[midway, above, black]{$-$}
                    -- (0.75,0.5) node[midway, above, black]{$-$}
                    -- (1.25,0.5) node[midway, above, black]{$-$};
\draw (-1.25,0.5) -- (-1,0) node[midway, above, black]{}
                    -- (-0.75,0.5) node[midway, above, black]{}
                    -- (-0.5,0) node[midway, above, black]{}
                    -- (-0.25,0.5) node[midway, above, black]{}
                    -- (0,0) node[midway, above, black]{}
                    -- (0.25,0.5) node[midway, above, black]{}
                    -- (0.5,0) node[midway, above, black]{}
                    -- (0.75,0.5) node[midway, above, black]{}
                    -- (1,0) node[midway, below, black]{}
                    -- (1.5,0) node[midway, below, black]{$-$};

\node at (0,0) [vertex, label={[label distance=0mm]270: \small $1$}, fill=black] {};
\node at (-0.5,0) [vertex, label={[label distance=0mm]270: \small $2$}, fill=black] {};
\node at (-0.25,0.5) [vertex, label={[label distance=0mm]90: \small $3$}, fill=black] {};
\node at (0.25,0.5) [vertex, label={[label distance=0mm]90: \small $4$}, fill=black] {};
\node at (0.5,0) [vertex, label={[label distance=0mm]270: \small $5$}, fill=black] {};
\node at (-1,0) [vertex, label={[label distance=0mm]270: \small $6$}, fill=black] {};
\node at (-0.75,0.5) [vertex, label={[label distance=0mm]90: \small $7$}, fill=black] {};
\node at (0.75,0.5) [vertex, label={[label distance=0mm]90: \small $8$}, fill=black] {};
\node at (1,0) [vertex, label={[label distance=0mm]270: \small $9$}, fill=black] {};
\node at (1.5,0) [vertex, label={[label distance=0mm]90: \small }, fill=black] {};
\node at (-1.75,0.5) [vertex, label={[label distance=0mm]90: \small }, fill=black] {};
\node at (-1.25,0.5) [vertex, label={[label distance=0mm]90: \small }, fill=black] {};

\end{tikzpicture}
\caption{The infinite graph with positive curvature.}
\label{fig:infinitepositive}
\end{figure}
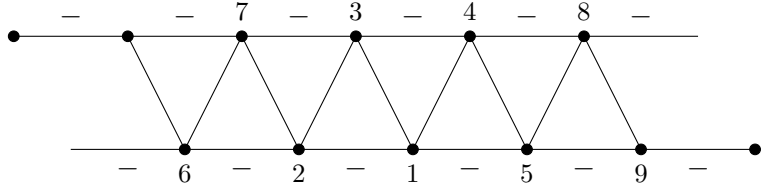
Notice that the local connection structures at every vertex are identical. The local connection structure at $1$ is depicted in Figure \ref{fig:local structure of 1 in infinite graph}.
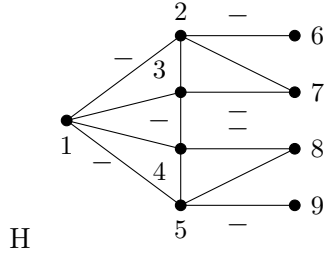
\begin{figure}[!htp]
\centering
\tikzset{vertex/.style={circle, draw, fill=black!20, inner sep=0pt, minimum width=4pt}}
\begin{tikzpicture}[scale=3.0]

\draw (0,0) -- (0.5,0.375) node[midway, above, black]{$-$}
		 -- (1,0.375) node[midway, above, black]{$-$};
\draw (0,0) -- (0.5,0.125) node[midway, below, black]{}
         -- (1,0.125) node[midway, below, black]{$-$}
         -- (0.5,0.375) node[midway, below,black]{};
\draw (0,0) -- (0.5,-0.125) node[midway, above, black]{}
         -- (1,-0.125) node[midway, above, black]{$-$}
         -- (0.5,-0.375) node[midway, above, black]{};
\draw (0,0) -- (0.5,-0.375) node[midway, left, black]{$-$}
         -- (1,-0.375) node[midway, below, black]{$-$};
\draw (0.5,0.375) -- (0.5,0.125) node[midway, below, black]{}
                  -- (0.5,-0.125) node[midway, left, black]{$-$}
                  -- (0.5,-0.375) node[midway, below,black]{};

\node at (0,0) [vertex, label={[label distance=0mm]270: \small $1$}, fill=black] {};
\node at (0.5,0.375) [vertex, label={[label distance=0mm]90: \small $2$}, fill=black] {};
\node at (0.5,0.125) [vertex, label={[label distance=0mm]135: \small $3$}, fill=black] {};
\node at (0.5,-0.125) [vertex, label={[label distance=0mm]225: \small $4$}, fill=black] {};
\node at (0.5,-0.375) [vertex, label={[label distance=0mm]270: \small $5$}, fill=black] {};
\node at (1,0.375) [vertex, label={[label distance=0mm]0: \small $6$}, fill=black] {};
\node at (1,0.125) [vertex, label={[label distance=0mm]0: \small $7$}, fill=black] {};
\node at (1,-0.125) [vertex, label={[label distance=0mm]0: \small $8$}, fill=black] {};
\node at (1,-0.375) [vertex, label={[label distance=0mm]0: \small $9$}, fill=black] {};

\end{tikzpicture}
\caption{The local connection structure of $1$.}
\label{fig:local structure of 1 in infinite graph}
\end{figure}

%
%
%

The matrices $4\Gamma^{\sigma}_{2}(1)$ and $4\Gamma^{\sigma}(1)$ are as follows:
$$4\Gamma^{\sigma}(1)=
\begin{pmatrix}
    8  & 2 & -2 & -2 & 2 \\
    2  & 2 & 0  & 0  & 0 \\
    -2 & 0 & 2  & 0  & 0 \\
    -2 & 0 & 0  & 2  & 0 \\
    2  & 0 & 0  & 0  & 2
\end{pmatrix};
$$
$$4\Gamma^{\sigma}_{2}(1)=
\begin{pmatrix}
    28  & 11 & -12 & -12 & 11 & 1 & -2 & -2 & 1 \\
    11  & 11 & -6  & -2  & 2  & 2 & -2 & 0  & 0 \\
    -12 & -6 & 12  & 6   & -2 & 0 & 2  & 0  & 0 \\
    -12 & -2 & 6   & 12  & -6 & 0 & 0  & 2  & 0 \\
    11  & 2  & -2  & -6  & 11 & 0 & 0  & -2 & 2 \\
    1   & 2  & 0   & 0   & 0  & 1 & 0  & 0  & 0 \\
    -2  & -2 & 2   & 0   & 0  & 0 & 2  & 0  & 0 \\
    -2  & 0  & 0   & 2   & -2 & 0 & 0  & 2  & 0 \\
    1   & 0  & 0   & 0   & 2  & 0 & 0  & 0  & 1
\end{pmatrix}.
$$

Using the canonical choice
\[B_0=\begin{pmatrix}
1 & -1 & 1 &1 &-1\\
0 &1 &0&0&0\\
0&0&1&0&0\\
0&0&0&1&0\\
0&0&0&0&1
\end{pmatrix},\]
we obtain the curvature matrix
\[A_{\infty}(1,B_0)=\frac{1}{4}\begin{pmatrix}
7 & -2 & 2 & 1\\
-2 & 8 & 0 &2 \\
2 & 0 &8 &-2 \\
1 &2 & -2 & 7
\end{pmatrix}.\]
Therefore, the curvature \[\mathcal{K}_{1}(\infty)=\lambda_{\min} (A_{\infty}(1,B_0))=\frac{7-\sqrt{17}}{4}>0.\]
\end{example}

\section*{Acknowledgement}
We are very grateful to the anonymous referee for suggestions that greatly improved the readability of the paper.
This work is supported by the National Key R and D Program of China 2020YFA0713100, the National Natural Science Foundation of China (No. 12031017), and Innovation Program for Quantum Science and Technology 2021ZD0302902.
\appendix
\section{Explicit calculations of the matrices \texorpdfstring{$\Gamma_2^{\sigma(x)}$}{Gamma2(sigma(x))} and \texorpdfstring{$Q(x)$}{Q(x)}}\label{section:appendix}
This appendix details the computation of the matrices $\Gamma_2^{\sigma}(x)$ and the associated Schur complement
\[
Q(x) := \Gamma_2^\sigma(x) \big/ \Gamma_2^\sigma(x)_{S_2, S_2}
\]
for a given vertex $x$ in the connection graph $(G, \sigma)$. Note that for any oriented edge from $x$ to $y$ with label $\sigma_{xy}$, the inverse is given by
\[
\sigma^{-1}_{xy} = \overline{\sigma_{xy}}^\top.
\]
 Recall that $p_{xy}\neq 0$ if and only if $x\sim y$. For simplicity, we will use the following notation. 
 \begin{equation*}
 d_{\mu}(x):=\frac{d_x}{\mu_x},\ \ p_{uv}^{(2)}:=\sum_{w\in V}p_{uw}p_{wv}, \ \ \text{for}\ u,v\in V,
 \end{equation*}
 and 
 \begin{equation*}
p_{xy}^\sigma:=p_{xy}\overline{\sigma_{xy}}, \ \ \text{for}\ x\sim y,  \ \ p_{uv}^{(2),\sigma}:=\sum_{w\in V}p_{uw}p_{wv}\overline{\sigma_{uw}\sigma_{wv}}, \ \ \text{for}\ u,v\in V.
 \end{equation*} 

We first present the explicit expression of the matrix $\Gamma_2^\sigma(x)$ given in (\ref{eq:Gamma2}).
We denote by
\[S_1(x)=\{y_1,\ldots,y_m\},\,\,\text{and}\,\,S_2(x)=\{z_1,\ldots,z_n\}.\]
We use the indices $i,j=1,\ldots, m$ and $k,\ell=1,\ldots,n$. 
The matrix $\Gamma_2^\sigma(x)$ is an Hermitian matrix such that
\begin{align}\label{eq:gamma 2 matrix}
&4\Gamma_2^\sigma(x)_{x,x}=\left(3p_{xx}^{(2)}+d_{\mu}(x)^{2}\right)I_{d},\\
&4\Gamma_2^\sigma(x)_{x,y_i}=p_{xy_i}^{(2),\sigma}-\left(2p_{y_ix}+d_{\mu}(y_i)+d_{\mu}(x)\right)p_{xy_i}^{\sigma},\,\forall \ i,\\
&4\Gamma_2^\sigma(x)_{x,z_k}=p_{xz_k}^{(2),\sigma},\,\forall\  k,\\
&4\Gamma_2^\sigma(x)_{y_i,y_i}=p_{xy_i}^{(2)}I_{d}+\left(2p_{xy_i}+3d_{\mu}(y_i)-d_{\mu}(x)\right)p_{xy_i}I_d,\,\forall\ i,\\
&4\Gamma_2^\sigma(x)_{y_{i},y_{j}}=-2(p_{xy_{i}}p_{y_iy_{j}}+p_{xy_j}p_{y_jy_i})\overline{\sigma_{y_{i}y_{j}}}+2\overline{p_{xy_i}^{\sigma}}^\top p_{xy_j}^{\sigma},\,\forall \ i, \,\forall\ j\neq i,\\
&4\Gamma_2^\sigma(x)_{y_i,z_k}=-2p_{xy_i}p_{y_iz_k}\overline{\sigma_{y_iz_k}},\,\forall\,i,k,\\
&4\Gamma_2^\sigma(x)_{z_{k},z_{k}}=p_{xz_k}^{(2)}I_{d},\,\forall\ k,\\
&4\Gamma_2^\sigma(x)_{z_{k},z_{\ell}}=\mathbf{0}_d,\,\forall\ k, \,\forall \ \ell\neq k.\label{eq:Gamma_2_last}
\end{align}
In particular, the matrix block
\begin{equation}\label{eq:Gamma2S2S2}
4\Gamma^\sigma_2(x)_{S_2,S_2}=\begin{pmatrix}
 p_{xz_1}^{(2)}I_{d} & & \\
&\ddots & \\
& &  p_{xz_n}^{(2)}I_{d}
\end{pmatrix}
\end{equation}
is real, diagonal, and invertible.

Next, we present the explicit expression for the associated Schur complement
\begin{align*}Q(x):=\Gamma_2^\sigma(x)/\Gamma_2^\sigma(x)_{S_2,S_2}=\Gamma_2^\sigma(x)_{B_1,B_1}-\Gamma_2^\sigma(x)_{B_1,S_2}\Gamma_2^\sigma(x)_{S_2,S_2}^{-1}\Gamma_2^\sigma(x)_{S_2,B_1}.
\end{align*}
The matrix $Q(x)$ is an Hermitian matrix such that
\begin{align}\label{eq:schur complement of gamma 2 matrix}
&4Q(x)_{x,x}=3p_{xx}^{(2)}I_{d}+d_{\mu}(x)^{2}I_{d}-\sum_{k=1}^n\frac{p_{xz_k}^{(2),\sigma}\overline{p_{xz_k}^{(2),\sigma}}^\top}{p_{xz_{k}}^{(2)}},\\
&4Q(x)_{x,y_i}=p_{xy_i}^{(2),\sigma}-\left(2p_{y_ix}+d_{\mu}(y_i)+d_{\mu}(x)\right)p_{xy_i}^{\sigma}+2\sum_{k=1}^n\frac{p_{xz_k}^{(2),\sigma}p_{xy_{i}}p_{y_{i}z_{k}}\sigma_{y_{i}z_{k}}^\top}{p_{xz_k}^{(2)}},\,\forall\ i,\\
&4Q(x)_{y_{i},y_{i}}=p_{xy_i}^{(2)}I_{d}+\left(2p_{xy_i}+3d_{\mu}(y_i)-d_{\mu}(x)\right)p_{xy_i}I_d-4\sum_{k=1}^n\frac{p_{xy_{i}}^2p_{y_{i}z_{k}}^2}{p_{xz_k}^{(2)}}I_{d},\,\forall \ i,\\
&4Q(x)_{y_{i},y_j}=-2(p_{xy_{i}}p_{y_iy_{j}}+p_{xy_j}p_{y_jy_i})\overline{\sigma_{y_{i}y_{j}}}+2\overline{p_{xy_{i}}^{\sigma}}^\top p_{xy_{j}}^{\sigma}\notag\\
&\hspace{6.7cm}-4\sum_{k=1}^n\frac{p_{xy_{i}}p_{xy_j}p_{y_{i}z_{k}}^{\sigma}\overline{p_{y_jz_{k}}^{\sigma}}^\top}{p_{xz_k}^{(2)}},\,\forall\ i,\,\forall\ j\neq i.\label{eq:Q13}
\end{align}
Next, we show the following key property.
\begin{proposition}\label{prop:aadaggeromega}
Let $x$ be a given vertex in a connection graph $(G,\sigma)$. For any nonsingular matrix $B$ satisfying (\ref{eq:B}), let $a:=a(G,\sigma,x,B)$ and $\omega^\top:=\omega^\top(G,\sigma,x,B)$ be given in (\ref{eq:a}) and (\ref{eq:omega}), respectively. Then we have
\begin{equation*}
 a\succeq 0, \,\,\text{and}\,\,\, aa^\dagger\omega^\top=\omega^\top.
\end{equation*}
Moreover, we have $a=\mathbf{0}_d, \omega=\mathbf{0}_{md \times d}$ when the local connection structure $B_2^{inc}(x)$ is balanced.
\end{proposition}

We first prepare the following key lemma for the proof of Proposition \ref{prop:aadaggeromega}
\begin{lemma}\label{lemma:schur condition}
Let $X_{i}, i=1,\ldots,m$ and $Y_{j}, j=1,\ldots,\ell$ be $n\times n$ square matrices, where $m\geq \ell$.
Define $A:=\sum_{i=1}^{m}X_{i}\overline{X_{i}}^\top$ and $B:=\sum_{j=1}^{\ell}X_{j}Y_{j}$.  Then we have $$AA^\dagger B=B.$$
\end{lemma}
\begin{proof}
Notice that $A\geq 0$. Thus there exists a nonsingular matrix $P$ such that
\begin{align}\label{eq:positive semidifinite}
A=P\begin{pmatrix}
    \lambda_{1}  &         &    \\
                 & \ddots  &    \\
                 &         & \lambda_{n}
\end{pmatrix}\overline{P}^\top,
\,\,\text{ with } \,\,\lambda_{k}\geq 0 ,\,\,k=1,\dots,n.
\end{align}
Denote for any real number $\lambda$ that
$$\lambda^\dagger=\left\{
\begin{aligned}
&\frac{1}{\lambda},&\textrm{if} \,\, \lambda\neq 0;\\
&0,&\textrm{otherwise.}
\end{aligned}
\right.
$$
Then we calculate
\begin{align*}
AA^{\dagger}B&=
=P\sum_{j=1}^{\ell}\begin{pmatrix}
   \lambda_{1}\lambda^{\dagger}_{1}  &         &      \\
                                     & \ddots  &      \\
                                     &         & \lambda_{n}\lambda^{\dagger}_{n}
\end{pmatrix}P^{-1}X_{j}Y_{j}.
\end{align*}

We claim that
\begin{equation}\label{eq:PinverseX}
\begin{pmatrix}
    \lambda_{1}\lambda^{\dagger}_{1}  &         &      \\
                                      & \ddots  &      \\
                                      &         & \lambda_{n}\lambda^{\dagger}_{n}
\end{pmatrix}P^{-1}X_{i}=P^{-1}X_{i},\,\,\text{for each}\ i=1,\dots,m.
\end{equation}

To show this claim, we first deduce from (\ref{eq:positive semidifinite}) that
$$\begin{pmatrix}
    \lambda_{1}  &         &       \\
                 & \ddots  &       \\
                 &         & \lambda_{n}
\end{pmatrix}=P^{-1}A\left(\overline{P}^\top\right)^{-1}=\sum_{i=1}^{m}(P^{-1}X_{i})\overline{(P^{-1}X_{i})}^\top.$$
Denote by
$[P^{-1}X_{i}]_{k}$ \label{notation:matrix_rows} the $k$-th row of $P^{-1}X_i$. Then, we reformulate the above identity as
 \[\sum_{i=1}^{m}[P^{-1}X_{i}]_{k}\overline{[P^{-1}X_{i}]_{k}}^\top=\lambda_{k}, \ k=1,\dots,n.\] 
 Observe that $\lambda_{k}=0$ if and only if 
$[P^{-1}X_{i}]_{k}=0$ for every $i=1,\dots,m$.
Therefore, we obtain
$$\lambda_{k}\lambda^{\dagger}_{k}[P^{-1}X_{i}]_{k}=[P^{-1}X_{i}]_{k},\ i=1,\ldots,n, \ k=1,\ldots,n.$$
That is, (\ref{eq:PinverseX}) holds true.

Applying (\ref{eq:PinverseX}), we continue our calculation as
$$AA^{\dagger}B
=P\sum_{j=1}^{\ell}(P^{-1}X_{j})Y_{j}=\sum_{j=1}^{\ell}X_{j}Y_{j}=B.$$
This finishes the proof.
\end{proof}
\begin{proof}[Proof of Proposition \ref{prop:aadaggeromega}]
Consider the nonsingular matrix
\[B_{0}=\begin{pmatrix}
p_0^\top \\
p_1^\top \\
\vdots\\
p_m^\top
\end{pmatrix}\]
satisfying (\ref{eq:B}) with $p_0^\top=\begin{pmatrix}
  I_d & \overline{\sigma_{xy_1}} & \cdots & \overline{\sigma_{xy_m}}
\end{pmatrix}$ and $p_{i}, i=1,\ldots,m$ given in  (\ref{eq:specialb_i}). We first calculate the $d\times (m+1)d$ matrix $p_{0}^\top4Q(x)$ by (\ref{eq:schur complement of gamma 2 matrix})-(\ref{eq:Q13}). Denote it as
$$p_{0}^\top4Q(x)=\begin{pmatrix}
  (p_{0}^\top4Q(x))_{0} & (p_{0}^\top4Q(x))_{1} & \cdots & (p_{0}^\top4Q(x))_{m}
\end{pmatrix}.$$
Then the first $d\times d$ block $(p_{0}^\top4Q(x))_{0}$ is 
\begin{align}\label{eq:B_0Q(x)}
&(p_{0}^\top4Q(x))_{0}=4Q(x)_{xx}+\sum_{j=1}^m\overline{\sigma_{xy_j}}4Q(x)_{y_jx}\notag\\=&\sum_{k=1}^{n}\frac{p_{xz_k}^{(2),\sigma}\overline{p_{xz_k}^{(2),\sigma}}^\top}{p_{xz_k}^{(2)}}
-\sum_{k=1}^{n}p_{xz_k}^{(2)}I_{d}+\sum_{i=1}^{m}\sum_{j\neq i}p_{xy_{j}}p_{y_{j}y_{i}}\overline{\sigma_{xy_{i}}}\overline{\sigma_{y_{i}y_{j}}}\overline{\sigma_{y_{j}x}}-\sum_{i=1}^{m}p_{xy_i}^{(2)}I_{d}\notag\\
=&-\sum_{k=1}^n\frac{\left(p_{xz_k}^{(2)}I_d\right)^2-p_{xz_k}^{(2),\sigma}\overline{p_{xz_k}^{(2),\sigma}}^\top}{p_{xz_k}^{(2)}}-\sum_{i=1}^{m}\sum_{j\neq i}p_{xy_{j}}p_{y_{j}y_{i}}(I_{d}-\overline{\sigma_{xy_{i}}}\overline{\sigma_{y_{i}y_{j}}}\overline{\sigma_{y_{j}x}})\notag\\
=&-\sum_{k=1}^n\frac{1}{p_{xz_k}^{(2)}}\sum_{i<j}p_{xy_i}p_{y_iz_k}p_{xy_j}p_{y_jz_k}\left(I_d-\overline{\sigma_{xy_j}}\overline{\sigma_{y_jz_k}}\sigma_{y_iz_k}^\top\sigma_{xy_i}^\top\right)\left(I_d-\overline{\sigma_{xy_i}}\overline{\sigma_{y_iz_k}}\sigma_{y_jz_k}^\top\sigma_{xy_j}^\top\right)\notag\\
&\hspace{0.5cm}-\sum_{i=1}^{m}\sum_{j\neq i}p_{xy_{j}}p_{y_{j}y_{i}}(I_{d}-\overline{\sigma_{xy_{i}}}\overline{\sigma_{y_{i}y_{j}}}\overline{\sigma_{y_{j}x}}).
\end{align}
In the above, we write $\sum_{i<j}:=\sum_{i=1}^m\sum_{j=i+1}^m$ for short. And we have for each $i=1,\ldots,m$,
\begin{align}
&(p_{0}^\top4Q(x))_{i}=4Q(x)_{xy_i}+\sum_{j=1}^m\overline{\sigma_{xy_j}}4Q(x)_{y_jy_i}.\notag\\
=&\sum_{j=1}^mp_{xy_{j}}p_{y_{j}y_{i}}(I_{d}-\overline{\sigma_{xy_{j}}}\overline{\sigma_{y_{j}y_{i}}}\overline{\sigma_{y_{i}x}})\overline{\sigma_{xy_{i}}}+2\sum_{j=1}^mp_{xy_{i}}p_{y_{i}y_{j}}(I_{d}-\overline{\sigma_{xy_{j}}}\overline{\sigma_{y_{j}y_{i}}}\overline{\sigma_{y_{i}x}})\overline{\sigma_{xy_{i}}}\notag\\
&\hspace{0.5cm}+2\sum_{k=1}^{n}p_{xy_{i}}p_{y_{i}z_{k}}\left(I_{d}-\frac{p_{xz_k}^{(2),\sigma}\sigma_{y_{i}z_{k}}^\top\sigma_{xy_{i}}^\top}{p_{xz_k}^{(2)}}\right)\overline{\sigma_{xy_{i}}}.\notag
\end{align}
Observe that 
\begin{align*}
    I_{d}-\frac{p_{xz_k}^{(2),\sigma}\sigma_{y_{i}z_{k}}^\top\sigma_{xy_{i}}^\top}{p_{xz_k}^{(2)}}
=&\sum_{j=1}^m\frac{p_{xy_j}p_{y_jz_k}}{p_{xz_k}^{(2)}}\left(I_{d}-\overline{\sigma_{xy_{j}}}\overline{\sigma_{y_{j}z_{k}}}\sigma_{y_{i}z_{k}}^\top\sigma_{xy_{i}}^\top\right).
\end{align*}
Therefore, we have
\begin{align}\label{eq:B_0Q(x)i}
(p_{0}^\top4Q(x))_{i}=&\sum_{j\neq i}p_{xy_{j}}p_{y_{j}y_{i}}(I_{d}-\overline{\sigma_{xy_{j}}}\overline{\sigma_{y_{j}y_{i}}}\overline{\sigma_{y_{i}x}})\overline{\sigma_{xy_{i}}}+2\sum_{j\neq i}p_{xy_{i}}p_{y_{i}y_{j}}(I_{d}-\overline{\sigma_{xy_{j}}}\overline{\sigma_{y_{j}y_{i}}}\overline{\sigma_{y_{i}x}})\overline{\sigma_{xy_{i}}}\notag\\
&+2\sum_{k=1}^{n}p_{xy_{i}}p_{y_{i}z_{k}}\left(\sum_{j\neq i}\frac{p_{xy_{j}}p_{y_{j}z_{k}}}{p_{xz_k}^{(2)}}\left(I_{d}-\overline{\sigma_{xy_{j}}}\overline{\sigma_{y_{j}z_{k}}}\sigma_{y_{i}z_{k}}^\top\sigma_{xy_{i}}^\top\right)\right)\overline{\sigma_{xy_{i}}}.
\end{align}
Next, we calculate by applying (\ref{eq:B_0Q(x)}) and (\ref{eq:B_0Q(x)i}) that
\begin{align}
&2a(B_{0}):=2a(G,\sigma,x,B_{0})=p_0^\top 4Q(x)\overline{p_0}\notag\\
=&\sum_{k=1}^n\frac{1}{p_{xz_k}^{(2)}}\sum_{i<j}p_{xy_i}p_{y_iz_k}p_{xy_j}p_{y_jz_k}\left(I_d-\overline{\sigma_{xy_j}}\overline{\sigma_{y_jz_k}}\sigma_{y_iz_k}^\top\sigma_{xy_i}^\top\right)\left(I_d-\overline{\sigma_{xy_i}}\overline{\sigma_{y_iz_k}}\sigma_{y_jz_k}^\top\sigma_{xy_j}^\top\right)\notag\\
&+\sum_{i=1}^{m}\sum_{j\neq i}p_{xy_{j}}p_{y_{j}y_{i}}\left(I_{d}-\overline{\sigma_{xy_{j}}}\overline{\sigma_{y_{j}y_{i}}}\overline{\sigma_{y_{i}x}}\right)\left(I_d-\overline{\sigma_{xy_{i}}}\overline{\sigma_{y_{i}y_{j}}}\overline{\sigma_{y_{j}x}}\right).\label{eq:appendixa}
\end{align}
Let us denote
\begin{align*}
\omega^\top(B_{0}):=&\omega^\top(G,\sigma,x,B_{0})
=\begin{pmatrix}
     \omega^{\top}_{1} & \omega^{\top}_{2} & \cdots & \omega^{\top}_{m}
\end{pmatrix}.
\end{align*}
We further calculate for each $i=1,\ldots,m$ that
\begin{align*}
2\omega^{\top}_{i}(B_0)&=p_{0}^\top4Q(x)\overline{p_{i}}=\sum_{j\neq i}\left(\frac{1}{\sqrt{p_{xy_{i}}}}p_{xy_{j}}p_{y_{j}y_{i}}+2\sqrt{p_{xy_{i}}}p_{y_{i}y_{j}}\right)\left(I_{d}-\overline{\sigma_{xy_{j}}}\overline{\sigma_{y_{j}y_{i}}}\overline{\sigma_{y_{i}x}}\right)\overline{\sigma_{xy_{i}}}\\
&\hspace{2.5cm}+2\sum_{k=1}^{n}\sqrt{p_{xy_{i}}}p_{y_{i}z_{k}}\sum_{j\neq i}\frac{p_{xy_{j}}p_{y_{j}z_{k}}}{p_{xz_k}^{(2)}}\left(I_{d}-\overline{\sigma_{xy_{j}}}\overline{\sigma_{y_{j}z_{k}}}\sigma_{y_{i}z_{k}}^\top\sigma_{xy_{i}}^\top\right)\overline{\sigma_{xy_{i}}}.
\end{align*}
Notice that the matrices $I_{d}-\overline{\sigma_{xy_{j}}}\overline{\sigma_{y_{j}y_{i}}}\overline{\sigma_{y_{i}x}}$ and $I_{d}-\overline{\sigma_{xy_{i}}}\overline{\sigma_{y_{i}y_{j}}}\overline{\sigma_{y_{j}x}}$ are conjugate to each other. The same relation holds between $I_{d}-\overline{\sigma_{xy_{j}}}\overline{\sigma_{y_{j}z_{k}}}\sigma_{y_{i}z_{k}}^\top\sigma_{xy_{i}}^\top$ and $I_{d}-\overline{\sigma_{xy_{i}}}\overline{\sigma_{y_{i}z_{k}}}\sigma_{y_{j}z_{k}}^\top\sigma_{xy_{j}}^\top$. Hence, we can reformulate  $2a(B_{0})$ as follows,
\begin{equation}\label{eq:Appendix_a}
2a(B_{0})=\sum_{i,j,k}X_{ij,k}\overline{X_{ij,k}}^\top+\sum_{i,j}X_{ij}\overline{X_{ij}}^\top,
\end{equation}
where we have for each $i,j=1,\ldots,m$ and each $k=1,\ldots,n$ that
$$X_{ij,k}=\sqrt{\frac{p_{xy_{i}}p_{y_{i}z_{k}}p_{xy_{j}}p_{y_{j}z_{k}}}{p_{xz_k}^{(2)}}}(I_{d}-\overline{\sigma_{xy_{j}}}\overline{\sigma_{y_{j}z_{k}}}\sigma_{y_{i}z_{k}}^\top\sigma_{xy_{i}}^\top), X_{ij}=\sqrt{p_{xy_{j}}p_{y_{j}y_{i}}}(I_{d}-\overline{\sigma_{xy_{j}}}\overline{\sigma_{y_{j}y_{i}}}\overline{\sigma_{y_{i}x}}).$$
In particular, we have $a(B_0)\succeq 0$ and $a(B_0)=\mathbf{0}_d$ if the local connection structure $B_2^{inc}(x)$ is balanced.
The blocks $2\omega^\top_{i}(B_{0})$ can be formulated into the following form
\begin{equation}\label{eq:Appendix_omega}
    2\omega^\top_{i}(B_{0})=\sum_{j,k}X_{ij,k}Y_{ij,k}+\sum_{j}X_{ij}Y_{ij}
\end{equation}
for some suitable choices of matrices $Y_{ij,k}$ and $Y_{ij}$, $i,j=1,\dots,m, k=1,\dots,n$. We observe particularly that $\omega(B_0)=\mathbf{0}_{md\times d}$ if the local connection structure $B_2^{inc}(x)$ is balanced.
Applying Lemma \ref{lemma:schur condition} to $2a(B_{0})$ and $2\omega^\top_{i}(B_{0})$ yields $$a(B_{0})a(B_{0})^\dagger2\omega^\top_{i}(B_{0})=2\omega^\top_{i}(B_{0}),\,\,\text{for every}\  i=1,\dots,m.$$
Therefore, we have 
$a(B_{0})a(B_{0})^\dagger\omega^\top(B_{0})=\omega^\top(B_{0}).$

To finish the proof, let us consider a general nonsingular matrix $B$ satisfying (\ref{eq:B}) with the form (\ref{eq:generalB}). We divide each $b_{i}^\top$, $i=1,\ldots, m$, into $(m+1)$ blocks of size $d\times d$ as follows
$$b_{i}^\top=\begin{pmatrix}
  b_{i,0}^\top & b_{i,1}^\top & \cdots & b_{i,m}^\top
\end{pmatrix}.$$
Then we have \begin{align*}
2a(B):=2a(G,\sigma,x,B)=&b_{0}^\top4Q(x)\overline{b_{0}}=E_{d}p_{0}^\top4Q(x)\overline{p_{0}}\overline{E_{d}}^\top=E_{d}2a(B_0)\overline{E_{d}}^\top\\
=&\sum_{i,j,k}(E_{d}X_{ij,k})\overline{(E_{d}X_{ij,k})}^\top+\sum_{i,j}(E_{d}X_{ij})\overline{(E_{d}X_{ij,k})}^\top.
\end{align*}
In particular, we observe that $a(B)\succeq 0$ and $a(B)=\mathbf{0}_d$ if and only if $a(B_0)=\mathbf{0}_d$.
Moreover, we observe for each $i=1,\ldots,m$ that
\begin{align}\label{eq:omega_i}
2\omega^\top_{i}(B)=b_{0}^\top4Q(x)\overline{b_{i}}=&(b_{0}^\top4Q(x))_0\overline{b_{i,0}}+\sum_{j=1}^m(b_{0}^\top4Q(x))_j\overline{b_{i,j}}\notag\\
=& (E_dp_{0}^\top4Q(x))_0\overline{b_{i,0}}+\sum_{j=1}^m(E_dp_{0}^\top4Q(x))_j\overline{b_{i,j}}.
\end{align}
Due to \eqref{eq:Appendix_omega}, we have for each $i=1,\ldots,m$
$$2\omega^\top_{i}(B)=\sum_{j,k}(E_{d}X_{ij,k})\tilde{Y}_{ij,k}+\sum_{j}(E_{d}X_{ij})\tilde{Y}_{ij}$$
 for suitable $\tilde{Y}_{ij,k}s$ and $\tilde{Y}_{ij}s$.
 Therefore, we can apply Lemma \ref{lemma:schur condition} to derive
$$a(B)a(B)^\dagger\omega^\top(B)=\omega^\top(B).$$
We also mention that $\omega(B)=\mathbf{0}_{md\times d}$ if the local connection structure $B_2^{inc}(x)$ is balanced. This completes the proof.
\end{proof}

To conclude, we consider the Cartesian product of two connection graphs $(G,\sigma)$ and $(G',\sigma')$ and provide a proof for  Lemma \ref{lemma:Qdecomp}. For vertices $x\in V$ and $x'\in V'$, we use the notation
$$S_{1}(x)=\{y_{1},y_{2},\dots,y_{m}\},\,\, S_{2}(x)=\{z_{1},z_{2},\dots,z_{n}\}$$and
$$ S_{1}(x')=\{y'_{1},y'_{2},\dots,y'_{m'}\},\,\, S_{2}(x')=\{z'_{1},z'_{2},\dots,z'_{n'}\}.$$
Next, we give the proof of Lemma \ref{lemma:Qdecomp}.
\begin{proof}[Proof of Lemma \ref{lemma:Qdecomp}]
Consider two vertices $x$ and $x'$ in the connection graphs $(G,\sigma)$ and $(G',\sigma')$, respectively. Keep in mind that we are working under the following assumption:
\begin{equation}\label{eq:localId2}
    \sigma_{xy}=I_d, \,\,\text{and}\,\,\sigma'_{x'y'}=I_d,\,\,\text{for any}\,\,y\in S_1(x), y'\in S_1(x').
\end{equation}
Applying (\ref{eq:gamma 2 matrix})-(\ref{eq:Gamma_2_last}), we derive the Hermitian matrix $4\Gamma^{\sigma^\times}_{2}(x,x')$ as follows.
%
The $d\times d$ blocks in the rows of $4\Gamma^{\sigma^\times}_{2}(x,x')$ corresponding to the vertex $(x,x')$ are
\begin{align*}
4\Gamma^{\sigma^\times}_{2}(x,x')_{(x,x')(x,x')}=\alpha^{2}4\Gamma^{\sigma}_{2}(x)_{x,x}+\beta^{2}4\Gamma^{\sigma'}_{2}(x')_{x',x'}+2\alpha\beta\sum_{i=1}^{m}\sum_{j=1}^{m'}p_{xy_{i}}p_{x'y'_{j}}I_{d},
\end{align*}
\begin{align*}
&4\Gamma^{\sigma^\times}_{2}(x,x')_{(x,x')(y_{i},x')}=\alpha^{2}4\Gamma^{\sigma}_{2}(x)_{x,y_{i}}-2\alpha\beta p_{xy_{i}}\sum_{j=1}^{m'}p_{x'y'_{j}}I_{d},\,\,\forall\,\,i=1,\dots,m,\\
&4\Gamma^{\sigma^\times}_{2}(x,x')_{(x,x')(x,y'_{i})}=\beta^{2}4\Gamma^{\sigma'}_{2}(x')_{x',y'_{i}}-2\alpha\beta p_{x'y'_{i}}\sum_{j=1}^{m}p_{xy_{j}}I_{d},\,\,\forall\,\,i=1,\dots,m',\\
&4\Gamma^{\sigma^\times}_{2}(x,x')_{(x,x')(z_{k},x')}=\alpha^{2}4\Gamma^{\sigma}_{2}(x)_{x,z_{k}},\,\,\forall\,\,k=1,\dots,n,\\
&4\Gamma^{\sigma^\times}_{2}(x,x')_{(x,x')(y_i,y'_{j})}=2\alpha\beta p_{xy_{i}}p_{x'y'_{j}}I_{d},\,\,\forall\,\,i=1,\dots,m,\,\,\forall\,\,j=1,\dots,m',\\
&4\Gamma^{\sigma^\times}_{2}(x,x')_{(x,x')(x,z'_{k})}=\beta^{2}4\Gamma^{\sigma'}_{2}(x')_{x',z'_{k}},\,\,\forall\,\,k=1,\dots,n'.
\end{align*}
The $d\times d$ blocks in the rows of $4\Gamma^{\sigma^\times}_{2}(x,x')$ corresponding to the vertices $\{(y_{1},x'),\dots,(y_{m},x')\}$ are
\begin{align*}
&4\Gamma^{\sigma^\times}_{2}(x,x')_{(y_{i},x')(y_{i},x')}=\alpha^{2}4\Gamma^{\sigma}_{2}(x)_{y_{i},y_{i}}+2\alpha\beta p_{xy_{i}}\sum_{j=1}^{m'}p_{x'y'_{j}}I_{d},\,\,\forall\,\,i=1,\dots,m,\\
&4\Gamma^{\sigma^\times}_{2}(x,x')_{(y_{i},x')(y_{j},x')}=\alpha^{2}4\Gamma^{\sigma}_{2}(x)_{y_{i},y_{j}},\,\,\forall\,\,i,j=1,\dots,m,\,\,\text{with}\,\,j\neq i,\\
&4\Gamma^{\sigma^\times}_{2}(x,x')_{(y_{i},x')(x,y'_{j})}=2\alpha\beta p_{xy_{i}}p_{x'y'_{j}}I_{d},\,\,\forall\,\,i=1,\dots,m,\,\,\forall\,\,j=1,\dots,m',\\
&4\Gamma^{\sigma^\times}_{2}(x,x')_{(y_{i},x')(z_{k},x')}=\alpha^{2}4\Gamma^{\sigma}_{2}(x)_{y_{i},z_{k}},\,\,\forall\,\,i=1,\dots,m,\,\,\forall\,\,k=1,\dots,n,\\
&4\Gamma^{\sigma^\times}_{2}(x,x')_{(y_{i},x')(y_{i},y'_{j})}=-2p_{xy_{i}}p_{xy'_{j}}I_{d},\,\,\forall\,\,i=1,\dots,m,\,\,\forall\,\,j=1,\dots,m',\\
&4\Gamma^{\sigma^\times}_{2}(x,x')_{(y_{i},x')(y_{r},y'_{j})}=\mathbf{0}_d,\,\,\forall\,\,i,r=1,\dots,m,\,\,\text{with}\,\,i\neq r,\,\,\forall\,\,j=1,\dots,m',\\
&4\Gamma^{\sigma^\times}_{2}(x,x')_{(y_{i},x')(x,z'_{k})}=\mathbf{0}_d,\,\,\forall\,\,i=1,\dots,m,\,\,\forall\,\,k=1,\dots,n'.
\end{align*}
The $d\times d$ blocks in the rows of $4\Gamma^{\sigma^\times}_{2}(x,x')$ corresponding to the vertices $\{(x,y'_{1}),\dots,(x,y'_{m'})\}$ are
\begin{align*}
&4\Gamma^{\sigma^\times}_{2}(x,x')_{(x,y'_{i})(x,y'_{i})}=\beta^{2}4\Gamma^{\sigma'}_{2}(x')_{y'_{i},y'_{i}}+2\alpha\beta p_{x'y'_{i}}\sum_{j=1}^{m}p_{xy_{j}}I_{d},\,\,\forall\,\,i=1,\dots,m',\\
&4\Gamma^{\sigma^\times}_{2}(x,x')_{(x,y'_{i})(x,y'_{j})}=\beta^{2}4\Gamma^{\sigma'}_{2}(x')_{y'_{i},y'_{j}},\,\,\forall\,\,i,j=1,\dots,m',\,\,\text{with}\,\,i\neq j,\\
&4\Gamma^{\sigma^\times}_{2}(x,x')_{(x,y'_{i})(z_{k},x')}=\mathbf{0}_d,\,\,\forall\,\,i=1,\dots,m',\,\,\forall\,\,k=1,\dots,n,\\
&4\Gamma^{\sigma^\times}_{2}(x,x')_{(x,y'_{i})(y_{j},y'_{i})}=-2p_{xy_{j}}p_{x'y'_{i}}I_{d},\,\,\forall\,\,j=1,\dots,m,\,\,\forall\,\,i=1,\dots,m',\\
&4\Gamma^{\sigma^\times}_{2}(x,x')_{(x,y'_{i})(y_{j},y'_{r})}=\mathbf{0}_d,\,\,\forall\,\,j=1,\dots,m,\,\,\forall\,\,i,r=1,\dots,m'\,\,\text{with}\,\,i\neq r,\\
&4\Gamma^{\sigma^\times}_{2}(x,x')_{(x,y'_{i})(x,z'_{k})}=\beta^{2}4\Gamma^{\sigma'}_{2}(x')_{y'_{i},z'_{k}},\,\,\forall\,\,i=1,\dots,m',\,\,\forall\,\,k=1,\dots,n'.
\end{align*}

The $d\times d$ blocks in the rows of $4\Gamma^{\sigma^\times}_{2}(x,x')$ corresponding to the vertices set $\{(z_{1},x'),\dots,(z_{n},x')\}$ are
\begin{align*}
&4\Gamma^{\sigma^\times}_{2}(x,x')_{(z_{k},x')(z_{k},x')}=\alpha^{2}4\Gamma^{\sigma}_{2}(x)_{z_{k},z_{k}},\,\,\forall\,\,k=1,\dots,n,\\
&4\Gamma^{\sigma^\times}_{2}(x,x')_{(z_{k},x')(z_{\ell},x')}=\mathbf{0}_d,\,\,\forall\,\,k,\ell=1,\dots,n\,\,\text{with}\,\,k\neq\ell,\\
&4\Gamma^{\sigma^\times}_{2}(x,x')_{(z_{k},x')(y_{i},y'_{j})}=\mathbf{0}_d,\,\,\forall\,\,k=1.\dots,n,\,\,\forall\,\,i=1,\dots,m,\,\,\forall\,\,j=1,\dots,m',\\
&4\Gamma^{\sigma^\times}_{2}(x,x')_{(z_{k},x')(x,z'_{\ell})}=\mathbf{0}_d,\,\,\forall\,\,k=1,\dots,n,\,\,\forall\,\,\ell=1,\dots,n'.
\end{align*}
The $d\times d$ block in the rows of $4\Gamma^{\sigma^\times}_{2}(x,x')$ corresponding to the vertices $\{(y_{1},y'_{1}),\dots,(y_{m},y'_{m'})\}$ are
\begin{align*}
&4\Gamma^{\sigma^\times}_{2}(x,x')_{(y_{i},y'_{j})(y_{i},y'_{j})}=2\alpha\beta p_{xy_{i}}p_{x'y'_{j}}I_{d},\,\,\forall\,\,i=1,\dots,m,\,\,\forall\,\,j=1,\dots,m',\\
&4\Gamma^{\sigma^\times}_{2}(x,x')_{(y_{i},y'_{j})(y_{r},y'_{s})}=\mathbf{0}_d,\,\,\forall\,\,i,r=1,\dots,m,\,\,\forall\,\,j,s=1,\dots,m',\,\,\text{with}\,\,i\neq r\,\,\text{or}\,\,j\neq s,\\
&4\Gamma^{\sigma^\times}_{2}(x,x')_{(y_{i},y'_{j})(x,z'_{k})}=\mathbf{0}_d,\,\,\forall\,\,i=1,\dots,m,\,\,\forall\,\,j=1,\dots,m',\,\,\forall\,\,k=1,\dots,n'.
\end{align*}
The $d\times d$ blocks in the rows of $4\Gamma^{\sigma^\times}_{2}(x,x')$ corresponding to the vertices $\{(x,z'_{1}),\dots,(x,z'_{n'})\}$ are
\begin{align*}
&4\Gamma^{\sigma^\times}_{2}(x,x')_{(x,z'_{k})(x,z'_{k})}=\beta^{2}4\Gamma^{\sigma'}_{2}(x')_{z'_{k},z'_{k}},\,\,\forall\,\,k=1,\dots,n',\\
&4\Gamma^{\sigma^\times}_{2}(x,x')_{(x,z'_{k})(x,z'_{\ell})}=\mathbf{0}_d,\,\,\forall\,\,k,\ell=1,\dots,n',\,\,\text{with}\,\,k\neq\ell.
\end{align*}


Then we calculate the matrix $L(x,x'):=4\Gamma^{\sigma^\times}_{2}(x,x')_{B_{1},S_{2}}(4\Gamma^{\sigma^\times}_{2}(x,x')_{S_{2},S_{2}})^{-1}4\Gamma^{\sigma^\times}_{2}(x,x')_{S_{2},B_{1}}$.
The $d\times d$ blocks in the rows of $L(x,x')$ corresponding to the vertex $(x,x')$ are
\begin{align*}
&L(x,x')_{(x,x')(x,x')}=\sum_{k=1}^n\frac{p_{xz_k}^{(2),\sigma}\overline{p_{xz_k}^{(2),\sigma}}^\top}{p_{xz_k}^{(2)}}+\sum_{k=1}^{n'}\frac{p_{x'z_k'}^{(2),\sigma}\overline{p_{x'z_k'}^{(2),\sigma}}^\top}{p_{x'z_k'}^{(2)}},\\
&L(x,x')_{(x,x')(y_{i},x')}=-2\alpha\beta p_{xy_{i}}\sum_{j=1}^{m'}p_{x'y'_{j}}I_{d}-2\alpha^{2}\sum_{k=1}^{n}\frac{p_{xy_{i}}p_{y_{i}z_{k}}(\sum_{j=1}^{m}p_{xy_{j}}p_{y_{j}z_{k}}\overline{\sigma_{y_{j}z_{k}}})}{p_{xz_k}^{(2)}},\notag\\
&\hspace{11cm}\,\,\forall\,\,i=1,\dots,m,\\
&L(x,x')_{(x,x')(x,y'_{i})}=-2\alpha\beta p_{x'y'_{i}}\sum_{j=1}^{m}p_{xy_{j}}I_{d}-2\beta^{2}\sum_{k=1}^{n'}\frac{p_{x'y'_{i}}p_{y'_{i}z'_{k}}\sum_{j=1}^{m'}p_{x'y'_{j}}p_{y'_{j}z'_{k}}\overline{\sigma_{y'_{j}z'_{k}}}}{p_{x'z_k'}^{(2)}},\notag\\
&\hspace{11cm}\,\,\forall\,\,i=1,\dots,m'.
\end{align*}
The $d\times d$ blocks in the rows of $L(x,x')$ corresponding to the vertices $\{(y_{1},x'),\dots,(y_{m},x')\}$ are
\begin{align*}
&L(x,x')_{(y_{i},x')(x,x')}=\overline{N(x,x')_{(x,x')(y_{i},x')}}^\top,\,\,\forall\,\,i=1,\dots,m,\\
&L(x,x')_{(y_{i},x')(y_{i},x')}=2\alpha\beta p_{xy_{i}}\sum_{j=1}^{m'}p_{x'y'_{j}}I_{d}+4\alpha^{2}\sum_{k=1}^{n}\frac{(p_{xy_{i}}p_{y_{i}z_{k}})^{2}}{p_{xz_k}^{(2)}},\,\,\forall\,\,i=1,\dots,m,\\
&L(x,x')_{(y_{i},x')(y_{j},x')}=4\alpha^{2}\sum_{k=1}^{n}\frac{p_{xy_{i}}p_{y_{i}z_{k}}p_{xy_{j}}p_{y_{j}z_{k}}\overline{\sigma_{y_{i}z_{k}}}\sigma_{y_{j}z_{k}}^\top}{p_{xz_k}^{(2)}},\,\,\forall\,\,i,j=1\dots,m,\,\,\text{with}\,\,j\neq i,\\
&L(x,x')_{(y_{i},x')(x,y'_{j})}=2\alpha\beta p_{xy_{i}}p_{x'y'_{j}}I_{d},\,\,\forall\,\,i=1,\dots,m,\,\,\forall\,\,j=1,\dots,m'.
\end{align*}
The $d\times d$ blocks in the rows of $L(x,x')$ corresponding to the vertices $\{(x,y'_{1}),\dots,(x,y'_{m'})\}$ are
\begin{align*}
&L(x,x')_{(x,y'_{i})(x,x')}=\overline{N(x,x')_{(x,x')(x,y'_{i})}}^\top,\,\,\forall\,\,i=1,\dots,m',\\
&L(x,x')_{(x,y'_{j})(y_{i},x)}=\overline{N(x,x')_{(y_{i},x')(x,y'_{j})}}^\top,\,\,\forall\,\,i=1,\dots,m,\,\,\forall\,\,j=1,\dots,m',
\end{align*}
\begin{align*}
&L(x,x')_{(x,y'_{i})(x,y'_{i})}=2\alpha\beta p_{x'y'_{i}}\sum_{j=1}^{m}p_{xy_{j}}I_{d}+4\beta^{2}\sum_{k=1}^{n'}\frac{(p_{x'y'_{i}}p_{y'_{i}z'_{k}})^{2}}{p_{x'z_k'}^{(2)}},\,\,\forall\,\,i=1,\dots,m',\\
&L(x,x')_{(x,y'_{i})(x,y'_{j})}=4\beta^{2}\sum_{k=1}^{n'}\frac{p_{x'y'_{i}}p_{y'_{i}z'_{k}}p_{x'y'_{j}}p_{y'_{j}z'_{k}}\overline{\sigma_{y'_{i}z'_{k}}}\sigma_{y'_{j}z'_{k}}^{\top}}{p_{x'z_k'}^{(2)}},\,\,\forall\,\,i,j=1,\dots,m',\,\,\text{with}\,\,j\neq i.
\end{align*}
Finally we derive the matrix $Q(x,x')$:
\begin{align*}
&Q(x,x')_{(x,x')(x,x')}=Q(x)_{x,x}+Q(x')_{x',x'},\\
&Q(x,x')_{(x,x')(y_{i},x')}=Q(x)_{x,y_{i}},\,\,\forall\,\,i=1,\dots,m,\\
&Q(x,x')_{(x,x')(x,y'_{i})}=Q(x')_{x',y'_{i}},\,\,\forall\,\,i=1,\dots,m',\\
&Q(x,x')_{(y_{i},x')(x,x')}=\overline{Q(x,x')_{(x,x')(y_{i},x')}}^\top,\,\,\forall\,\,i=1,\dots,m,\\
&Q(x,x')_{(x,y'_{i})(x,x')}=\overline{Q(x,x')_{(x,x')(x,y'_{i})}}^\top,\,\,\forall\,\,i=1,\dots,m',\\
&Q(x,x')_{(y_{i},x')(y_{i},x')}=Q(x)_{y_{i},y_{i}},\,\,\forall\,\,i=1,\dots,m,\\
&Q(x,x')_{(y_{i},x')(y_{j},x')}=Q(x)_{y_{i},y_{j}},\,\,\forall\,\,i,j=1,\dots,m,\,\,\text{with}\,\,j\neq i,\\
&Q(x,x')_{(x,y'_{i})(x,y'_{i})}=Q(x')_{y'_{i},y'_{i}},\,\,\forall\,\,i=1,\dots,m',\\
&Q(x,x')_{(x,y'_{i})(x,y'_{j})}=Q(x')_{y'_{i},y'_{j}},\,\,\forall\,\,i,j=1,\dots,m',\,\,\text{with}\,\,j\neq i.
\end{align*}
That is, we obtain the identity (\ref{eq:Qdecomp}).
\end{proof}

\section{Glossary of notations}\label{section:notation}
\begin{center}
\begin{longtable}{c|c|c}
 Notations & Definitions & The first-\\
     &     & appearing page\\
 \hline

$\sigma_0\equiv-1$ & The anti-balanced signature & \pageref{sym:antibalanced}\\

 $\Gamma^\sigma_2(x)$ & The matrix corresponding to the & \pageref{notation:gamma_2} \\
  &   sesquilinear form of $\Gamma^\sigma_2(f,f)(x)$. & \\
  
$Q(x)$ & The Schur complement $\Gamma^\sigma_2(x)_{B_1,B_1}/\Gamma^{\sigma}_2(x)_{S_2,S_2}$ & \pageref{notation:schur_complement}\\

 $\sigma_{xy}:=\sigma((x,y))$ & A connection signature on a directed edge $(x,y)$ & \pageref{sym:connection}\\
$B_r(x)$ & A ball center at $x$ with radius $r$. & \pageref{notation:Br_x}\\
$S_r(x)$ & The $r$-sphere of $x$. & \pageref{notation:Sr_x} \\

$p_{xy}$ & Transition rate from $x$ to $y$, $p_{xy}:=\frac{w_{xy}}{\mu(x)}$. & \pageref{notation:transition_rate} \\


 $\sigma^\tau_{xy}:=\tau(x)^{-1}\sigma_{xy}\tau(y)$ & The connection $\sigma$ switched by $\tau$ on the edge $(x,y)$ &  \pageref{notation:switching} \\

$\mathbb{K}=\mathbb{R}$ or $\mathbb{C}$ & The real or complex space & \pageref{notation:real_complex_space} \\

$D(\tau)$ & The diagonal matrix of the function $\tau$ & \pageref{notation:diagonal_tau} \\

$S_1(x)=\{y_1,\cdots,y_m\}$ & The neighbors of $x\in V$ & \pageref{notation:x_S2}\\

$S_2(x)=\{z_1,\cdots,z_n\}$ & The vertices in $2$-sphere of $x$ & \pageref{notation:x_S2}\\
  
 $\Gamma^\sigma_2(x)_{*,\star}$ & The block submatrix with rows and columns & \pageref{eq:Gamma2} \\
  & corresponding to vertex set $*$ and $\star$, respectively & \\
  
   $B_2^{inc}(x)$ & The local structure of incomplete $2$-ball around $x$ & \pageref{eq:Gamma2} \\  

$S_{11}^{\dagger}$ & The pseudoinverse of the matrix $S_{11}$ & \pageref{notation:S_dagger}\\

$S/S_{11}$ & The Schur complement of $S_{11}$ in & \pageref{eq:schur_complement} \\
   &  $S=([S_{11},S_{12}],[S_{21},S_{22}])$ &  \\
   
   $S\succeq 0$ & The symmetric matrix
$S$ is positive semidefinite & \pageref{notation:positive_semidefinite} \\

$S^{1/2}_{11}$ & The square root of the matrix $S_{11}$,& \pageref{notation:S_square_root} \\
  &  i.e. $S^{1/2}_{11}\overline{(S^{1/2}_{11})}^{\top}=S_{11}$ &   \\

  $\mathbf{0}_d$ & The zero matrix of size $d \times d$ & \pageref{notation:zero_matrix} \\

$I_{d}$ & The identity matrix of size $d\times d$ & \pageref{notation_identity_matrix}\\

    $(A)_{\hat 1}$ & The matrix obtained by removing & \pageref{notation:removing_d_rows} \\
      &  the first $d$ rows and the first $d$ columns of $A$ & \\

$A_{\infty}(G,\sigma,x,B)$ & The $md\times md$ $\infty$-dimensional curvature matrix & \pageref{eq:curvature_matrix}\\
  &  at $x$ depending on $B$ & \\

  $A_N(G,\sigma,x,B)$ & The $md \times md $ $N$-dimensional curvature matrix   & \pageref{eq:curvature_matrix_N} \\
    & at $x$ depending on $B$  & \\

     $\K_{G,\sigma,x}(N)$ & $N$-dimensional curvature at $x$ in $(G,\sigma)$ & \pageref{notation:curvature} \\

   $\Re(\cdot)$ & The real part of a complex number & \pageref{notation:real_part} \\

   $\{\epsilon_i\}_{i=1}^d$ & The standard orthonormal basis of $\mathbb{K}^d$ & \pageref{eq:standard_coordinate}\\
      &  $\epsilon_i:=(0,\cdots,1,\cdots, 0)$  &  \\

   $E_{\min}(\cdot)$ & The minimal eigenspace of a matrix & \pageref{notation:eigenspace}\\
 
   $G\times_{\alpha,\beta}G'$ & Cartesian product of $G$ and $G'$ assigned with weights & \pageref{notation:cartesian_product}\\
     &  depending on $\alpha$ and $\beta$ & \\
     
   $\sigma^\times$ & Connection in the Cartesian product of two graphs & \pageref{notation:cartesian_connection}\\

     $f_1*f_2$ & The star product of functions $f_1$ and $f_2$ & \pageref{notation:star_product}\\

   $[A]_i$ & The $i$-row vector of the matrix $A$ & \pageref{notation:matrix_rows}\\

   \end{longtable}
    \label{tab:my_label}
\end{center}


\end{document}